\documentclass[final,10pt]{elsarticle}

\usepackage{color}
\usepackage{verbatim}
\usepackage{algorithmic,algorithm2e,appendix}
\usepackage{tikz}
\usepackage{varwidth}
\usepackage{xr}
\usetikzlibrary{calc,trees,positioning,arrows,chains,shapes.geometric,%
    decorations.pathreplacing,decorations.pathmorphing,shapes,%
    matrix,shapes.symbols}
\usepackage{makecell}
\tikzset{
>=stealth',
  punktchain/.style={
    rectangle,
    rounded corners,
    draw=black, very thin,
    text width=12em,
    minimum height=3em,
    text centered,
    on chain},
  punktchaindash/.style={
    rectangle,
    rounded corners,
    draw=black, dashed, very thin,
    text width=12em,
    minimum height=3em,
    text centered,
    on chain},
  optpunktchain/.style={
    rectangle,
    rounded corners,
    draw=black, ultra thick,
    text width=12em,
    minimum height=3em,
    text centered,
    on chain},
  graychain/.style={
    rectangle,
    rounded corners,
    draw=gray, very thick,
    text width=10em,
    minimum height=3em,
    text centered,
    on chain},
  approxMeth/.style={
    rectangle,
    rounded corners,
    draw=none, thin,
    text width=10em,
    minimum height=3em,
    text centered,
		on chain },
  line/.style={draw, thick, <-},
  element/.style={
    tape,
    top color=white,
    bottom color=blue!50!black!60!,
    minimum width=8em,
    draw=blue!40!black!90, very thick,
    text width=10em,
    minimum height=3.5em,
    text centered,
    on chain},
  every join/.style={->, thick,shorten >=1pt},
  decoration={brace},
  tuborg/.style={decorate},
  tubnode/.style={midway, right=2pt},
}

\usepackage{subfigure}
\usepackage{verbatim}
\usepackage{graphicx}
\usepackage{epstopdf}
\usepackage[crop=pdfcrop]{pstool}
\usepackage{geometry}
\usepackage{graphicx}
\usepackage{color}
\usepackage{amssymb,amsmath,mathtools}
\usepackage{stmaryrd}
\usepackage{enumerate}

\usepackage{amsmath}

\usepackage{bm}
\usepackage{commath}

\usepackage[notcite,notref]{showkeys} 
\usepackage{color}
\definecolor{Blue}{rgb}{0,0,1}
\definecolor{Red}{rgb}{1,0,0}

\newcommand{\continuousTerm}{\upsilon}
\newcommand{\A}{{\vec A}}
\newcommand{\B}{{\vec B}}

\newcommand{\reviewerA}[1]{#1}
\newcommand{\reviewerB}[1]{#1}
\newcommand{\reviewerC}[1]{#1}
\newcommand{\reviewerANew}[1]{#1}
\newcommand{\reviewerCNew}[1]{#1}
\newcommand{\reviewerANewer}[1]{#1}
\newcommand{\reviewerCNewer}[1]{#1}

\newcommand{\state}{\vec x}	
\newcommand{\f}{\vec f}	
\newcommand{\zero}{\vec 0}	
\newcommand{\D}{\vec D}	
\newcommand{\lipschitzConstant}{\kappa}	
\newcommand{\ndof}{N}	
\newcommand{\statevarf}{\vec{\xi}}	
\newcommand{\timevarf}{\tau}	
\newcommand{\maxT}{T}	
\newcommand{\timedomain}{\ensuremath{\left[0,\maxT\right]}}	
\newcommand{\velocityDummy}{\vec v}	
\newcommand{\velocityDummyRed}{\hat{\vec v}}	
\newcommand{\objFunNo}{g}	
\newcommand{\objFun}[1]{\objFunNo\left(#1\right)}	
\newcommand{\stateInitialNo} {\ensuremath{\state_0}}
\newcommand{\podstate}{\ensuremath{\vec \Phi}}	
\newcommand{\relProjErrorNo}{\ensuremath{\bar \mu_\star}}	
\newcommand{\relProjErrorNok}{\ensuremath{\relProjErrorNo^k}}	
\newcommand{\relProjErrork}[2]{\ensuremath{\relProjErrorNok(#1,#2)}}	
\newcommand{\testBasis}{\ensuremath{\vec \Psi}}	
\newcommand{\testBasisEntry}[2]{\ensuremath{\psi_{#1#2}}}	
\newcommand{\testBasisijEntry}[2]{\ensuremath{\left[\testBasis^\timestepit_{ij}\right]_{#1#2}}}	
\newcommand{\testBasisiEntry}[2]{\ensuremath{\left[\testBasis^\timestepit_{i}\right]_{#1#2}}}	
\newcommand{\podstateEntry}[2]{\phi_{#1 #2}}	
\newcommand{\podstateRK}{\ensuremath{\bar{\podstate}}}	
\newcommand{\podstateRKEntry}[2]{\bar\phi_{#1#2}}	
\newcommand{\podres}{\ensuremath{\podstate_{r}}}	
\newcommand{\podjac}{\ensuremath{\podstate_{J}}}	
\newcommand{\sampleMat}{\ensuremath{\vec P}}	
\newcommand{\res}{\ensuremath{{\vec r}}}	
\newcommand{\resEntry}[1]{\ensuremath{r_{#1}}}	
\newcommand{\resiEntry}[1]{\ensuremath{[\res_i^\timestepit]_{#1}}}	
\newcommand{\RRed}{\ensuremath{\hat \res}}	
\newcommand{\resRK}{\ensuremath{\bar\res}}	
\newcommand{\resRKexplicit}{\ensuremath{\vec q}}	
\newcommand{\resRKexplicitRed}{\hat{\resRKexplicit}}	
\newcommand{\resRKexplicitiEntry}[1]{\ensuremath{[\resRKexplicit^\timestepit_i]_{#1}}}	
\newcommand{\resRKEntry}[1]{\bar r_{#1}}	
\newcommand{\stateRed}{\ensuremath{{ \hat \state}}}	
\newcommand{\stateRedP}{\ensuremath{\stateRed_P}}	
\newcommand{\stateRedG}{\ensuremath{{\stateRed_G}}}	
\newcommand{\stateRedFOM}{\ensuremath{{\stateRed_\star}}}	
\newcommand{\optVec}{\ensuremath{\vec z}}	
\newcommand{\optVecRed}{\ensuremath{\hat \optVec}}	
\newcommand{\unknown}{\ensuremath{\vec w}}	
\newcommand{\unknownEntry}[1]{ w_{#1}}	
\newcommand{\unknownjEntry}[1]{ [\unknown_j]_{#1}}	
\newcommand{\unknownRK}{\ensuremath{\bar\unknown}}	
\newcommand{\unknownRKRed}{\ensuremath{\hat{\bar\unknown}}}	
\newcommand{\unknownRKEntry}[1]{\bar w_{#1}}	
\newcommand{\nstate}{\ensuremath{{p}}}	
\newcommand{\range}[1]{\ensuremath{\mathrm{Ran}\left(#1\right)}}	
\newcommand{\unknownRed}{\ensuremath{ \hat \unknown}}	
\newcommand{\unknownRedPG}{\ensuremath{ \hat {\vec y}}}	
\newcommand{\unknownVariable}{\ensuremath{\vec u}}	
\newcommand{\dt}{\ensuremath{\Delta t}}	
\newcommand{\rightChol}{\ensuremath{\vec U}}	
\newcommand{\rightCholEntry}[2]{\ensuremath{u_{#1#2}}}	
\newcommand{\rightCholDIRK}{\ensuremath{\underline{\vec U}}}	
\newcommand{\rightCholDIRKEntry}[3]{\ensuremath{[\rightCholDIRK_{#3}]_{#1#2}}}	
\newcommand{\rightCholRK}{\ensuremath{\bar{\vec U}}}	
\newcommand{\rightCholRKEntry}[2]{\ensuremath{\bar u_{#1#2}}}	
\newcommand{\weightingMatrix}{\ensuremath{\vec A}}	
\newcommand{\timestepit}{\ensuremath{n}}	
\newcommand{\nrows}{\ensuremath{z}}	

\newcommand{\bbP}{\mathbb{P}}

\newcommand{\bbR}{\mathbb{R}}

\newcommand{\bbV}{\mathbb{V}}

\newcommand{\mat}[1]{\bm{{#1}}}

\newcommand{\matI}{\mat I}

\newcommand{\normtwo}[1]{\norm{#1}_{\reviewerA{2}}}

\renewcommand{\vec}[1]{\bm{{#1}}}

\newcommand{\threshold}{\epsilon}
\newcommand{\thresholdLSPG}{\bar\epsilon}
\newcommand{\thresholdRK}{\omega}
\newcommand{\thresholdRKLSPG}{\bar\omega}

\newtheorem{theorem}{Theorem}[section]
\newtheorem{lemma}[theorem]{Lemma}
\newtheorem{corollary}[theorem]{Corollary}
\newtheorem{remark}[theorem]{Remark}

\newenvironment{proof}{\paragraph{Proof}}{$\square$}

\newcommand{\defeq}{\vcentcolon=}

\newcommand{\natNo}{\ensuremath{\mathbb N}} 
\newcommand{\nat}[1]{\ensuremath{\natNo(#1)}} 
\newcommand{\natZero}[1]{\ensuremath{\natNo_0(#1)}} 
\newcommand{\innat}[1]{\ensuremath{\in\nat{#1}}} 
\newcommand{\innatsequence}[1]{\ensuremath{\reviewerC{=1,\ldots,#1}}} 

\newcommand{\snapsNo}{\ensuremath{\mathcal X}}
\newcommand{\w} {\ensuremath{{\boldsymbol w}}} 

\newcommand{\nsnap} {\ensuremath{{n_w}}}
\newcommand{\energyCrit} {\ensuremath{\nu}}
\newcommand{\podArgsNo}{\ensuremath{\podstate}}  %
\newcommand{\podArgs}[2]{\ensuremath{\podArgsNo\left(#1,#2\right)}} %
\newcommand{\snapmat} {\ensuremath{\boldsymbol W}}

\newcommand{\nenergy}{\ensuremath{n_e}}

\newcommand{\pressure}{\ensuremath{p}}
\newcommand{\pressuretruth}{\ensuremath{p_\star}}
\newcommand{\errorValue}{\ensuremath{\varepsilon}}
\newcommand{\timeproj}{\ensuremath{P_\star}}
\newcommand{\dttruth}{\ensuremath{\dt_\star}}

\newcommand{\resFOM}{\ensuremath{\bar\res_\star^n}}
\newcommand{\resGal}{\ensuremath{\bar{\res}_G^n}}
\newcommand{\resDisc}{\ensuremath{\bar{\res}_P^n}}
\newcommand{\resFOMTotal}[1]{\ensuremath{\res_\star^{#1}}}
\newcommand{\resGalTotal}[1]{\ensuremath{{\res}_G^{#1}}}
\newcommand{\resDiscTotal}[1]{\ensuremath{{\res}_P^{#1}}}
\newcommand{\deltaresDisc}{\ensuremath{\delta{\res}_P}}
\newcommand{\deltastateRedG}{\ensuremath{\delta{\state}_G}}
\newcommand{\deltastateRedP}{\ensuremath{\delta{\state}_P}}

\newcommand{\firstCoeffMulti}[1]{\ensuremath{\varepsilon_{#1}}}
\newcommand{\firstCoeffMultiD}[1]{\ensuremath{\bar\varepsilon_{#1}}}

\newcommand{\secondCoeffMulti}[1]{\ensuremath{\gamma_{#1}}}
\newcommand{\secondCoeffMultiD}[1]{\ensuremath{\bar\gamma_{#1}}}
\newcommand{\firstCoeffMultin}[1]{\ensuremath{\varepsilon^n_{#1}}}
\newcommand{\firstCoeffMultinTwoarg}[2]{\ensuremath{\varepsilon^{#2}_{#1}}}

\newcommand{\secondCoeffMultin}[1]{\ensuremath{\gamma^n_{#1}}}
\newcommand{\secondCoeffMultinTwoarg}[2]{\ensuremath{\gamma^{#2}_{#1}}}

\newcommand{\multiIndexSet}{\ensuremath{(\eta_i)}}
\newcommand{\multiIndex}[1]{\ensuremath{\eta_{#1}}}
\newcommand{\indicator}{\ensuremath{\boldsymbol 1}}

\newcommand{\ellstar}[1]{\ensuremath{\ell^{#1}}}
\newcommand{\elltwostar}[1]{\ensuremath{\ell_\varepsilon^{#1}}}
\newcommand{\elltwostarmax}{\ensuremath{\ell_\varepsilon^\star}}
\newcommand{\barelltwostar}[1]{\ensuremath{\bar\ell_\varepsilon^{#1}}}
\newcommand{\barelltwostarmax}{\ensuremath{\bar\ell_\varepsilon^\star}}
\newcommand{\timestepset}[1]{\ensuremath{\mathcal L(#1)}}
\newcommand{\sizetimestepset}{\ensuremath{\bar n}}
\newcommand{\bmax}{\ensuremath{b^\star}}
\newcommand{\astar}{\ensuremath{a^\star}}

\newcommand{\barastar}{\ensuremath{\bar a^\star}}

\newcommand{\betamax}{\ensuremath{\beta_\mathrm{max}}}

\newcommand{\aSymb}{\ensuremath{|\alpha^\star|}}
\newcommand{\bSymb}{\ensuremath{|\beta^\star|}}
\newcommand{\zSymb}{\ensuremath{k}}
\newcommand{\dSymb}{\ensuremath{|\alpha_0^\star|}}
\newcommand{\fSymb}{\ensuremath{|\beta_0^\star|}}
\newcommand{\tSymb}{\ensuremath{\dt}}
\newcommand{\kSymb}{\ensuremath{\lipschitzConstant}}

\newcounter{pfctr}
\newcounter{propctr}
\newcounter{proposctr}
\newcommand{\RR}[1]{\ensuremath{\mathbb{R}^{ #1 }}}

\newcommand{\vecmat}[2]{\left[#1^1 \ \cdots\ #1^{#2}\right]}

\graphicspath{{}}

\journal{Journal of Computational Physics}
\begin{document}
\numberwithin{equation}{section}
\begin{frontmatter}

\title{Galerkin v.\ \reviewerA{least-squares Petrov--Galerkin} projection\\ in nonlinear model reduction}
\author[sandia]{Kevin Carlberg\corref{sandiacor}}
\ead{ktcarlb@sandia.gov}
\ead[url]{sandia.gov/~ktcarlb}
\address[sandia]{Sandia National Laboratories}
\cortext[sandiacor]{7011 East Ave, MS 9159, Livermore, CA 94550. Sandia is a
multiprogram laboratory operated by Sandia
Corporation, a Lockheed Martin Company, for the United States Department of
Energy under contract DE-AC04-94-AL85000.}
\address[gmu]{George Mason University}
\cortext[gmucor]{4400 University Drive, MS: 3F2, Exploratory Hall, Room 4201, Fairfax, Virginia 22030.}
\author[sandia]{Matthew Barone\corref{sandiacor}}
\ead{mbarone@sandia.gov}
\author[gmu]{Harbir Antil\corref{gmucor}}
\ead{hantil@gmu.edu}

\begin{abstract}

\reviewerA{Least-squares Petrov--Galerkin (LSPG)} model-reduction techniques such as the Gauss--Newton with
Approximated Tensors (GNAT) method have shown promise, as they have generated
stable, accurate solutions for large-scale turbulent, compressible flow
problems where standard Galerkin techniques have failed. However, there has
been limited comparative analysis of the two approaches. This is due in part
to difficulties arising from the fact that Galerkin techniques perform
\reviewerC{optimal} projection \reviewerCNewer{associated with residual
minimization} at the time-continuous level, while
\reviewerA{LSPG} techniques do
so at the time-discrete level.

This work provides a detailed theoretical and \reviewerB{computational} comparison of the
two techniques for two common classes of time integrators: linear
multistep schemes and Runge--Kutta schemes. We present a number
of new findings, including conditions under which the \reviewerA{LSPG} ROM has
a time-continuous representation, conditions under which the two techniques
are equivalent, and time-discrete error bounds for the two approaches.
Perhaps most surprisingly, we demonstrate both theoretically and
\reviewerB{computationally} that decreasing the time step does not necessarily decrease the
error for the \reviewerA{LSPG} ROM; instead, the time step should be `matched'
to the spectral content of the reduced basis. In numerical experiments carried out on  a
turbulent compressible-flow problem with over one million unknowns, we show
that increasing the time step to an intermediate value decreases both the
error and the simulation time of the \reviewerA{LSPG} reduced-order model by
an order of magnitude.
\end{abstract}

\begin{keyword}
model reduction \sep  GNAT \sep \reviewerC{least-squares Petrov--Galerkin projection} \sep Galerkin
projection  \sep CFD 


\end{keyword}

\end{frontmatter}
\section{Introduction}

While modeling and simulation of parameterized systems has become an essential
tool in many industries, the computational cost of executing high-fidelity
simulations is infeasibly high for many time-critical applications. For
example, real-time scenarios (e.g., model predictive control) require
simulations to execute in seconds or minutes, while many-query scenarios
(e.g., statistical inversion) can require thousands of simulations
corresponding to different \reviewerC{parameter} instances of the system.

Reduced-order models (ROMs) have been developed to mitigate this computational
bottleneck. First, they execute an \emph{offline} stage during which
computationally expensive training tasks (e.g., evaluating the high-fidelity
model at several points in the \reviewerC{parameter} space) compute a representative
low-dimensional `trial' basis for the system state.  Then, during the
inexpensive \emph{online} stage, these methods quickly compute approximate
solutions for arbitrary points in the \reviewerC{parameter} space via projection: they compute
solutions in the span of the trial basis while enforcing the
high-fidelity-model residual to be orthogonal to a low-dimensional `test'
basis.  They also introduce other approximations in the presence of general
\reviewerC{nonlinearities} or non-affine parameter
dependence.  See
Ref.~\cite{surveyWillcoxGugercin} and references within for a survey of
current methods.

By far the most popular model-reduction technique for nonlinear ordinary
differential equations (ODEs) is Galerkin projection \cite{sirovich1987tad3},
wherein the test basis is set to be equal to the trial basis, which is often
computed via proper orthogonal decomposition (POD) \cite{POD}. \reviewerA{Galerkin projection
can be considered \emph{continuous optimal}, as the 
Galerkin-ROM velocity minimizes the ODE (time-continuous) residual
in the $\ell^2$-norm.}
In addition, for specialized dynamical systems
(e.g., Lagrangian dynamical systems), performing Galerkin projection is
necessary to preserve problem structure
\cite{lall2003structure,carlbergStructureAiaa,carlberg2012spd}.  However,
theoretical \reviewerC{analyses}---in the form
of time-continuous error bounds \cite{rathinam:newlook} and 
stability analysis \cite{foias1991dissipativity}---as well as numerical experiments
have shown that Galerkin projection can lead to significant problems when
applied to general nonlinear ODEs: instability \cite{rempfer2000low},
inaccurate long-time responses \cite{sirisup2004spectral,noack2005need}, and no
guarantee of \textit{a priori} convergence (i.e., adding basis vectors can
degrade the solution) \reviewerB{\cite[Section 5]{rowley2004mrc}}. In turbulent fluid flows, some
of this poor performance can be attributed to the trial basis `filtering out'
small-scale modes essential for energy dissipation.

To address these shortcomings, alternative projection techniques have been
developed, particularly in fluid dynamics. These include stabilizing inner
products 
\cite{rowley2004mrc,barone2009stable,kalashnikova2010stability};  introducing
dissipation via closure models
\cite{aubry1988dynamics,sirisup2004spectral,Bergmann2009516,wang2012proper,san2013proper}
or numerical dissipation \cite{iollo2000stability};
performing nonlinear Galerkin projection based on approximate inertial
manifolds \cite{marion1989nonlinear,shen1990long,jolly1991preserving};
including a pressure-term representation \cite{noack2005need,galletti2004low};
modifying the POD basis by including many modes (such that dissipative
modes are captured), changing the norm \cite{iollo2000stability},
\reviewerC{enabling}
adaptivity \cite{Bergmann2009516}, or \reviewerC{including} basis functions that resolve a
range of scales \cite{balajewicz2012stabilization} or respect the attractor's
power balance \cite{balajewicz2013low}; and performing
Petrov--Galerkin projection \cite{Fang2013540}.

Alternatively, a promising new model-reduction methodology eschews Galerkin
projection in favor of performing projection at the \emph{fully discrete
level}, i.e., after the ODE has been discretized in time \cite{CarlbergGappy}.
This \emph{discrete-optimal} method\reviewerC{---known as least-squares
Petrov--Galerkin (LSPG) projection---}computes the solution that minimizes the
$\ell^2$-norm of the \reviewerC{(time-discrete)} residual arising at each time
step; this \reviewerANew{ensures that} adding basis vectors
\reviewerANew{yields a} monotonic decrease in the least-squares objective function.
When equipped with gappy POD \cite{sirovichOrigGappy} \reviewerC{(a least-squares
generalization of---and precursor to---the discrete empirical interpolation
method \cite{chaturantabut2010journal}) to approximate}
the discrete residual as a complexity-reduction mechanism,
this approach is known as the Gauss--Newton with Approximated Tensors (GNAT)
method \cite{carlbergJCP}. \reviewerB{While LSPG projection does
not necessarily guarantee \textit{a priori} accuracy and stablility for
turbulent, compressible flows, it has been computationally demonstrated to
\reviewerANew{generate accurate and stable responses for such problems on
which Galerkin projection yielded unstable responses
\cite{CarlbergGappy,carlbergHawaii,carlbergJCP}.}}

In spite of its promise, theoretical analysis has been limited to developing
consistency conditions for snapshot collection
\cite{CarlbergGappy,carlbergJCP} and discrete-time error bounds for simple
time integrators \cite{carlbergJCP,amsallemLocalGnat}.  In particular, major outstanding
questions include: (1) What are time-continuous and time-discrete
representations of the Galerkin and \reviewerA{LSPG} ROMs for broad classes of
time integrators? (2) Are there conditions under which the two techniques are
equivalent? (3) What discrete-time error bounds are available for the two
techniques for broad classes time integrators? Related to the third issue is
how parameters (e.g., time step or basis dimension) for the \reviewerA{LSPG} ROM
should be chosen to optimize performance.  This work aims to fill this gap by
performing a number of detailed theoretical and \reviewerB{computational} studies that
compare Galerkin and \reviewerA{LSPG} ROMs for the two most important classes
of time integrators: linear multistep methods and Runge--Kutta schemes. We
summarize the most important theoretical results (which map to the three
questions above) as follows:

\newcommand{\myline}[2]{
\path(#1.south) --(#2.north)  coordinate[pos=0.4](mid);
\draw[-latex,black,very thick] (#1.south) -| (mid) -| (#2.north);
}
\newcommand{\mylineHoriz}[2]{
\path(#1.west) --(#2.east)  coordinate[pos=0.4](mid);
\draw[-latex,black,very thick] (#1.west) |- (mid) |- (#2.east);
}

 \begin{figure}[htbp]
  \centering
\scriptsize
\begin{center}

  \begin{tikzpicture} [node distance=.5cm, start chain=going below,inner sep=5pt]
   \node[punktchain, join, text width=9em,minimum height=0.5cm] (ODE) {Full-order model\\ ODE};
	\begin{scope}
	[start branch=pgCont going left]
		 \node[approxMeth,text width=7.5em,minimum height=0.5cm] (PgprojCont) {Petrov--Galerkin\\ projection};
		 \node[punktchaindash] (DiscOptROMODE) {\reviewerA{Least-squares\\
		 Petrov--Galerkin} ROM\\ ODE};
	\begin{scope}
	[start branch=discOptGo going below]
		 \node[approxMeth,text width=9em,minimum height=0.5cm] (discOptprojtd) {time discretization};
	\end{scope}
	\end{scope}
	\begin{scope}
	[start branch=galCont going right]
		 \node[approxMeth,join, text width=5em,minimum height=0.5cm] (GalprojCont)
		 {Galerkin projection};
		 \node[optpunktchain,join] (DiscOptROM) {\reviewerA{Galerkin} ROM\\ ODE};

	\begin{scope}
	[start branch=galDiscGo going below]
		 \node[approxMeth,join, text width=9em,minimum height=0.5cm] (Galprojtd) {time discretization};
	\end{scope}
	\end{scope}
   \node[approxMeth, join, text width=9em,minimum height=0.5cm] (FOMtd) {time discretization};
   \node[punktchain, join, text width=9em,minimum height=0.5cm] (ODeltaE) {Full-order model\\ O$\Delta$E};
	\begin{scope}
	[start branch=galDisc going right]
		 \node[approxMeth,join, text width=5em,minimum height=0.5cm] (Galproj) {Galerkin\\ projection};
		 \node[punktchain,join] (galODeltaE) {\reviewerA{Galerkin} ROM\\ O$\Delta$E};
	\end{scope}
	\begin{scope}
	[start branch=discopt going left]
		 \node[approxMeth,join, text width=7.5em,minimum height=0.5cm]
		 (optProjDisc) {\reviewerA{Petrov--Galerkin\\ projection}};
		 \node[optpunktchain,join] (DiscOptROM) {\reviewerA{Least-squares\\
		 Petrov--Galerkin} ROM\\ O$\Delta$E};
	\end{scope}
\myline{Galprojtd}{galODeltaE}
\path [line,dashed] (PgprojCont) -- (ODE);
\path [line,dashed] (DiscOptROMODE) -- (PgprojCont);
\path [line,dashed] (discOptprojtd) -- (DiscOptROMODE);
\path [line,dashed] (DiscOptROM) -- (discOptprojtd);
  \end{tikzpicture}
  \end{center}
\caption{Relationship between Galerkin and discrete-optimal ROMs at the
time-continuous \reviewerB{(ODE)} and time-discrete \reviewerB{(O$\Delta$E)}
levels. Bolded outlines imply \reviewerA{the ROM associates with a minimum-residual
solution at that time-discretization level}.
Dashed lines imply the relationships are valid under certain conditions (see
Theorems \ref{thm:discOptROMcontDiscEqual} and
\ref{thm:discOptROMcontDiscEqualRK}).}
\label{fig:flowchart}
\end{figure}
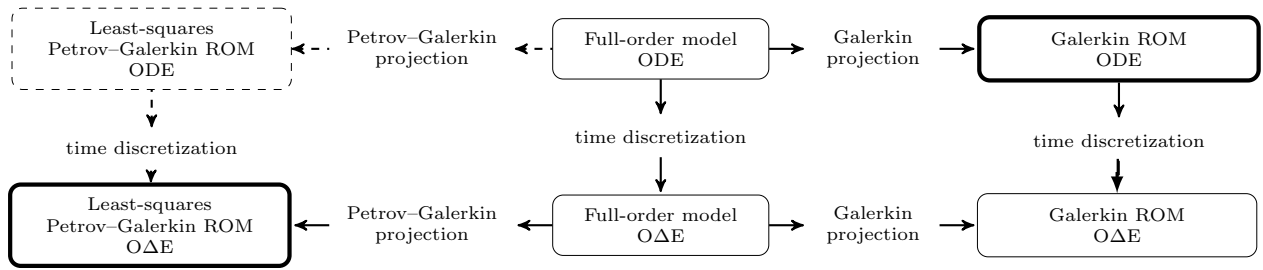
 \begin{enumerate} 
 \item Continuous and discrete representations
 \begin{itemize} 
  \item \reviewerC{Galerkin projection and time discretization are
	commutative} (Theorem \ref{thm:galProjDiscCommutative}).
	\item \reviewerA{LSPG} ROMs can be derived for 
	Runge--Kutta schemes (Section \ref{sec:discreteOptDiscrete}). 
	\item The \reviewerA{LSPG} ROM has a time-continuous (i.e., ODE) representation under
	certain conditions (\reviewerC{Section \ref{sec:LSPGcontinuous}}, Figure
	\ref{fig:flowchart}). This ODE depends on the time step $\dt$.
 \end{itemize} 
 \item Equivalence conditions
 \begin{itemize} 
	\item Galerkin \reviewerC{and LSPG} ROMs are \reviewerC{equivalent} for explicit time integrators
	(Corollaries \ref{cor:GalDiscOptExpLin} and \ref{cor:GalDiscOptExpRK}).
  \item  Galerkin \reviewerC{and LSPG} ROMs are \reviewerC{equivalent} in the limit
	of $\dt\rightarrow 0$ for implicit time integrators (Theorem \ref{thm:dtZeroEquiv}).
	\item Galerkin ROMs are discrete optimal for \reviewerC{symmetric-}positive-definite residual
	Jacobians (\reviewerC{Theorems \ref{thm:galDiscOptLinear}--\ref{thm:galDiscOptRK}}).
 \end{itemize} 
 \item Error analysis
 \begin{itemize} 
	\item We provide \reviewerCNewer{local and global \textit{a posteriori} and \textit{a
	priori}} error bounds for both Galerkin and
	\reviewerA{LSPG} ROMs for linear multistep schemes
	\reviewerC{(Section \ref{sec:errorLinear})}.
	\item For \reviewerCNewer{backward differentiation formulas}, we show that the
	\reviewerA{LSPG} ROM \reviewerC{can yield} a lower \reviewerCNewer{local \textit{a
	posteriori}} error bound than the Galerkin ROM (Corollary \ref{cor:discreteOptBeatsGal}) because it solves a
	time-global optimization problem (over a time window) rather than a
	time-local optimization problem \reviewerCNewer{(Remark \ref{rem:optimalityCompare})}.
	\item For the backward Euler time integrator, we show that \emph{an
	intermediate} time step should yield the lowest error bound (Corollary
	\ref{corr:auxiliaryProblem} \reviewerC{and Remark \ref{rem:modestTimestep}}).
	\item For the backward Euler time integrator, we show that a larger basis
	size leads to a smaller optimal time step for the \reviewerA{LSPG}
	ROM (Corollary \ref{corr:auxiliaryProblem}).
	\item We provide \reviewerCNewer{global \textit{a posteriori} and \textit{a
	priori}} error bounds for \reviewerA{both Galerkin
	and LSPG ROMs}
	for Runge--Kutta schemes (Section \ref{sec:errorRK}).
 \end{itemize} 
	 \end{enumerate}
Figure \ref{fig:flowchart} summarizes time-continuous and time-discrete
representations of the two techniques.

In addition to the above theoretical results, we present numerical results for
a large-scale compressible fluid-dynamics problem with turbulence
model\reviewerC{ing}
characterized by over one million degrees of freedom. These results illustrate
the practical significance of the above theoretical results. Critically, we show that
employing an intermediate time step for the \reviewerA{LSPG} ROM can decrease both the error and the
simulation time by an order of magnitude, which is a highly non-intuitive
result that is of immense practical significance.

The remainder of the paper is organized as follows. Section \ref{sec:FOM}
formulates the full-order model, including its representation at the
time-continuous and time-discrete levels. Section \ref{sec:Galerkin} presents
the Galerkin ROM at the continuous and discrete levels, and Section
\ref{sec:discreteOpt} does so for the \reviewerA{LSPG} ROM. In particular, we
provide conditions under which the \reviewerA{LSPG} ROM has a time-continuous
representation. Section \ref{sec:comparative} provides conditions under which
the Galerkin and \reviewerA{LSPG} ROMs are equivalent; in particular,
equivalence holds for explicit integrators (Section \ref{sec:equalExplicit}),
in the limit of the time step going to zero for implicit integrators (Section
\ref{sec:equalDt}), and for symmetric-positive-definite residual Jacobians
(Section \ref{sec:equalSPD}). Section \ref{sec:error} provides error analysis
for Galerkin and \reviewerA{LSPG} ROMs for linear multistep schemes (Section
\ref{sec:errorLinear}), Runge--Kutta schemes (Section \ref{sec:errorRK}),
and a detailed analysis in the case of backward Euler (Section
\ref{sec:errorBackEuler}). Section \ref{s:numerics} provides detailed
numerical examples that illustrate the practical importance of the analysis.
Finally, Section \ref{sec:conclusions} provides conclusions.

In the remainder of this paper,  matrices are denoted by capitalized bold
letters, vectors by lowercase bold letters, scalars by unbolded letters. The
columns of a matrix $\bm{A}\in\RR{m\times {\timestepit}}$ are denoted by
$\bm{a}_i\in\RR{m}$, $i\innat{{\timestepit}}$ with $\nat{a}\defeq\{1,\ldots,
a\}$ such that $\bm{A}\defeq\left[\bm{a}_1\ \cdots\
\bm{a}_{\timestepit}\right]$. \reviewerCNewer{We also define $\natZero{a}\defeq\{0,\ldots,
a\}$.} The scalar-valued matrix elements are denoted
by $a_{ij}\in\RR{}$ such that  $\bm{a}_j\defeq\left[a_{1j}\ \cdots\
a_{mj}\right]^T$, $j\innat{{\timestepit}}$. A superscript
denotes the value of a variable at that time instance, e.g.,
$\state^\timestepit$ is the value of $\state$ at time $\timestepit\dt$, where $\dt$ is the time
step.

%
%
\section{Full-order model}\label{sec:FOM}
We begin by formulating both the time-continuous (ODE) and time-discrete
(O$\Delta$E) representations of the full-order model (FOM).
\subsection{Continuous representation}
In this work, we consider the full-order model (FOM) to be an initial value
problem characterized by a system of
nonlinear ODEs
\begin{align} \label{eq:ODE}
\frac{d\state}{dt}  &= {\f }(\state,t) , \qquad 
\state(0) = \stateInitialNo,
\end{align} 
where $\state:\timedomain\rightarrow\RR{\ndof }$ denotes the (time-dependent)
state, $\stateInitialNo\in\RR{\ndof}$ denotes the
initial condition, and $\f :\RR{\ndof }\times \timedomain\rightarrow\RR{\ndof }$ with
$\left(\statevarf,\timevarf\right)\mapsto\f\left(\statevarf,\timevarf\right)$.
This ODE can arise, for example, from applying spatial discretization (e.g.,
finite element, finite volume, or finite difference) to a
partial differential equation with time dependence. We note that most
model-reduction techniques are applied to parameterized systems wherein the velocity
$\f$ \reviewerC{is also parameter dependent}. However, we limit our presentation to
unparameterized systems for notational simplicity, as we are
interested comparing Galerkin and \reviewerA{LSPG} ROMs for a given instance
of the ODE.

\subsection{Discrete representation}
A time-discretization method is required to solve Eq.~\eqref{eq:ODE}
numerically. We now characterize the full-order-model O$\Delta$E, which is the time-discrete
representation of the model, for two classes of time integrators:
linear multistep schemes and Runge--Kutta schemes.
\subsubsection{Linear multistep schemes}
A linear $k$-step method applied to numerically solve Eq.\ \eqref{eq:ODE} can be expressed as
 \begin{equation} \label{eq:linearMultistepDef}
 \sum_{j=0}^k\alpha_j\state^{\timestepit -j} = \dt \sum_{j=0}^k\beta_j \f \left(\state^{\timestepit -j},
 t^{\timestepit -j}\right),
 \end{equation} 
 where $\dt$ is the time step\reviewerA{, the coefficients $\alpha_j$ and $\beta_j$ define a specific linear multistep
 scheme}, $\alpha_0\neq 0$, and $\sum\limits_{j=0}^k\alpha_j = 0$ is necessary for
 consistency.  In this case, the O$\Delta$E is characterized by the following
 system of \reviewerC{algebraic} equations to be solved at each time instance $\timestepit\in\nat{\maxT/\dt}$:
 \begin{equation} \label{eq:resLinMultiSolve}
\res^\timestepit \left(\unknown^\timestepit \right)=0,
 \end{equation}
 where $\unknown^\timestepit\in\RR{\ndof}$ is the unknown variable  and
 $\res^\timestepit:\RR{\ndof}\rightarrow\RR{\ndof}$ denotes the 
 linear multistep residual defined as
 \begin{equation} \label{eq:resLinMulti}
 \res^\timestepit \left(\unknown\right) \defeq \alpha_0 \unknown - \dt \beta_0 \f (\unknown,t^\timestepit )+
 \sum_{j=1}^k\alpha_j\state^{\timestepit -j} -
 \dt \sum_{j=1}^k\beta_j \f \left(\state^{\timestepit -j},t^{\timestepit -j}\right).
  \end{equation} 
	Then, the state can be updated explicitly as 
 \begin{equation*} 
 \state^\timestepit  = \unknown^\timestepit .
  \end{equation*} 
Hence, the unknown is equal to the state.
These methods are implicit if $\beta_0\neq 0$.

\subsubsection{Runge--Kutta schemes}
For an $s$-stage Runge--Kutta scheme, the O$\Delta$E is characterized by the
following system of \reviewerC{algebraic} equations to be
solved at each time step:
\begin{equation} \label{eq:resRKSolve}
\res^\timestepit _i\left(\unknown^\timestepit _1,\ldots,\unknown^\timestepit
_s\right) = 0,\quad i\innat{s}.
\end{equation} 
Here, the Runge--Kutta residual is defined as
\begin{equation} \label{eq:resRK}
\res^\timestepit _i\left(\unknown_1,\ldots,\unknown_s\right) \defeq \unknown_i -
 \f (\state^{\timestepit -1} + \dt \sum_{j=1}^sa_{ij}\unknown_j, t^{\timestepit -1} +
c_i\dt),\quad i\innat{s}
\end{equation} 
and the state is explicitly updated as
\begin{equation} \label{eq:fom_update}
\state^{\timestepit } = \state^{\timestepit -1} + \dt \sum_{i=1}^sb_i
\unknown^\timestepit _i.
\end{equation} 
 Here, the unknowns $\unknown^\timestepit_i$ correspond to the velocity $ d\state/dt$ at
 times $t^{\timestepit -1} + c_i\dt$, $i\innat{s}$\reviewerA{, and the coefficients
 $b_i$, $c_i$, and $a_{ij}$ define a specific Runge--Kutta scheme}.
These methods are \reviewerC{explicit if $a_{ij}= 0$, $\forall j \geq i$ and
are diagonally implicit if $a_{ij}=0$, $\forall j> i$}.

\section{Galerkin ROM}\label{sec:Galerkin}
This section provides the time-continuous and time-discrete representations of
the Galerkin ROM, as well as key results related to optimality and
commutativity of projection and time discretization.
\subsection{Continuous representation}
Galerkin-projection reduced-order models compute an approximate solution
$\tilde \state\approx \state$ with $\tilde \state\in\RR{\ndof}$ to
Eq.~\eqref{eq:ODE} by introducing two approximations. First, they restrict the
approximate solution to lie in a low-dimensional affine trial subspace $\stateInitialNo +
\range{\podstate}$,
where $\podstate\in\RR{\ndof \times \nstate}$ with $\podstate^T\podstate =
\reviewerCNewer{\matI}$
denotes the given reduced basis (in matrix form) of dimension $\nstate \ll
\ndof $ \reviewerCNewer{and $\range{\weightingMatrix}$ denotes the range of
matrix $\weightingMatrix$}.
This basis can be computed by a variety of techniques, e.g., eigenmode analysis,
POD \cite{POD}, or the reduced-basis method
\cite{prud2002reliable,rozza2007reduced,veroy2003peb,nguyen2005certified,veroy2005certified}.
Then, 
the approximate solution can be expressed as 
\begin{equation} \label{eq:romSubspace}
\tilde \state(t) = \stateInitialNo + \podstate \stateRed(t),
\end{equation} 
where
$\stateRed:\timedomain\rightarrow\RR{\nstate}$ denotes the (time-dependent) generalized
coordinates of the approximate solution. Second,
these methods substitute $\state \leftarrow \tilde \state$ into Eq.\ \eqref{eq:ODE} and enforce the
ODE residual to be orthogonal to $\range{\podstate}$, which results in a
low-dimensional system of nonlinear ODEs
\begin{align} \label{eq:romODE}
\frac{d\stateRed}{dt} &= \podstate^T \f (\stateInitialNo + \podstate\stateRed,t) , \qquad
\stateRed(0) = 0.
\end{align} 
\begin{remark} \rm
In order to obtain computational efficiency, it is necessary to reduce the
computational complexity of repeatedly computing matrix--vector products of
the form $\podstate^T \f$. This can be done using a variety of methods, e.g.,
collocation \cite{astrid2007mpe,ryckelynck2005phm,LeGresleyThesis}, gappy POD
\cite{sirovichOrigGappy,bos2004als,astrid2007mpe,CarlbergGappy,carlbergJCP},
the discrete empirical interpolation method (DEIM)
\cite{barrault2004eim,chaturantabut2010journal,galbally2009non,drohmannEOI,
HAntil_MHeinkenschloss_DCSorensen_2013a}, reduced-order quadrature
\cite{HAntil_SField_RHNochetto_MTiglio_2013},
finite-element subassembly methods
\cite{an2008optimizing,farhat2014dimensional}, or reduced-basis-sparsification
techniques \cite{carlberg2012spd}. However, in this work we limit ourselves to
comparatively analyzing different projection techniques. For this reason, we do
not perform additional analysis for such complexity-reduction mechanisms; this
is the subject of follow-on work.
\end{remark}

We now restate the well-known result that Galerkin projection leads to a
notion of \reviewerC{minimum-residual} optimality at the continuous level. This is reflected in the
top-right box of Figure \ref{fig:flowchart}, where the bolded outline
indicates \reviewerC{minimum-residual} optimality.
\begin{theorem}[Galerkin: continuous optimality]\label{thm:GalContOpt}
\reviewerC{If the reduced basis is orthogonal, i.e., $\podstate^T\podstate =
\reviewerCNewer{\matI}$, then t}he Galerkin ROM \eqref{eq:romSubspace}--\eqref{eq:romODE} is continuous
optimal in the sense that the approximated velocity
minimizes the \reviewerA{$\ell^2$-norm of the FOM ODE residual \eqref{eq:ODE}}  over
$\range{\podstate}$, i.e.,
\begin{equation} \label{eq:optimalityOrth}
\frac{d{\tilde \state}}{dt}(\stateInitialNo + \podstate\stateRed,t) =
\arg\min_{\velocityDummy\in\range{\podstate}}\|\velocityDummy -
\f(\stateInitialNo + \podstate\stateRed,t)\|_2^2.
\end{equation} 
\end{theorem}
\begin{proof} 
Because $\frac{d{\tilde \state}}{dt}=
\podstate\frac{d\stateRed}{dt}$, problem \eqref{eq:optimalityOrth} can be
written as 
\begin{equation} \label{eq:optimalityOrth2}
\frac{d{ \stateRed}}{dt}(\stateInitialNo + \podstate\stateRed,t) =
\arg\min_{\velocityDummyRed\in\RR{\nstate}}\objFun{\velocityDummyRed}\reviewerC{,}
\end{equation} 
where $\objFun{\velocityDummyRed}\defeq\|\podstate\velocityDummyRed -
\f(\stateInitialNo + \podstate\stateRed,t)\|_2^2$. We now assess whether
Eq.~\eqref{eq:optimalityOrth2} holds, i.e., whether $\frac{d{ \stateRed}}{dt}$ as
defined by Eq.~\eqref{eq:romODE} is
the minimizer of $\objFunNo$.

The function $\objFunNo$ can be expressed as
$\objFun{\velocityDummyRed}=\velocityDummyRed^T\podstate^T\podstate\velocityDummyRed
- 2\velocityDummyRed^T\podstate^T\f(\stateInitialNo + \podstate\stateRed,t) +
\f(\stateInitialNo + \podstate\stateRed,t)^T\f(\stateInitialNo + \podstate\stateRed,t)$. Due to the strict convexity of the
function $\objFunNo$, the global minimizer $\velocityDummyRed^\star$ is
equal to the stationary point of $\objFunNo$, i.e.,
$\velocityDummyRed^\star$ satisfies
 \begin{gather} 
 0 = \frac{d g}{d\velocityDummyRed}\left(\velocityDummyRed^\star\right)=
  2\podstate^T\podstate \velocityDummyRed^\star -
	2\podstate^T\f(\stateInitialNo + \podstate\stateRed,t)\\
  \label{eq:optSolution}
	\velocityDummyRed^\star=
	\podstate^T\f(\stateInitialNo + \podstate\stateRed,t),
  \end{gather} 
	where orthogonality of $\podstate$ has been used. Comparing
	Eqs.~\eqref{eq:optSolution} and \eqref{eq:romODE} shows 
$\frac{d{ \stateRed}}{dt}(\stateInitialNo + \podstate\stateRed,t)=\velocityDummyRed^\star$, which is the desired result.
\end{proof} 
\begin{remark}[\reviewerANew{Galerkin ROM enrichment yields a monotonic decrease in the FOM ODE residual}] \rm
Due to optimality property \eqref{eq:optimalityOrth} \reviewerANew{of the
Galerkin ROM,} adding vectors to the trial basis---which
\reviewerCNew{enriches} the trial subspace $\range{\podstate}$---results in a
monotonic decrease \reviewerC{in} the \reviewerC{minimum-residual} objective
function in problem \eqref{eq:optimalityOrth}\reviewerANew{, which is simply the
$\ell^2$-norm of the FOM ODE residual}. \reviewerANew{Because the FOM ODE
residual is equivalent to the difference between the ROM and FOM velocities,
this implies that the $\ell^2$-norm of the error in the ROM velocity will
monotonically decrease as the trial subspace is enriched.}
\end{remark}


\reviewerA{Thus, Galerkin ROMs exhibit desirable properties (i.e.,
minimum-residual \reviewerANew{optimality}) at the
time-continuous level. We now derive a time-discrete representation for the
Galerkin ROM, noting that these properties are lost at the
time-discrete level.}

\subsection{Discrete representation}
As before, a time-discretization method is needed to numerically solve Eq.~\eqref{eq:romODE}. We
now characterize the O$\Delta$E for the Galerkin ROM.
\subsubsection{Linear multistep schemes}
A linear $k$-step method applied to numerically solve Eq.\ \eqref{eq:romODE}
can be expressed as
 \begin{equation*} 
 \sum_{j=0}^k\alpha_j\stateRed^{\timestepit -j} = \dt \sum_{j=0}^k\beta_j \podstate^T
 \f \left(\stateInitialNo + \podstate\stateRed^{\timestepit -j},
 t^{\timestepit -j}\right).
 \end{equation*} 
Here, the O$\Delta$E is characterized by the following system of 
\reviewerC{algebraic} equations to be solved at each time
step:
 \begin{equation} \label{eq:romResLinMultiSolve}
 \RRed^\timestepit \left(\unknownRed^\timestepit \right) = 0.
  \end{equation} 
Here, the discrete residual corresponds to 
 \begin{equation} \label{eq:romResLinMulti}
 \RRed^\timestepit \left(\unknownRed\right) \defeq \alpha_0\unknownRed - \dt \beta_0
 \podstate^T \f (\stateInitialNo + \podstate\unknownRed,t^\timestepit )+
 \sum_{j=1}^k\alpha_j\stateRed^{\timestepit -j} -
 \dt \sum_{j=1}^k\beta_j \podstate^T \f \left(\stateInitialNo +
 \podstate\stateRed^{\timestepit -j},t^{\timestepit -j}\right)
  \end{equation} 
	and the generalized state is explicitly updated as
 \begin{equation*} 
 \stateRed^\timestepit  = \unknownRed^\timestepit.
  \end{equation*} 

\subsubsection{Runge--Kutta schemes}
Applying an $s$-stage Runge--Kutta method to solve Eq.\ \eqref{eq:romODE}
leads to an O$\Delta$E characterized by the
following system of \reviewerC{algebraic} equations to be solved at each time step:
\begin{equation} \label{eq:romResRKSolve}
\RRed^\timestepit _i\left(\unknownRed^\timestepit
_1,\ldots,\unknownRed^\timestepit _s\right) = 0,\quad i\innat{s}.
\end{equation} 
Here, discrete the residual is defined as
\begin{equation} \label{eq:romResRK}
\RRed^\timestepit _i\left(\unknownRed_1,\ldots,\unknownRed_s\right) \defeq \vec
\unknownRed_i -
 \podstate^T\f (\stateInitialNo + \podstate\stateRed^{\timestepit -1} + \dt
 \sum_{j=1}^sa_{ij}\vec \podstate\unknownRed_j, t^{\timestepit -1} +
c_i\dt),\quad i\innat{s}
\end{equation} 
and the generalized state is computed explicitly as
\begin{equation} \label{eq:romG_RK_update}
\stateRed^{\timestepit } = \stateRed^{\timestepit -1} + \dt \sum_{i=1}^sb_i \unknownRed^\timestepit _i.
\end{equation} 

\reviewerA{Note that the Galerkin-ROM solution satisfying
Eqs.~\eqref{eq:romResLinMultiSolve} or \reviewerANew{\eqref{eq:romResRKSolve}} does
not in general associate with the solution to an optimization problem;
therefore, the optimality property the method exhibits at the continuous level
has been lost at the discrete level.} We now show that \reviewerC{Galerkin} projection and time discretization are
commutative\reviewerA{; this implies that
	Galerkin ROMs can be analyzed, implemented, and interpreted
equivalently at both the time-discrete and time-continuous levels.
} This corresponds to the rightmost part of Figure
\ref{fig:flowchart}.
\begin{theorem}[Galerkin: commutativity of projection and time discretization]\label{thm:galProjDiscCommutative}
\ \\ Performing a 
Galerkin projection on the governing ODE and subsequently
applying time discretization yields the same model as first applying time
discretization on the governing ODE and subsequently performing Galerkin
projection.
\end{theorem}
\begin{proof}
 \underline{Linear multistep schemes}. Eq.~\eqref{eq:romResLinMultiSolve} was
 derived by performing Galerkin projection on the continuous representation of
 the FOM and
 subsequently applying
 time discretization. If instead we apply Galerkin projection to the discrete
 representation of the FOM in Eq.~\eqref{eq:resLinMultiSolve},
 set $\unknown= \stateInitialNo+ \podstate \unknownRed$ and $\state^i= \stateInitialNo+ \podstate \stateRed^i$,
 $i\innat{\timestepit} $, and use
 $\sum_{j=1}^k\alpha_j=0$ and $\podstate^T\podstate = \reviewerCNewer{\matI}$, we obtain the
 following O$\Delta$E to be solved at each time step: $\vec \Phi^T \res^\timestepit \left(\stateInitialNo +
\podstate\unknownRed\right)=0$. Comparing the definition of the linear
multistep residual \eqref{eq:resLinMulti} with
Eq.~\eqref{eq:romResLinMulti} reveals
 \begin{equation} \label{eq:reducedRedGalLin}
\RRed^\timestepit \left(\unknownRed\right) = \vec \Phi^T \res^\timestepit \left(\stateInitialNo +
\podstate\unknownRed\right),
  \end{equation} 
	which shows that the same discrete equations $\RRed^\timestepit
	\left(\unknownRed\right)=0$ are obtained at each time step regardless of the
	ordering of time discretization and Galerkin projection.\\
\noindent\underline{Runge--Kutta schemes}. 
Eq.~\eqref{eq:romResRKSolve} was
 derived by performing Galerkin projection on the continuous FOM
 representation and
 then applying
 time discretization. If instead we apply Galerkin projection to the discrete
 FOM representation in Eq.~\eqref{eq:resRKSolve},
set $\state^{\timestepit -1}= \stateInitialNo + \podstate \stateRed^{\timestepit -1}$, $\unknown_i =
 \podstate\unknownRed_i$, $i\innat{s}$, and use $\podstate^T\podstate = \reviewerCNewer{\matI}$, we
 obtain the following O$\Delta$E to be solved at each time step: $\vec \Phi^T \res^\timestepit _i\left(
\podstate\unknownRed_1,\ldots,\podstate\unknownRed_s\right)=0,\
i\innat{s}$. Comparing the definition of the Runge--Kutta
residual \eqref{eq:resRK} with
Eq.~\eqref{eq:romResRK} reveals
  \begin{equation} \label{eq:reducedRedGalRK}
\RRed^\timestepit _i\left(\unknownRed_1,\ldots,\unknownRed_s\right) = \vec \Phi^T \res^\timestepit _i\left(
\podstate\unknownRed_1,\ldots,\podstate\unknownRed_s\right),\quad i\innat{s},
  \end{equation} 
which shows that the same discrete equations $\RRed^\timestepit
_i\left(\unknownRed_1,\ldots,\unknownRed_s\right)=0$, $i\innat{s}$ are
obtained at each time step regardless of the ordering of time discretization
and Galerkin projection.
\end{proof}

\section{\reviewerC{Least-squares Petrov--Galerkin} ROM}\label{sec:discreteOpt}

\reviewerC{Rather than performing minimum-residual optimal projection
on the full-order model ODE (i.e., at the continuous level), this can be
executed on the full-order model O$\Delta$E (i.e., at the discrete level).
Doing so enables \textit{discrete optimality}, which contrasts
with the continuous optimality exhibited by Galerkin projection. In
particular, we consider optimal projections that minimize the discrete
residual(s) (in some weighted $\ell^2$-norm) arising at each time instance.}

We note that other residual-minimizing approaches have been developed in the
case of \reviewerC{linear \cite{bui2008model} and nonlinear \cite{LeGresleyThesis}
steady-state problems,}
and space--time solutions \cite{constantineResMin}. In addition, a
recently developed approach \cite{abgrall2015robust} has suggested
\reviewerB{$L^1$} minimization of the
residual arising at each time instance for hyperbolic problems.

\subsection{Discrete representation}\label{sec:discreteOptDiscrete}
We begin by developing the time-discrete representation for the
\reviewerA{LSPG} ROM for both linear multistep schemes and Runge--Kutta schemes. The
latter 
is a novel contribution, as previous work has derived
discrete-optimal \reviewerA{LSPG} ROMs only for linear multistep schemes
\cite{CarlbergGappy,carlbergJCP}. Optimality of this approach corresponds to the bolded bottom-left box
of Figure \ref{fig:flowchart}.
\subsubsection{Linear multistep schemes}\label{s:disc_opt_integ}
As before with Galerkin projection, discrete-optimal ROMs compute 
solutions using two approximations. First, they restrict the
approximate solution to lie in the same low-dimensional affine trial subspace
$\tilde\state\in\stateInitialNo +
\range{\podstate}$ as Galerkin \reviewerC{methods}; thus, the approximate solution can be written according to
Eq.~\eqref{eq:romSubspace}. In the case of linear multistep schemes, the
unknown at time step $\timestepit$ is simply the state, i.e.,
$\unknown^\timestepit= \state^\timestepit$, which implies that $\tilde
\unknown^\timestepit = \stateInitialNo+ \podstate\unknownRed^\timestepit$.
Second, the discrete-optimal ROM substitutes
$\unknown^\timestepit \leftarrow \tilde \unknown^\timestepit$ into the O$\Delta$E \eqref{eq:resLinMultiSolve} and solves a
minimization problem to ensure the approximate solution is optimal in
\reviewerC{a minimum-residual sense}
at the discrete level:
 \begin{equation} \label{eq:discreteOptLinear}
\tilde \unknown^\timestepit = \arg
 \underset{\optVec\in\stateInitialNo +
 \range{\podstate}}{\min}\|
\weightingMatrix\left(\optVec\right)
\res^\timestepit \left(\optVec\right)\|_2^2
  \end{equation} 
	or equivalently
 \begin{equation} \label{eq:discreteOptLinear2}
\unknownRed^\timestepit = \arg
 \min_{\optVecRed\in\RR{\nstate}}\|
\weightingMatrix\left(\stateInitialNo + \podstate\optVecRed\right)
\res^\timestepit \left(\stateInitialNo + \podstate\optVecRed\right)\|_2^2.
  \end{equation} 
	Here, $\weightingMatrix\in\RR{\nrows\times\ndof}$ with $\nrows\leq\ndof$ is a
	weighting matrix that enables the definition of a weighted (semi)norm. Examples of
	such reduced-order models include the \reviewerA{LSPG} method
	\reviewerC{presented in Refs.\ \cite{CarlbergGappy,carlbergJCP,LeGresleyThesis}}
		($\weightingMatrix = \matI$)\reviewerC{, LSPG with collocation ($\weightingMatrix =
	\sampleMat$ with $\sampleMat$ consisting of selected rows of the
		identity matrix) \cite{LeGresleyThesis}, and the related GNAT method
	\cite{CarlbergGappy,carlbergJCP} ($\weightingMatrix =
	\left(\sampleMat\podres\right)^+\sampleMat$ with $\podres$ a basis
	for the residual and the superscript $+$ denoting the Moore--Penrose
		pseudoinverse).}

Note that \reviewerCNewer{necessary conditions for the solution} to
Eq.~\eqref{eq:discreteOptLinear2} \reviewerCNewer{corresponds to 
stationary of} the objective function in Eq.~\eqref{eq:discreteOptLinear2},
i.e., \reviewerCNewer{the solution} satisfies 
\begin{equation} \label{eq:LSPGlinmulti}
\testBasis^\timestepit(\unknownRed^\timestepit)^T
\res^\timestepit \left(\stateInitialNo +
\podstate\unknownRed^\timestepit\right) = 0\reviewerC{,}
\end{equation} 
where the entries of  $\testBasis^\timestepit\in\RR{\ndof\times\nstate}$ are
 \begin{align}\label{eq:LSPGlinmultiDef}
 \begin{split}
\testBasisEntry{i}{j}^\timestepit(\unknownRed)=&a_{mi}(\stateInitialNo
+\podstate\unknownRed)\frac{\partial a_{ml}(\stateInitialNo
+\podstate\unknownRed)}{\partial
\unknownEntry{k}}\podstateEntry{k}{j}\resEntry{\ell}^\timestepit(\stateInitialNo+\podstate\unknownRed)+\\
&a_{mi}(\stateInitialNo
+\podstate\unknownRed)a_{ml}(\stateInitialNo+\podstate\unknownRed)\frac{\partial
\resEntry{\ell}^\timestepit}{\partial
\unknownEntry{k}}(\stateInitialNo+\podstate\unknownRed)\podstateEntry{k}{j},\quad
i\in\nat{\ndof},\ j\in\nat{\nstate},
\end{split} 
\end{align} 
where a repeated index implies summation. Because Eq.~\eqref{eq:LSPGlinmulti}
corresponds to a Petrov--Galerkin projection with trial subspace
$\range{\podstate}$ and test subspace $\range{\testBasis}$, the
discrete-optimal projection \reviewerC{can be} referred to as a least-squares
Petrov--Galerkin projection \cite{carlbergJCP,CarlbergGappy}.

\subsubsection{Runge--Kutta schemes}

\reviewerA{LSPG} ROMs for Runge--Kutta schemes also approximate the solution according to
Eq.~\eqref{eq:romSubspace}. However, because the unknown at time step
$\timestepit$ and stage $i$ is
the velocity at an intermediate time point, i.e., $\unknown^\timestepit_i =
\dot \state\left(t^{\timestepit-1} + c_i\dt\right) $ for $i\innat{s}$, we have
$\tilde
\unknown_i^\timestepit = \podstate\dot\stateRed\left(t^{\timestepit-1} +
c_i\dt\right)$ for this case.
Then, these techniques substitute
$\unknown^\timestepit \leftarrow \tilde \unknown^\timestepit$ into the
O$\Delta$E \eqref{eq:resRKSolve} and solve the following
\reviewerC{minimum-residual} problem:
 \begin{equation} \label{eq:discreteOptRK}
\left(\tilde \unknown^\timestepit_1,\ldots,\tilde \unknown^\timestepit_s\right) = \arg
 \underset{\left(\optVec_1,\ldots,\optVec_1\right)\in
 \range{\podstate}^s}{\min}\sum_{i=1}^s\|
\weightingMatrix_i\left( \optVec_1,\ldots,\optVec_s\right)
\res_i^\timestepit \left( \optVec_1,\ldots,\optVec_s\right)\|_2^2
  \end{equation} 
	or equivalently
 \begin{equation} \label{eq:discreteOptRK2}
\left(\hat \unknown^\timestepit_1,\ldots,\hat\unknown^\timestepit_s\right) = \arg
 \underset{\left(\optVecRed_1,\ldots,\optVecRed_s\right)\in
 \RR{\nstate\times s}}{\min}\sum_{i=1}^s\|
\weightingMatrix_i\left( \podstate\optVecRed_1,\ldots,\podstate\optVecRed_s\right)
\res_i^\timestepit \left(
\podstate\optVecRed_1,\ldots,\podstate\optVecRed_s\right)\|_2^2.
  \end{equation} 
	Here,
	$\weightingMatrix_i\in\RR{\nrows\times\ndof}$, $i\in\nat{s}$ with
	$\nrows\leq\ndof$ are weighting matrices.
As before, the solution to Eq.~\eqref{eq:discreteOptRK2} corresponds to a
stationary point of the objective function,
i.e., it satisfies 
\begin{equation} \label{eq:LSPGRK}
\sum_{j=1}^s\testBasis_{ij}^\timestepit(\unknownRed_1^\timestepit,\ldots,\unknownRed_s^\timestepit)^T
\res_j^\timestepit \left(
\podstate\unknownRed_1^\timestepit,\ldots,\podstate\unknownRed_s^\timestepit\right)
= 0,\quad i=1,\ldots, s,
\end{equation} 
where entries of the test bases $\testBasis_{ij}^\timestepit\in\RR{\ndof\times \nstate}$,
$i,j\in\nat{s}$ are
\begin{align} \label{eq:discreteOptRKTest}
\begin{split}
\testBasisijEntry{k}{\ell}(\unknownRed_1,\ldots,\unknownRed_s)=& 
[\weightingMatrix_i]_{uk}(\podstate\unknownRed_1,\ldots,\podstate\unknownRed_s)\frac{\partial
[\weightingMatrix_i]_{um}(\podstate\unknownRed_1,\ldots,\podstate\unknownRed_s)}{\partial
\unknownjEntry{n}}\podstateEntry{n}{\ell}\resiEntry{m}(\podstate\unknownRed_1,\ldots,\podstate\unknownRed_s)+\\
&[\weightingMatrix_i]_{uk}(\podstate\unknownRed_1,\ldots,\podstate\unknownRed_s)[\weightingMatrix_i]_{um}(\podstate\unknownRed_1,\ldots,\podstate\unknownRed_s)\frac{\partial \resiEntry{m}}{\partial
\unknownjEntry{n}}(\podstate\unknownRed_1,\ldots,\podstate\unknownRed_s)\podstateEntry{n}{\ell},\\
\end{split}
\end{align} 
where $[\cdot]_{ij}$ denotes entry $(i,j)$ of the argument.
This again leads to a least-squares Petrov--Galerkin interpretation for the
discrete-optimal ROM.

In the explicit \reviewerA{or diagonally implicit} case\reviewerA{s}, we can consider \reviewerC{an alternative} notion of discrete
optimality. 
Explicit Runge--Kutta schemes are characterized by $a_{ij}=0$, $\forall j\geq
i$\reviewerA{, while diagonally implicit Runge--Kutta (DIRK)
schemes \cite{alexander1977diagonally} are characterized by $a_{ij}=0$, $\forall j>
i$}. In \reviewerA{these cases}, solutions $\unknown^\timestepit_i$, $i\innat{s}$ can be
computed sequentially, i.e., 
\begin{equation*} 
\resRKexplicit^\timestepit _i\left(\unknown_i^\timestepit \right) = 0,\quad i\innatsequence{s}
\end{equation*} 
with
\begin{equation}\label{eq:resRKexplicitDef}
\resRKexplicit^\timestepit _i(\unknown)\defeq\unknown-\f(\state^{\timestepit -1}\reviewerA{+\dt a_{ii}\unknown} + \dt\sum_{j=1}^\reviewerA{{i-1}}a_{ij}\unknown^\timestepit_j,t^{\timestepit -1}+c_i\dt),\quad
i\innatsequence{s}.
\end{equation}
We can then formulate the following sequence of optimization problems to
compute discrete \reviewerC{minimum-residual} approximations:
 \begin{equation} \label{eq:galOptRKexplicit}
\tilde \unknown^\timestepit_i = \arg\min_{\optVec\in \range{\podstate}}\|
\weightingMatrix_i(\optVec)\resRKexplicit_i^\timestepit (\optVec)
\|_2^2,\quad
i\innatsequence{s},
  \end{equation} 
 or equivalently
 \begin{equation} \label{eq:galOptRKexplicit2}
\hat \unknown^\timestepit_i = \arg\min_{\optVecRed\in \RR{\nstate}}\|
\weightingMatrix_i(\podstate\optVecRed)\resRKexplicit_i^\timestepit (\podstate\optVecRed)
\|_2^2,\quad
i\innatsequence{s}.
  \end{equation} 
	Here, the associated Petrov--Galerkin projection is 
\begin{equation} \label{eq:LSPGRKexplicit}
\testBasis_{i}^\timestepit(\unknownRed_i^\timestepit)^T\resRKexplicit_i^\timestepit
(\podstate\unknownRed_i^\timestepit) = 0,\quad i\innatsequence{s},
\end{equation} 
with test-basis entries of 
\begin{equation} \label{eq:testBasisExplicitRK}
\testBasisiEntry{j}{k}(\reviewerC{\unknownRed}) =
[\weightingMatrix_i]_{uj}(\podstate\unknownRed)\frac{\partial
[\weightingMatrix_i]_{u\ell}(\podstate\unknownRed)}{\partial
\unknownEntry{m}}\podstateEntry{m}{k}\resRKexplicitiEntry{\ell}(\podstate\unknownRed)+
[\weightingMatrix_i]_{uj}(\podstate\unknownRed)[\weightingMatrix_i]_{um}(\podstate\unknownRed)\reviewerA{\frac{\partial [\resRKexplicit^\timestepit_i]_m(\podstate\unknownRed)}{\partial \unknownEntry{\ell}}\podstateEntry{\ell}{k}},
\end{equation} 
\reviewerA{Note that
in the explicit case,  the \reviewerA{LSPG} ROM generally requires solving a system of
nonlinear equations at each Runge--Kutta stage. Because 
$ {\partial
\resRKexplicit^\timestepit_i}/\partial \unknown= \matI$,
the system of
equations is linear if $\weightingMatrix_i$ are constant
matrices, and only an explicit solution update is required if
$\weightingMatrix_i= \matI$ and $\podstate^T\podstate =
\matI$.}
 
\begin{remark}[\reviewerANew{LSPG ROM enrichment yields a monotonic decrease in the FOM
O$\Delta$E residual}]\label{rem:discreteOptAPriori} \rm
Due to optimality property \eqref{eq:discreteOptLinear} \reviewerANew{of the
LSPG ROM,} 
adding vectors to the trial basis---which \reviewerCNew{enriches} the trial subspace
$\range{\podstate}$---results in a monotonic decrease in the 
\reviewerC{minimum-residual} objective function in problem
\eqref{eq:discreteOptLinear}\reviewerANew{, which is simply the weighted $\ell^2$-norm
of the FOM O$\Delta$E residual associated with linear multistep schemes.}
This result also holds for \reviewerA{LSPG} ROMs applied to Runge--Kutta
schemes, as the computed solutions satisfy alternative optimality properties
in the implicit \eqref{eq:discreteOptRK} and explicit/\reviewerCNew{diagonally implicit}
\eqref{eq:galOptRKexplicit} cases.
\end{remark}

\reviewerA{
We now see the critical tradeoff between the Galerkin and LSPG ROMs: Galerkin
ROMs exhibit continuous (minimum-residual) optimality, while LSPG ROMs exhibit
discrete optimality.
Without further analysis, it is unclear which of these attributes is
preferable. Numerical experiments (Section \ref{s:numerics}) and supporting
error analysis (Section \ref{sec:error}) will highlight the benefits of
discrete optimality over continuous optimality in practice.}

\subsection{Continuous representation}\label{sec:LSPGcontinuous}
Because the \reviewerA{LSPG} ROM introduces approximations at the discrete
level, it is unclear whether it can be interpreted at the continuous level.
\reviewerA{In fact, it has not previously been shown that a continuous
representation of the \reviewerA{LSPG} ROM even exists.} We now show that an ODE
representation of the \reviewerA{LSPG} ROM \reviewerA{does indeed exist} for
both linear multistep schemes and Runge--Kutta schemes under certain
conditions; however, the ODE depends on the time step used to define the
\reviewerA{LSPG} ROM. This associates with the top-left section of the
relationship diagram in Figure \ref{fig:flowchart}.
		\begin{theorem}[\reviewerA{LSPG} ROM continuous representation: linear
		multistep schemes]\label{thm:discOptROMcontDiscEqual}
		\ \\ The \reviewerA{LSPG} ROM for linear multistep integrators is
		equivalent to applying a Petrov--Galerkin projection to the ODE with test basis (in
		matrix form) 
		\begin{equation}\label{eq:testBasisDiscOptLinMulti}
		\testBasis(\stateRed,t) =
		\weightingMatrix^T\weightingMatrix\left(\alpha_0\matI - \Delta
		t\beta_0\frac{\partial\f}{\partial \statevarf}(\stateInitialNo +
		\podstate\stateRed, t)\right)\podstate
		\end{equation}
 and subsequently applying time integration with a linear multistep scheme
 with time step $\dt$ if $\weightingMatrix$ is a constant matrix and (at
 least) one of the following conditions holds:
 \begin{enumerate} 
  \item $\beta_j=0$, $j\geq 1$ (e.g., a single-step method), 
  \item  the velocity $\f$ is linear in the state, or
  \item  $\beta_0= 0$ (i.e., explicit schemes).
	 \end{enumerate}
\end{theorem}	 
		\begin{proof}
		Applying Petrov--Galerkin projection to the full-order model ODE~\eqref{eq:ODE} using a trial
		subspace $\stateInitialNo + \range{\podstate}$ and test subspace
		$\range{\testBasis}$ yields the following ODE
	\begin{align} \label{eq:discreteOptODE}
	\testBasis(\stateRed,t)^T\podstate\frac{d\stateRed}{dt}  &=
	\testBasis(\stateRed,t)^T{\f }(\stateInitialNo
	+ \podstate\stateRed,t) , \qquad 
	\stateRed(0) = 0,
	\end{align} 
	which can be written in standard form as
	\begin{align} \label{eq:discreteOptODEstandard}
	\frac{d\stateRed}{dt}  &=
	\left(\testBasis(\stateRed,t)^T\podstate\right)^{-1}\testBasis(\stateRed,t)^T{\f }(\stateInitialNo
	+ \podstate\stateRed,t) , \qquad 
	\stateRed(0) = 0.
	\end{align} 
		\underline{Case 1}
	Applying a linear multistep time integrator with the stated assumption of $\beta_j=0$, $j\geq 1$ to
	numerically solve
	Eq.~\eqref{eq:discreteOptODEstandard} results in the following discrete
	equations to be solved at each time instance:
 \begin{equation} \label{eq:romResLinMultiFirststep01}
  \alpha_0\unknownRedPG^n - \dt \beta_0
 \left(\testBasis(\unknownRedPG^n,t^\timestepit)^T\podstate\right)^{-1}\testBasis(\unknownRedPG^n,t^\timestepit)^T
 \f (\stateInitialNo + \podstate\unknownRedPG^n,t^\timestepit )+
 \sum_{j=1}^k\alpha_j\stateRed^{\timestepit -j}=0.
  \end{equation} 
Pre-multiplying by $\testBasis(\unknownRedPG^n,t^\timestepit)^T\podstate$ yields 
discrete equations
 $\RRed^\timestepit \left(\unknownRedPG^\timestepit \right) = 0$
with residual
 \begin{equation} \label{eq:romResLinMultiFirststep1}
 \RRed^\timestepit \left(\unknownRed\right) \defeq
 \alpha_0\testBasis(\unknownRed,t^\timestepit)^T\podstate\unknownRed - \dt \beta_0
 \testBasis(\unknownRed,t^\timestepit)^T \f (\stateInitialNo + \podstate\unknownRed,t^\timestepit )+
 \sum_{j=1}^k\alpha_j\testBasis(\unknownRed,t^\timestepit)^T\podstate\stateRed^{\timestepit
 -j}.
  \end{equation} 
	Comparing Eqs.~\eqref{eq:romResLinMultiFirststep1} and \eqref{eq:resLinMulti}
	reveals 
$\RRed^\timestepit \left(\unknownRed\right) =
\testBasis(\unknownRed,t^\timestepit)^T\res^\timestepit\left(\stateInitialNo + \podstate
\unknownRed\right)$ and so the solution $\unknownRedPG^\timestepit$ satisfies
\begin{equation} \label{eq:LSPGlinmultiproof1continuous}
\testBasis(\unknownRedPG^\timestepit,t^\timestepit)^T\res^\timestepit\left(\stateInitialNo + \podstate
\unknownRedPG^\timestepit\right) = 0.
\end{equation}

Under the stated assumptions, we have $\partial
\res^n/\partial\unknown(\state)= \alpha_0\matI - \Delta
t\beta_0\frac{\partial\f}{\partial \statevarf}(\state, t^\timestepit)$
and so the \reviewerA{LSPG} test basis $\testBasis^\timestepit$ defined in
Eq.~\eqref{eq:LSPGlinmultiDef} is equal to the test basis in
Eq.~\eqref{eq:testBasisDiscOptLinMulti} evaluated at time instance $n$, i.e.,
$\testBasis^\timestepit(\unknownRed) = \testBasis(\unknownRed,t^\timestepit)$.
Therefore, the solution $\unknownRed^\timestepit$ to the \reviewerA{LSPG}
O$\Delta$E \eqref{eq:LSPGlinmulti} satisfies
\begin{equation} \label{eq:LSPGlinmultiproof1}
\testBasis(\unknownRed^\timestepit,t^\timestepit)^T\res^\timestepit \left(\stateInitialNo +
\podstate\unknownRed^\timestepit\right) = 0.
\end{equation} 
This shows that $\unknownRed^\timestepit = \unknownRedPG^\timestepit$, i.e., the
solutions to the \reviewerA{LSPG} O$\Delta$E and the \reviewerA{O$\Delta$E}
obtained after applying Petrov--Galerkin projection with test basis
$\testBasis(\state,t)$ defined by Eq.~\eqref{eq:testBasisDiscOptLinMulti} to the full-order model ODE and subsequently applying
time integration
are equivalent under the stated
assumptions, which is the desired result.\\
\noindent\underline{Case 2}
In this case, the test basis is independent of the state, i.e., 
 \begin{equation} \label{eq:testBasisConst}
 \testBasis(t) = \weightingMatrix^T\weightingMatrix\left(\alpha_0\matI -
 \dt\beta_0\frac{\partial \f}{\partial \statevarf}(\cdot,t)\right)\podstate.
  \end{equation} 
	Applying a linear multistep time integrator to solve
	Eq.~\eqref{eq:discreteOptODEstandard} and subsequently pre-multiplying by
	the constant matrix $\testBasis(t^\timestepit)^T\podstate$ yields the following
	discrete equations arising at each time step
 \begin{equation} \label{eq:romDiscreteOptLinMultiSolve}
 \RRed^\timestepit \left(\unknownRedPG^\timestepit \right) = 0,
  \end{equation} 
where the residual is defined as
 \begin{align} \label{eq:romResLinMulti2}
 \begin{split}
 \RRed^\timestepit \left(\unknownRed\right) \defeq
 &\alpha_0\testBasis(t^\timestepit)^T\podstate\unknownRed - \dt \beta_0
 \testBasis(t^\timestepit)^T \f (\stateInitialNo +
 \podstate\unknownRed,t^\timestepit )+
 \sum_{j=1}^k\alpha_j\testBasis(t^\timestepit)^T\podstate\stateRed^{\timestepit
 -j} -\\
 &\dt \sum_{j=1}^k\beta_j \testBasis(t^\timestepit)^T \f \left(\stateInitialNo +
 \podstate\stateRed^{\timestepit -j},t^{\timestepit -j}\right).
  \end{split} 
  \end{align} 
	Comparing Eqs.~\eqref{eq:romResLinMulti2} and \eqref{eq:resLinMulti}
	reveals 
$\RRed^\timestepit \left(\unknownRed\right) =
\testBasis(t^\timestepit)^T\res^\timestepit\left(\stateInitialNo+\podstate
\unknownRed\right)
  $ and so the solution $\unknownRedPG^\timestepit$ satisfies
 \begin{equation} \label{eq:LSPGlinmultiproof2continuous}
\testBasis(t^\timestepit)^T\res^\timestepit\left(\stateInitialNo+\podstate
\unknownRedPG^\timestepit\right) = 0.
  \end{equation} 
Under these assumptions, we have $\partial\res^\timestepit /\partial\unknown =
\alpha_0\matI -
	\dt\beta_0\partial\f/\partial\statevarf(\cdot,t^\timestepit )$
and so the \reviewerA{LSPG} test basis $\testBasis^\timestepit$ defined in
Eq.~\eqref{eq:LSPGlinmultiDef} is equal to the test basis in
Eq.~\eqref{eq:testBasisConst} at time instance $n$, i.e.,
$\testBasis^\timestepit(\unknownRed) = \testBasis(t^\timestepit)$.
Therefore, the \reviewerA{LSPG} O$\Delta$E \eqref{eq:LSPGlinmulti} can be
expressed as
\begin{equation} \label{eq:LSPGlinmultiproof2}
\testBasis(t^\timestepit)^T\res^\timestepit \left(\stateInitialNo +
\podstate\unknownRed^\timestepit\right) = 0.
\end{equation} 
This shows that $\unknownRed^\timestepit = \unknownRedPG^\timestepit$, i.e.,
the solutions to the \reviewerA{LSPG} O$\Delta$E
and the \reviewerA{O$\Delta$E}
obtained after applying Petrov--Galerkin projection with test basis
$\testBasis(t)$ defined by Eq.~\eqref{eq:testBasisDiscOptLinMulti} to the
full-order model ODE and subsequently applying time integration are equivalent
under the stated assumptions.\\
\noindent\underline{Case 3}
	The assumption $\beta_0=0$ results in a constant test basis
 \begin{equation} \label{eq:testBasis3}
 \testBasis = \ \alpha_0\weightingMatrix^T\weightingMatrix\podstate.
  \end{equation} 
	Applying a linear multistep time integrator to solve
	Eq.~\eqref{eq:discreteOptODEstandard} and subsequently pre-multiplying by
	the constant matrix $\testBasis^T\podstate$ yields
 \begin{equation} \label{eq:case3SolPG}
 \RRed^\timestepit \left(\unknownRedPG^\timestepit \right) = 0,
  \end{equation} 
which is
	to be solved at 
	each time step with a residual defined as
 \begin{equation} \label{eq:romResLinMulti3}
 \RRed^\timestepit \left(\unknownRed\right) \defeq
 \alpha_0\testBasis^T\podstate\unknownRed - \dt \beta_0
 \testBasis^T \f (\stateInitialNo +
 \podstate\unknownRed,t^\timestepit )+
 \sum_{j=1}^k\alpha_j\testBasis^T\podstate\stateRed^{\timestepit
 -j} -
 \dt \sum_{j=1}^k\beta_j \testBasis^T \f \left(\stateInitialNo +
 \podstate\stateRed^{\timestepit -j},t^{\timestepit -j}\right).
  \end{equation}
	As in Case 2, this leads to $\RRed^\timestepit \left(\unknownRed\right) =
\testBasis^T\res^\timestepit\left(\stateInitialNo+\podstate
\unknownRed\right)$. Because $\frac{\partial \res^n}{\partial\unknown}(\state)=
	\alpha_0\matI $, we also again have $\testBasis^\timestepit(\unknownRed) =
	\testBasis$. This leads to the desired result, as the 
O$\Delta$Es for the \reviewerA{LSPG} ROM
and the ROM obtained after applying
Petrov--Galerkin projection with test basis $\testBasis$ to the full-order
model ODE and subsequently applying time integration
both satisfy $\testBasis^T\res^\timestepit(\stateInitialNo +
\podstate\unknownRed^\timestepit)=0$ under the stated assumptions.
\end{proof}

We now provide conditions under which the \reviewerA{LSPG} ROM for
Runge--Kutta schemes can be expressed as an ODE.

\reviewerA{\begin{theorem}[\reviewerA{LSPG} ROM continuous representation: Runge--Kutta schemes]\label{thm:discOptROMcontDiscEqualRK}
The \reviewerA{LSPG} ROM for Runge--Kutta integrators is
equivalent to applying a Petrov--Galerkin projection to the ODE with test basis (in
matrix form) 
\begin{equation}\label{eq:testBasisDiscreteOptRK}
\testBasis(\stateRed,t) =
\weightingMatrix^T\weightingMatrix\left(\matI - \dt a\frac{\partial\f}{\partial \statevarf}(\stateInitialNo +
\podstate\stateRed, t
		)\right)\podstate
\end{equation}	
		and
		subsequently applying time integration if
		$\weightingMatrix_i=\weightingMatrix$, $\forall i$ are constant
		matrices and 
the integrator is a \textit{singly diagonally implicit Runge--Kutta (SDIRK) scheme, i.e., $a_{ij}=0$, $\forall j>i$} and $a_{ii}=a$, $\forall
i$.\footnote{\reviewerA{Note that summation on repeated indicies
is not implied.}}
\end{theorem}}
\begin{proof}
\reviewerA{Applying Petrov--Galerkin projection to Eq.~\eqref{eq:ODE} using a trial
subspace $\stateInitialNo + \range{\podstate}$ and test subspace
$\range{\testBasis}$ yields the following ODE (in standard form)
	\begin{align} \label{eq:discreteOptODEstandard2}
	\frac{d\stateRed}{dt}  &=
	\left(\testBasis(\stateRed,t)^T\podstate\right)^{-1}\testBasis(\stateRed,t)^T{\f }(\stateInitialNo
	+ \podstate\stateRed,t) , \qquad 
	\stateRed(0) = 0.
	\end{align} 
	Applying an SDIRK time integrator to numerically solve
	Eq.~\eqref{eq:discreteOptODEstandard2} results in the following sequence of discrete
	equations to be solved at each time step:
\begin{align} \label{eq:romResRKPG}
\begin{split}
 \vec
\unknownRedPG_i^n -
 &[(\testBasis(\stateRed^{\timestepit -1} + \dt
 \sum_{j=1}^ia_{ij}\vec \unknownRedPG_j^n, t^{\timestepit -1} +
c_i\dt)^T\podstate]^{-1}\testBasis(\stateRed^{\timestepit -1} + \dt
 \sum_{j=1}^ia_{ij}\vec \unknownRedPG_j^n, t^{\timestepit -1} +
c_i\dt)^T
\\
&\f (\stateInitialNo + \podstate\stateRed^{\timestepit -1} + \dt
 \sum_{j=1}^ia_{ij}\vec \podstate\unknownRedPG_j^n, t^{\timestepit -1} +
c_i\dt)=0,\quad i\innatsequence{s}.
\end{split} 
\end{align} 
Pre-multiplying by $\testBasis(\stateRed^{\timestepit -1} + \dt
 \sum_{j=1}^ia_{ij}\vec \unknownRedPG_j, t^{\timestepit -1} +
c_i\dt)^T\podstate$ yields the following
discrete equations
\begin{equation*}\label{eq:romDiscreteOptRKSolveFirst}
\resRKexplicitRed^\timestepit _i\left(\unknownRedPG_i^\timestepit \right) = 0,\quad i\innatsequence{s}
\end{equation*} 
with residual
\begin{align} \label{eq:romResRKPG2}
\begin{split}
\resRKexplicitRed^\timestepit_i(\unknownRed) \defeq  &
\testBasis(\stateRed^{\timestepit -1} + \dt a\unknownRed + \dt
 \sum_{j=1}^{i-1}a_{ij}\unknownRedPG_j^\timestepit, t^{\timestepit -1} +
c_i\dt)^T\cdot\\
&\left[\podstate\unknownRed -
\f (\stateInitialNo + \podstate\stateRed^{\timestepit -1} + \dt a \unknownRed +\dt
 \sum_{j=1}^{i-1}a_{ij}\vec \podstate\unknownRedPG_j^\timestepit, t^{\timestepit -1} +
c_i\dt)\right].
\end{split} 
\end{align} 
Comparing Eqs.~\eqref{eq:romResRKPG2} and \eqref{eq:resRKexplicitDef} reveals 
\begin{equation*}
\resRKexplicitRed^\timestepit_i(\unknownRed) =
\testBasis(\stateRed^{\timestepit -1} + \dt a\unknownRed + \dt
 \sum_{j=1}^{i-1}a_{ij}\unknownRedPG_j^\timestepit, t^{\timestepit -1} +
c_i\dt)^T
\resRKexplicit^\timestepit_i\left(\podstate\unknownRed\right),\quad
i\innatsequence{s}
\end{equation*}
such that the solutions
$\unknownRedPG_i^\timestepit$ satisfy
\begin{equation}\label{eq:RKPGequations}
\testBasis(\stateRed^{\timestepit -1} + \dt
 \sum_{j=1}^{i}a_{ij} \unknownRedPG^\timestepit_j, t^{\timestepit -1} +
c_i\dt)^T
\resRKexplicit^\timestepit_i\left(\unknownRedPG^\timestepit_i\right)
= 0,\quad i\innatsequence{s}.
\end{equation}}

\reviewerA{Now, under the stated assumptions, we have 
$$\frac{\partial
\resRKexplicit^\timestepit_i}{\partial\unknown}(\unknownVariable) = 
\matI - \dt a\frac{\partial\f}{\partial \statevarf}(\state^{\timestepit
-1} + \dt a\unknownVariable + \dt \sum\limits_{j=1}^{i-1}a_{ij}\unknown_j^\timestepit, t^{\timestepit -1} +
c_i\dt)$$
such that the \reviewerA{LSPG} test basis $\testBasis_{i}^\timestepit$
defined in Eq.~\eqref{eq:testBasisExplicitRK} is related to the test basis in
Eq.~\eqref{eq:testBasisDiscreteOptRK} as follows:
$$\testBasis_{i}^\timestepit(\unknownRed) = 
\testBasis(\stateRed^{\timestepit -1} +\dt a\unknownRed  + \dt
 \sum_{j=1}^{i-1}a_{ij}\vec \unknownRed_j^\timestepit, t^{\timestepit -1} +
c_i\dt).
$$
Therefore, the solutions
$\unknownRed^\timestepit_i$ to the
\reviewerA{LSPG} O$\Delta$E \eqref{eq:LSPGRKexplicit} satisfy
\begin{equation} \label{eq:LSPGRK2}
\testBasis(\stateRed^{\timestepit -1} + \dt
 \sum_{j=1}^ia_{ij}\vec \unknownRed^\timestepit_j, t^{\timestepit -1} +
c_i\dt)^T
\resRKexplicit^\timestepit_i\left(\unknownRed^\timestepit_i\right)
= 0,\quad i\innatsequence{s}.
\end{equation}
Comparing Eqs.~\eqref{eq:RKPGequations} and \eqref{eq:LSPGRK2} shows that the
$\unknownRed_i^\timestepit = \unknownRedPG_i^\timestepit$, $i\in\nat{s}$, i.e.,
the
solutions to the \reviewerA{LSPG} O$\Delta$E and the O$\Delta$E
obtained after applying Petrov--Galerkin projection with test basis
$\testBasis(\state,t)$ defined by Eq.~\eqref{eq:testBasisDiscreteOptRK} to the
full-order model ODE and subsequently applying time integration are equivalent
under the stated assumptions, which is the desired result.}
\end{proof}

We now show that the \reviewerA{LSPG} ROM has a time-continuous representation
for all \reviewerA{explicit and} single-state Runge--Kutta schemes, \reviewerA{which include the widely
	used forward Euler, backward Euler, and implicit midpoint schemes.}
\reviewerA{\begin{corollary}[\reviewerA{LSPG} ROM continuous representation: explicit Runge--Kutta]
The \reviewerA{LSPG} ROM for Runge--Kutta integrators is
equivalent to applying a Petrov--Galerkin projection to the ODE with test basis (in
matrix form) $$\testBasis(\stateRed,t) =
\weightingMatrix^T\weightingMatrix\podstate$$ and
		subsequently applying time integration if
		$\weightingMatrix_i=\weightingMatrix$, $\forall i$ are constant
		matrices and an explicit Runge--Kutta scheme is employed.
\end{corollary}		
\begin{proof}
Explicit Runge--Kutta schemes are characterized by $a_{ij} = 0$, $j \geq i$ and
so they satisfy the conditions Theorem
\ref{thm:discOptROMcontDiscEqualRK} with $a = 0$.
\end{proof}
}
\begin{corollary}[\reviewerA{LSPG} ROM continuous representation: single-stage Runge--Kutta]
The \reviewerA{LSPG} ROM for \reviewerC{Runge--Kutta integrators} is
equivalent to applying a Petrov--Galerkin projection to the ODE with test basis (in
matrix form) $$\testBasis(\stateRed,t) =
\weightingMatrix^T\weightingMatrix\left(\matI - \dt a_{11}\frac{\partial\f}{\partial \statevarf}(\stateInitialNo +
\podstate\stateRed, t
		)\right)\podstate$$ and
		subsequently applying time integration if
		$\weightingMatrix_i=\weightingMatrix$, $\forall i$ are constant
		matrices and a single-stage Runge--Kutta scheme is employed.
\end{corollary}		
\begin{proof}
Single-stage Runge--Kutta schemes are characterized by $s=1$ and so they
satisfy the conditions of \reviewerA{Theorem \ref{thm:discOptROMcontDiscEqualRK} with $a=a_{11}$}.
\end{proof}

\section{Equivalence conditions}\label{sec:comparative}
This section performs theoretical analysis that highlights cases in which
Galerkin and \reviewerA{LSPG} ROMs are equivalent\reviewerA{, in which case
the Galerkin ROM exhibits 
\textit{both}
continuous and discrete optimality. This suggests that \reviewerCNewer{time
discretization and minimum-residual projection are commutative in these scenarios.}} 
Section \ref{sec:equalExplicit} shows that equivalence holds for explicit time
integrators, Section \ref{sec:equalDt} demonstrates equivalence in the limit
of $\dt\rightarrow 0$, and Section \ref{sec:equalSPD} shows equivalence in the
case of symmetric-positive-definite residual Jacobians.
\subsection{Equivalence for explicit integrators}\label{sec:equalExplicit}

\begin{corollary}[\reviewerC{Equivalence}: explicit linear multistep
scheme]\label{cor:GalDiscOptExpLin} 
\ \\ Galerkin projection is \reviewerA{equivalent to LSPG projection with
$\weightingMatrix=\frac{1}{\sqrt{\alpha_0}}\matI$} for explicit linear multistep schemes.
\end{corollary}
\begin{proof}
In the case of explicit linear multistep schemes, $\beta_0=0$ and so Galerkin
projection corresponds to Case 3 of Theorem \ref{thm:discOptROMcontDiscEqual}
with $\weightingMatrix=\frac{1}{\sqrt{\alpha_0}}\matI$, as $\testBasis =
\podstate$ in this case.
\end{proof}

\begin{corollary}[\reviewerC{Equivalence}: explicit Runge--Kutta
scheme]\label{cor:GalDiscOptExpRK} 
Galerkin projection is \reviewerA{equivalent to LSPG projection with
$\weightingMatrix=\matI$} for explicit Runge--Kutta
schemes.
\end{corollary}
\begin{proof}
In the case of explicit Runge--Kutta schemes, \reviewerC{$a_{ij}=0$ $\forall
j\geq i$} and so Galerkin
projection corresponds \reviewerA{to Theorem} \ref{thm:discOptROMcontDiscEqualRK}
with $\weightingMatrix = \matI$ \reviewerC{and $a = 0$}, as $\testBasis =
\podstate$ in this case.
\end{proof}

\subsection{Equivalence in the limit of $\dt\rightarrow 0$}\label{sec:equalDt}

\begin{theorem}[\reviewerC{Limiting equivalence}]\label{thm:dtZeroEquiv}
\ \\ In the limit of $\dt \rightarrow 0$, Galerkin 
\reviewerA{projection is equivalent to LSPG projection with
$\weightingMatrix=\frac{1}{\sqrt{\alpha_0}}\matI$ for linear multistep
schemes and $\weightingMatrix_i=\reviewerCNewer{\matI}$, $i\innat{s}$ for Runge--Kutta schemes}.
\end{theorem}
\begin{proof}
\underline{Linear multistep schemes}. Consider solving
the \reviewerA{LSPG} O$\Delta$E \eqref{eq:LSPGlinmulti}
with $\weightingMatrix =
\frac{1}{\sqrt{\alpha_0}}\matI$.
Then, the test basis defined in Eq.~\eqref{eq:LSPGlinmultiDef}
is simply 
\begin{equation*}
\testBasis^\timestepit(\unknownRed) = \frac{1}{\alpha_0}\frac{\partial \res^\timestepit}{\partial
\unknown}\left(\stateInitialNo+\podstate\unknownRed\right)\podstate.
\end{equation*}
From Eq.~\eqref{eq:resLinMulti}, we can write the residual Jacobian as
\begin{equation*} 
\frac{\partial \res^\timestepit}{\partial \unknown}(\unknownVariable) =
\alpha_0\matI - \dt \beta_0\frac{\partial\f}{\partial
\statevarf}(\unknownVariable,t^\timestepit).
\end{equation*} 
Therefore, we have
 \begin{equation*} 
 \lim_{\dt\rightarrow 0} \testBasis^\timestepit(\unknownRed)=
 \lim_{\dt\rightarrow 0} \frac{1}{\alpha_0}\left(\alpha_0\matI - \dt
 \beta_0\frac{\partial\f}{\partial
 \statevarf}(\stateInitialNo+\podstate\unknownRed,t^\timestepit)\right)
\podstate=\podstate
  \end{equation*} 
	and so in the limit of $\dt\rightarrow 0$, the \reviewerA{LSPG} ROM solution
	satisfies
\begin{equation} \label{eq:limitingCaseLinear}
\lim_{\dt\rightarrow 0}
\testBasis^\timestepit(\unknownRed)^T
\res^\timestepit
\left(\stateInitialNo+\podstate\unknownRed^\timestepit\right)
=
\podstate^T
\res^\timestepit
\left(\stateInitialNo+\podstate\unknownRed^\timestepit\right)  = 0.
\end{equation} 
Because the Galerkin ROM solution also satisfies Eq.~\eqref{eq:limitingCaseLinear}
(see Eq.~\eqref{eq:reducedRedGalLin} of Theorem
\ref{thm:galProjDiscCommutative}), the two techniques are equivalent in this limit, which is the desired
result.\\ 
\noindent \underline{Runge--Kutta schemes}.
Consider solving
the \reviewerA{LSPG} O$\Delta$E \eqref{eq:LSPGRK}
with $\weightingMatrix_i = \matI$, $i\in\nat{s}$.
Then, the test basis defined in Eq.~\eqref{eq:discreteOptRKTest}
is simply 
 \begin{equation*} 
 \testBasis^\timestepit_{ij}(\unknownRed_1,\ldots,\unknownRed_s) = \frac{\partial \res^\timestepit_i}{\partial
 \unknown_j}(\podstate\unknownRed_1,\ldots,\podstate\unknownRed_s)\podstate.
  \end{equation*} 
Now, from Eq.~\eqref{eq:resRK} the Jacobian can be expressed as
 \begin{equation*} 
\frac{\partial
\res_i^n}{\partial\unknown_j}(\unknownVariable_1,\ldots, \unknownVariable_s) = 
\matI\delta_{ij} - \dt a_{ij}\frac{\partial\f}{\partial \statevarf}(\state^{\timestepit
-1} + \dt \sum_{j=1}^sa_{ij}\unknownVariable_j, t^{\timestepit -1} +
c_i\dt).
  \end{equation*} 
	Therefore, we have
 \begin{equation*} 
 \lim_{\dt\rightarrow 0} \testBasis^\timestepit_{ij}(\unknownRed_1,\ldots,\unknownRed_s)=
 \lim_{\dt\rightarrow 0} \left(\matI\delta_{ij} - \dt a_{ij}\frac{\partial\f}{\partial \statevarf}(\state^{\timestepit
-1} + \dt \sum_{j=1}^sa_{ij}\unknownVariable_j, t^{\timestepit -1} +
c_i\dt)\right)
\podstate=\podstate\delta_{ij}
  \end{equation*} 
	and so in the limit of $\dt\rightarrow 0$, the \reviewerA{LSPG} ROM solution
	satisfies
\begin{equation} \label{eq:limitingCaseRK}
\lim_{\dt\rightarrow 0}
\sum_{j=1}^s\testBasis^\timestepit_{ij}(\unknownRed_1,\ldots,\unknownRed_s)^T
\res_j^\timestepit
\left(\podstate\unknownRed_1^\timestepit,\ldots,\podstate\unknownRed_s^\timestepit\right)
=
\podstate^T
\res_i^\timestepit
\left(\stateInitialNo+\podstate\unknownRed^\timestepit\right)  = 0,\quad
i\in\nat{s}.
\end{equation} 
Because the Galerkin ROM solution also satisfies Eq.~\eqref{eq:limitingCaseRK}
(see Eq.~\eqref{eq:reducedRedGalRK} of Theorem
\ref{thm:galProjDiscCommutative}), the two techniques are equivalent in this limit, which is the desired
result.
\end{proof}
\subsection{Equivalence for symmetric-positive-definite residual
Jacobians}\label{sec:equalSPD}
\begin{theorem}[\reviewerC{Equivalence}: linear multistep
schemes]\label{thm:galDiscOptLinear}
In the case of linear multistep schemes, Galerkin projection \reviewerA{is equivalent to LSPG projection}
with $\weightingMatrix\left(\optVec\right) = \rightChol\left(\optVec\right)$,
where $\rightChol$ is the Cholesky
factor\footnote{
Its derivative can be computed by solving the Lyapunov equation
$
\frac{\partial \rightChol}{\partial \unknownEntry{k}}^T\rightChol +
\rightChol\frac{\partial\rightChol}{\partial\unknownEntry{k}}= - \left[\frac{\partial
\res^\timestepit }{\partial\unknown}\right]^{-1}  \frac{\partial^2 \res^\timestepit }{\partial
\unknown\partial \unknownEntry{k}}\left[\frac{\partial \res^\timestepit }{\partial\unknown}\right]^{-1}
$.
	} of the residual-Jacobian inverse
\begin{equation}\label{eq:cholDef}
\left[\frac{\partial \res^\timestepit
}{\partial\unknown}\right]^{-1}=\rightChol^T\rightChol,
\end{equation}
if $\partial \res^\timestepit /\partial
\unknown\left(\unknown^\timestepit ,t^\timestepit \right) = \alpha_0 \matI -
\dt
\beta_0\frac{\partial \f }{\partial \statevarf}\left(\unknown^\timestepit ,t^\timestepit \right)$
is symmetric positive
definite and if
\begin{gather}\label{eq:galDiscOptLinearMultiAssumption}
\frac{\partial\rightCholEntry{i}{\ell}}{\partial\unknownEntry{k}}\podstateEntry{k}{j}\resEntry{\ell}^\timestepit 
=0,\quad\forall i,k.
\end{gather}
Here, index notation has been used. 
\end{theorem}

\begin{proof}
Under the stated assumptions, the \reviewerA{LSPG} test basis defined in
Eq.~\eqref{eq:LSPGlinmultiDef} is equal to the trial basis, i.e.,
$\testBasis^\timestepit(\unknownRed^\timestepit) = \podstate$. By invoking
Eq.~\eqref{eq:reducedRedGalLin}, we can see that the
O$\Delta$Es for the the \reviewerA{LSPG} ROM
\eqref{eq:LSPGlinmulti} and Galerkin ROM \eqref{eq:romResLinMultiSolve} both
satisfy $\vec \Phi^T \res^\timestepit \left(\stateInitialNo +
\podstate\unknownRed^\timestepit\right) = 0$,
which is the
desired result.
\end{proof}

\reviewerC{
\begin{theorem}[\reviewerC{Equivalence}: \reviewerC{diagonally implicit} Runge--Kutta
schemes]
In the case of diagonally implicit Runge--Kutta schemes, Galerkin projection is equivalent to LSPG projection
with $\weightingMatrix_i\left(\optVec\right) = \rightCholDIRK_i\left(\optVec\right)$,
where $\rightCholDIRK_i$ is the Cholesky
factor
of the residual-Jacobian inverse
\begin{equation}\label{eq:cholDef}
\left[\frac{\partial \resRKexplicit^\timestepit_i
}{\partial\unknown}\right]^{-1}=\rightCholDIRK_i^T\rightCholDIRK_i,
\end{equation}
if $\partial \resRKexplicit^\timestepit_i /\partial
\unknown\left(\unknown^\timestepit ,t^\timestepit \right) = \matI -
\dt
a_{ii}\frac{\partial \f }{\partial \statevarf}\left(
\state^{\timestepit -1}+ \dt a_{ii}
\unknown^\timestepit + \dt\sum_{j=1}^{i-1}a_{ij}\unknown^\timestepit_j,t^{\timestepit -1}+c_i\dt \right)$
is symmetric positive
definite and if
\begin{gather}\label{eq:galDiscOptLinearMultiAssumption}
\frac{\partial\rightCholDIRKEntry{j}{\ell}{i}}{\partial\unknownEntry{m}}\podstateEntry{m}{k}\resRKexplicitiEntry{\ell}^\timestepit 
=0,\quad\forall j,k.
\end{gather}
Here, index notation has been used. 
\end{theorem}
\begin{proof}
Under the stated assumptions, the LSPG test basis defined in
Eq.~\eqref{eq:testBasisExplicitRK} is equal to the trial basis, i.e.,
$\testBasis_{i}^\timestepit(\unknownRed_i^\timestepit) = \podstate$,
$i\innat{s}$. By invoking
Eq.~\eqref{eq:reducedRedGalRK}, we can see that the
O$\Delta$Es for the the LSPG ROM
\eqref{eq:LSPGRKexplicit} and Galerkin ROM \eqref{eq:romResRKSolve} both
satisfy $\vec \Phi^T \resRKexplicit_i^\timestepit
\left(\podstate\unknownRed_i^\timestepit\right) = 0$, $i\innatsequence{s}$,
which is the
desired result.
\end{proof}
}

\begin{theorem}[\reviewerC{Equivalence}: \reviewerC{implicit} Runge--Kutta
schemes]\label{thm:galDiscOptRK}
In the case of Runge--Kutta schemes, Galerkin projection 
exhibits discrete optimality if 
$\partial \resRK^\timestepit /\partial \unknownRK\left(\unknownRK^\timestepit ,t^\timestepit \right)$
is symmetric positive
definite and if
\begin{gather}\label{eq:galDiscOptLinearMultiAssumption2}
\frac{\partial\rightCholRKEntry{i}{\ell}}{\partial\unknownRKEntry{k}}\podstateRKEntry{k}{j}\resRKEntry{\ell}^\timestepit 
=0,\quad\forall i,k.  
\end{gather}
Here, index notation has been used and
$\rightCholRK\reviewerC{\in\RR{s\ndof\times s\ndof}}$ is the Cholesky
factor of the residual-Jacobian inverse, i.e.,
\begin{equation}\label{eq:cholDef2}
\left[\frac{\partial
\resRK^\timestepit
}{\partial\unknownRK}\right]^{-1}=\rightCholRK^T\rightCholRK.
\end{equation}
Here,
\begin{equation*}
\unknownRK\defeq\left[
\begin{array}{c}
\unknown_1\\
\vdots\\
\unknown_s
\end{array}
\right]\in\RR{s\ndof},\quad
\resRK^\timestepit :
\unknownRK
\mapsto
\left[
\begin{array}{c}
\res^\timestepit _1\left(\unknown_1,\ldots,\unknown_s\right)\\
\vdots\\
\res^\timestepit _s\left(\unknown_1,\ldots,\unknown_s\right)
\end{array}
\right]\in\RR{s\ndof},
\quad
\quad \podstateRK\defeq\left[\begin{array}{c c c }
\podstate & &\\
&\ddots  &\\
&&\podstate  \\
\end{array}
\right]\in\RR{s\ndof \times s\nstate}. 
\end{equation*}
\end{theorem}

\begin{proof}
First, note that solution
$(\unknownRed^\timestepit_1,\ldots,\unknownRed^\timestepit_s)$ to the Galerkin
O$\Delta$E \eqref{eq:reducedRedGalRK} 
equivalently satisfies
\begin{equation*} 
\podstateRK^T\resRK^\timestepit \left(\podstateRK\unknownRKRed^\timestepit\right)=0,
\end{equation*} 
where 
\reviewerCNewer{
$
\unknownRKRed\defeq\left[
\unknownRed_1^T\
\cdots\
\unknownRed_s^T
\right]^T\in\RR{s\nstate}.
$
}
We are now precisely in the situation of Theorem \ref{thm:galDiscOptLinear}:
the Galerkin solution is the solution to the (discrete) optimization problem
 \begin{equation} \label{eq:galOptRK}
 \underset{\optVec\in
 \range{\reviewerC{\podstateRK}}}{\mathrm{minimize}}\|
\rightCholRK\left(\optVec\right)
\resRK^\timestepit \left(\optVec\right)\|_2^2
  \end{equation} 
under the assumed conditions. \reviewerC{This objective function can be
written equivalently as 
\begin{equation} \label{eq:equivRKSPD}
\|
\rightCholRK\left(\optVec\right)
\resRK^\timestepit \left(\optVec\right)\|_2^2 =
\sum_{i=1}^s\|\sum_{j=1}^s\rightCholRK_{ij}\left(\optVec_1,\ldots,\optVec_s\right)
\res^\timestepit_j(\optVec_1,\ldots,\optVec_s)
\|_2^2,
\end{equation} 
where $\rightCholRK_{ij}\in\RR{\ndof\times\ndof}$ denotes the $(i,j)$ block of $\rightCholRK$.
Therefore, the Galerkin ROM solution satisfies
\begin{equation} \label{eq:galROMOptRK}
(\unknownRed^\timestepit_1,\ldots,\unknownRed^\timestepit_s) = \underset{\left(\optVecRed_1,\ldots,\optVecRed_s\right)\in
 \RR{\nstate\times s}}{\arg\min}\sum_{i=1}^s\|\sum_{j=1}^s
\weightingMatrix_{ij}\left( \podstate\optVecRed_1,\ldots,\podstate\optVecRed_s\right)
\res_j^\timestepit \left(
\podstate\optVecRed_1,\ldots,\podstate\optVecRed_s\right)\|_2^2,
\end{equation} 
where $\weightingMatrix_{ij}=\rightCholRK_{ij}$.
Comparing Eqs.~\eqref{eq:galROMOptRK} 
and \eqref{eq:discreteOptRK2} reveals that the Galerkin ROM satisfies a
slightly more general notion of discrete optimality than the LSPG schemes
considered in this work.
}
\end{proof}


\reviewerA{
This analysis demonstrates that Galerkin projection exhibits discrete
optimality
when the residual Jacobian is symmetric
positive definite.
This is aligned with recent work that has shown
Galerkin projection to be effective for Lagrangian dynamical
systems \cite{lall2003structure,carlbergStructureAiaa,carlberg2012spd}---which
are characterized by symmetric-positive-definite residual
Jacobians---due to the fact that Galerkin projection preserves properties such
as symplectic time evolution and energy conservation. For these reasons, using
Galerkin projection is sensible for problems exhibiting these characteristics.}

\section{Error analysis}\label{sec:error}

\reviewerA{Ultimately, we are interested in assessing the state-space error between the
(computed) time-discrete ROM solution and the (unknown) time-continuous FOM
solution. This error comprises two contributions:
the state-space error between (1) the time-continuous FOM and time-discrete FOM
solutions (i.e., time-discretization error), and (2) the
time-discrete ROM and time-discrete FOM solutions.
}
This section \reviewerA{focuses on the latter and} performs
time-discrete state-space error analyses for Galerkin and \reviewerC{LSPG}
ROMs applied to different time integrators. 

\reviewerA{Section \ref{sec:errorLinear} derives  error bounds for
the Galerkin and LSPG ROMs for linear multistep schemes. Here, Theorem~\ref{eq:error_multi_step}
provides \textit{a posteriori} bounds that depend on the ROM solution, while
Theorem~\ref{thm:error_multi_step_apriori} reports \textit{a priori} bounds.
Section \ref{sec:errorBackEuler} provides \textit{a posteriori} error bounds
for the backward Euler scheme, as well as additional analyses that highlight
	the important role of the time step in the LSPG ROM, which is discussed in
	Remark \ref{rem:modestTimestep}. Section~\ref{sec:errorRK} derives ROM error
	bounds for 
Runge--Kutta schemes. Here, Theorem~\ref{thm:RKerrors} provides \textit{a posteriori}
bounds, Corollary \ref{cor:DIRK_error} specializes these results to explicit Runge--Kutta
and DIRK schemes, and Theorem~\ref{thm:RKerrors_apriori} and Corollary
\ref{cor:RKerrors_apriori_explicit} report \textit{a priori} bounds for
Runge--Kutta schemes.
}

\subsection{Linear multistep schemes}\label{sec:errorLinear}
Here, we perform error analysis for implicit linear multistep
schemes. We will use subscripts $\star$, $G$ and \reviewerA{$P$} to denote the solution to
full-order model O$\Delta$E~\eqref{eq:resLinMultiSolve}, Galerkin ROM
O$\Delta$E~\eqref{eq:romResLinMultiSolve}, and the \reviewerC{LSPG} ROM
O$\Delta$E~\eqref{eq:LSPGlinmulti}, respectively. \reviewerA{We also acknowledge that linear multistep schemes with $k > 1$
usually employ different coefficients $\beta_j$ and $\alpha_j$ for different
time \reviewerCNew{instances}; this is necessary because at time \reviewerCNew{instance} $n$, a maximum of $n+1$
states is available from past history (starting with the initial condition at
$n=0$). Therefore, we allow for coefficients
 that depend on the time \reviewerCNew{instance} $n$, i.e., $\alpha^n_j$ and $\beta^n_j$. In
 addition, we} define
$\testBasis^n\defeq\testBasis^\timestepit(\stateRedP^\timestepit)$ whose
entries are defined by Eq.~\eqref{eq:LSPGlinmultiDef}. 
\reviewerC{We can then write the discrete equations arising at each time
\reviewerCNew{instance}
$n$ for linear multistep schemes as}
\begin{align}
    \alpha^{\reviewerA{n}}_0 \state^{n}_\reviewerA{\star} 
&= \beta^{\reviewerA{n}}_0 \dt   \f \del{\stateInitialNo +
\state_\reviewerA{\star}^{n},t^n} +
\resFOM\sbr{\state_\reviewerA{\star}^{n-k}, \dots,
\state_\reviewerA{\star}^{n-1}} , && \state_\reviewerA{\star}^0 = \zero  \label{eq:fom_multi_step} \\
     \alpha^{\reviewerA{n}}_0 \stateRedG^{n} 
&= \beta^{\reviewerA{n}}_0 \dt  \podstate ^T \f \del{\stateInitialNo + \podstate \stateRedG^{n},t^n} +
\resGal\sbr{\stateRedG^{n-k}, \dots, \stateRedG^{n-1}} , && \stateRedG^0 =
\zero  \label{eq:rom_G_multi_step} \\
    \alpha^{\reviewerA{n}}_0 \stateRedP^{n} 
&= \beta^{\reviewerA{n}}_0 \dt  \del{(\testBasis^n)^T \podstate }^{-1}
(\testBasis^n)^T 
\f\del{\stateInitialNo + \podstate  \stateRedP^{n},t^n} + \resDisc\sbr{\stateRedP^{n-k},
\dots, \stateRedP^{n-1}} , && \stateRedP^0 = \zero
\label{eq:rom_P_multi_step}, 
\end{align}
where 
\begin{align} \label{eq:residuals}
\begin{aligned}
    \resFOM\sbr{\state^{n-k}, \dots, \state^{n-1}} 
&\defeq \sum_{\ell=1}^k \del{ \beta^{\reviewerA{n}}_\ell \dt  \f 
\del{\stateInitialNo + \state^{n-\ell}, t^{n-\ell}} -
\alpha^{\reviewerA{n}}_\ell \state^{n-\ell} }, \\
    \resGal\sbr{\stateRed^{n-k}, \dots, \stateRed^{n-1}}
&\defeq \sum_{\ell=1}^k \del{ \beta^{\reviewerA{n}}_\ell \dt  \podstate ^T \f 
\del{\stateInitialNo + \podstate  \stateRed^{n-\ell}, t^{n-\ell}} -
\alpha^{\reviewerA{n}}_\ell \stateRed^{n-\ell} }, \\
    \resDisc\sbr{\stateRed^{n-k}, \dots, \stateRed^{n-1}}  
&\defeq    \sum_{\ell=1}^k \del{ \beta^{\reviewerA{n}}_\ell \dt
\del{(\testBasis^{n})^T \podstate }^{-1} (\testBasis^{n})^T \f  \del{\stateInitialNo +
\podstate  \stateRed^{n-\ell}, t^{n-\ell}} - \alpha^{\reviewerA{n}}_\ell
\stateRed^{n-\ell} },
\end{aligned}
\end{align}
\reviewerC{and  $\state_\star^{k}\defeq \state^{k} - 
\stateInitialNo$.}
\noindent \reviewerCNewer{We also define the FOM residuals at time instance $n$ associated with the trajectories
associated with the FOM, Galerkin ROM,
and LSPG ROM O$\Delta$Es as
\begin{align}
\resFOMTotal{n}(\state)\defeq\alpha_0^n\state  -
\dt\beta_0^n\f(\stateInitialNo + \state) - \resFOM\sbr{\state_\star^{n-k},
\dots, \state_\star^{n-1}}\\
\resGalTotal{n}(\state)\defeq\alpha_0^n\state  -
\dt\beta_0^n\f(\stateInitialNo + \state) - \resFOM\sbr{\podstate\stateRedG^{n-k},
\dots, \podstate\stateRedG^{n-1}}\\
\resDiscTotal{n}(\state)\defeq\alpha_0^n\state  -
\dt\beta_0^n\f(\stateInitialNo + \state) - \resFOM\sbr{\podstate\stateRedP^{n-k},
\dots, \podstate\stateRedP^{n-1}}.
\end{align}}
We define the Galerkin and \reviewerC{LSPG} operators as
\[
    \bbV \defeq \podstate  \podstate ^T  , 
    \quad 
    \mbox{and}
    \quad
    \bbP^n \defeq \podstate  \del{(\testBasis^n)^T \podstate }^{-1}
		(\testBasis^n)^T  ,
\]
\reviewerC{respectively}, and Galerkin and \reviewerC{LSPG} state-space errors at time
instance $\timestepit$ as
\[
    \deltastateRedG^n \defeq \state_\star ^n - \podstate  \stateRedG^n,
    \quad 
    \mbox{and}
    \quad
    \deltastateRedP^n \defeq \state_\star ^n - \podstate  \stateRedP^n, 
\]
respectively. \reviewerCNewer{Note that the Galerkin operator $\bbV$
is an orthogonal projector ($\podstate$ is assumed to be orthogonal), while the LSPG operator $\bbP$ is an
oblique projector.} As the second argument in $\f$ does not play any role for linear
multistep schemes
\reviewerC{(the time index always matches that of the first argument)}, will
drop it for notational convenience \reviewerC{in this section and in Section
\ref{sec:errorBackEuler}}.

\subsubsection{\reviewerCNewer{\textit{A posteriori} error bounds}}
\reviewerCNewer{We proceed by deriving \textit{a posteriori} error bounds for the Galerkin and
LSPG ROMs for linear multistep schemes.}
\reviewerCNewer{We} assume Lipschitz continuity of $\f$ in the first argument:
\begin{itemize}
\item[$\del{\bf A_1}$] There \reviewerCNew{exists} a constant $\lipschitzConstant >0$ such that for $\vec x, \vec y \in \bbR^N$ 
    \begin{equation*}
        \normtwo{\f (\state) - \f(\vec y)} \le \lipschitzConstant  \normtwo{\state - \vec y}.
    \end{equation*}
\end{itemize}

\reviewerCNewer{
\begin{theorem}[Local {\textit{a posteriori} error bounds: linear multistep
schemes}]\label{theorem:localAPosterioriLMM}
If $\del{\bf A_1}$ holds and \reviewerA{
$\dt < \abs{\alpha^{\reviewerA{j}}_0}/(\abs{\beta^{\reviewerA{j}}_0}
\lipschitzConstant )$}, \reviewerA{$\forall j\innat{n}$}, then
\begin{align}
    \normtwo{\deltastateRedG^{n}} 
&\le \sum_{\ell=0}^k \firstCoeffMultin{\ell} \normtwo{ \del{ \matI  - \bbV }
\f \del{\stateInitialNo + \podstate  \stateRedG^{n-\ell}} }  
    + \sum_{\ell=1}^k \secondCoeffMultin{\ell} \normtwo{\deltastateRedG^{n-\ell}}
		\label{eq:error_FOM_ROM_G_multi_stepfirst2}\\
\normtwo{\deltastateRedP^{n}} 
&\le \sum_{\ell=0}^k \firstCoeffMultin{\ell} \normtwo{ \del{ \matI  - \bbP^{n}
} \f \del{\stateInitialNo + \podstate  \stateRedP^{n-\ell}} }  
    + \sum_{\ell=1}^k \secondCoeffMultin{\ell} \normtwo{\deltastateRedP^{n-\ell}}
		\label{eq:error_FOM_ROM_P_multi_stepfirst2}
\end{align}
 where we have defined $\firstCoeffMulti{\ell}^\reviewerCNewer{m}\defeq
 \abs{\beta^{m}_\ell} \dt
/h^{m}$,
$\secondCoeffMulti{\ell}^{m}\defeq (
\reviewerCNew{\abs{\alpha^{m}_\ell}+\abs{\beta^{m}_\ell} \lipschitzConstant
		\dt} )/h^{m}$, and $h^{m} \defeq
		\abs{\alpha^{m}_0} -
		\abs{\beta^{m}_0} \lipschitzConstant  \dt $.
\end{theorem}
}
\begin{proof}
It is \reviewerCNewer{sufficient} to show bound \eqref{eq:error_FOM_ROM_P_multi_step}, as the
arguments for \eqref{eq:error_FOM_ROM_G_multi_step} are similar. Let $n$ be
fixed but arbitrary, then subtracting Eq.~\eqref{eq:rom_P_multi_step}
from Eq.~\eqref{eq:fom_multi_step} yields
\begin{align}\label{eq:beforeAddSubtract}
		\abs{\alpha^\reviewerA{n}_0} \normtwo{\deltastateRedP^{n}} &\le
		\abs{\beta^\reviewerA{n}_0} \dt
		\normtwo{\f \del{\stateInitialNo + \state_\star ^{n}} - \bbP^n 
		\f\del{\stateInitialNo +
		\podstate  \stateRedP^{n}}} + \normtwo{\deltaresDisc^{n-1}},
\end{align}
where $\deltaresDisc^{n-1} \defeq \resFOM\sbr{\state_\star ^{n-k}, \dots,
\state_\star ^{n-1}} - \podstate  \resDisc\sbr{\stateRedP^{n-k}, \dots,
\stateRedP^{n-1}}$. Adding and subtracting $\f \del{\stateInitialNo + \podstate 
\stateRedP^{n}}$ and applying the triangle inequality leads to 
\begin{align}\label{eq:afterAddSubtract}
	 \abs{\alpha^\reviewerA{n}_0} \normtwo{\deltastateRedP^{n}} &\le
	 \abs{\beta^\reviewerA{n}_0} \dt  \del{ \normtwo{ \del{ \matI  - \bbP^n }
	 \f \del{\stateInitialNo + \podstate  \stateRedP^{n}} } + \normtwo{\f\del{
	 \stateInitialNo + \state_\star ^{n}} - \f \del{\stateInitialNo + \podstate  \stateRedP^{n}}}
	 } + \normtwo{\deltaresDisc^{n-1}}  . 
\end{align}
Invoking $\del{\bf A_1}$, and using \reviewerA{$\dt  <
\abs{\alpha^\reviewerA{n}_0}/\abs{\beta^\reviewerA{n}_0} \lipschitzConstant $}, we deduce 
\begin{align}\label{eq:error_FOM_ROM_P_multi_step_a}
    \normtwo{\deltastateRedP^{n}} 
&\le \frac{\abs{\beta^\reviewerA{n}_0} \dt }{h^\reviewerA{n}} \normtwo{ \del{
	\matI  - \bbP^n } \f \del{\stateInitialNo + \podstate  \stateRedP^{n}} }  
    + \frac{1}{h^\reviewerA{n}} \normtwo{\deltaresDisc^{n-1}} . 
\end{align}
Next, we will estimate $\normtwo{\deltaresDisc^{n-1}}$. Using the definition of
$\resFOM$, $\resDisc$ from \eqref{eq:residuals} we derive 
\begin{align}\label{eq:beforeAddSubtract2}
    \normtwo{\deltaresDisc^{n-1}} 
&\le \sum_{\ell=1}^k \del{ \abs{\beta^\reviewerA{n}_\ell} \dt  \normtwo{
\f\del{\stateInitialNo + \state_\star ^{n-\ell}} - \bbP^{n}
\f\del{\stateInitialNo + \podstate  \stateRedP^{n-\ell}}} 
    + \abs{\alpha^\reviewerA{n}_\ell} \normtwo{\deltastateRedP^{n-\ell}}  }  . 
\end{align}
Adding and subtracting $\f \del{\stateInitialNo + \podstate  \stateRedP^{n-\ell}}$,
applying the triangle inequality in conjunction with $\del{\bf A_1}$ yields 
\begin{align}\label{eq:error_FOM_ROM_P_multi_step_b}
    \normtwo{\deltaresDisc^{n-1}} 
&\le \sum_{\ell=1}^k \abs{\beta^\reviewerA{n}_\ell} \dt  \normtwo{ \del{\matI
- \bbP^{n} } \f \del{\stateInitialNo + \podstate  \stateRedP^{n-\ell}}}  
    + \sum_{\ell=1}^k \del{\abs{\beta^\reviewerA{n}_\ell} \lipschitzConstant  \dt  +
		\abs{\alpha^\reviewerA{n}_\ell}} \normtwo{\deltastateRedP^{n-\ell}}. 
\end{align}
Then \eqref{eq:error_FOM_ROM_P_multi_step_a} and
\eqref{eq:error_FOM_ROM_P_multi_step_b} \reviewerC{imply}
\reviewerA{\reviewerCNewer{Eq.~\eqref{eq:error_FOM_ROM_P_multi_stepfirst2}} where we have used the definitions for $\firstCoeffMultin{\ell}$ and 
$\secondCoeffMultin{\ell}$.
\end{proof}
}

\reviewerCNewer{We now establish a result that aids interpretability, as it
provides a connection between the terms in the \textit{a posteriori} error
bounds and the optimality properties of LSPG and Galerkin ROMs. First, we note
that 
\begin{equation}
\label{eq:galObliqueDiscFirst}\|(\matI - \bbV)\f \del{\stateInitialNo +
\podstate  \stateRedG^{n}, t^{n}}\|_{2} = \min_{\vec
y\in\range{\podstate}}\|\vec y - \f \del{\stateInitialNo +
\podstate  \stateRedG^{n}, t^{n}}\|_2,
\end{equation}
due to the optimality property associated
with orthogonal projection; 
this term appears in local \textit{a posteriori} error bound
\eqref{eq:error_FOM_ROM_G_multi_stepfirst2}.
\begin{lemma}[Oblique projection as discrete-residual
minimization]\label{lemm:obliqueDiscreteResMin}
If 
$\beta_j^n=0$, $j\geq 1$
(e.g., backward differentiation formulas), then
\begin{align} 
\label{eq:galObliqueDisc}\abs{\beta_0^n}\dt\|(\matI - \bbV)\f \del{\stateInitialNo +
\podstate \stateRedG^{n}, t^{n}}\|_{2} &= \|\resGalTotal{\timestepit}(
\podstate \stateRedG^{n})\|_2\\
\label{eq:lspgObliqueDisc}\abs{\beta_0^n}\dt\|(\matI - \bbP^n)\f
\del{\stateInitialNo +
\podstate \stateRedP^{n}, t^{n}}\|_{2} &= \|\resDiscTotal{\timestepit}(
\podstate \stateRedP^n)\|_2..
\end{align} 
If additionally the LSPG ROM employs $\weightingMatrix = \matI$, then
\begin{align} 
\label{eq:lspgObliqueDisc2}\|\resDiscTotal{\timestepit}(
\podstate\stateRedP^n)\|_2= \min_{\vec
y\in\range{\podstate}}\|\resDiscTotal{\timestepit}(\vec y)\|_2.
\end{align} 
\end{lemma}
\begin{proof}
From Eq.~\eqref{eq:rom_P_multi_step}, we have
\begin{equation} 
\beta_0^n\dt\bbP^n\f \del{\stateInitialNo + \podstate \stateRedP^{n}, t^{n}} =
\sum_{\ell=0}^k\alpha^{\reviewerA{n}}_\ell
\podstate\stateRedP^{n-\ell} 
-\sum_{\ell=1}^k\beta^{\reviewerA{n}}_\ell \dt\bbP\f \del{\stateInitialNo +
\podstate  \stateRedP^{n-\ell}, t^{n-\ell}}
\end{equation} 
such that
\begin{align} 
\beta_0^n\dt(\matI - \bbP^n)\f \del{\stateInitialNo + \podstate  \stateRedP^{n},
t^{n}} &=
\beta_0^n\dt\f \del{\stateInitialNo + \podstate  \stateRedP^{n}, t^{n}} - \sum_{\ell=0}^k\alpha^{\reviewerA{n}}_\ell
\podstate\stateRedP^{n-\ell} 
+\sum_{\ell=1}^k\beta^{\reviewerA{n}}_\ell \dt\bbP\f  \del{\stateInitialNo +
\podstate  \stateRedP^{n-\ell}, t^{n-\ell}}\\
&=-\res^\timestepit(\stateInitialNo + \podstate\stateRedP^{n-\ell}) -
(\matI-\bbP)\sum_{\ell=1}^k\beta^{\reviewerA{n}}_\ell \dt\f 
\del{\stateInitialNo +
\podstate  \stateRedP^{n-\ell}, t^{n-\ell}}
\end{align} 
Then, if $\beta_\ell^n=0$, $\ell\geq 1$, the final term vanishes and we have
Eq.~\eqref{eq:lspgObliqueDisc},
where a similar result (i.e., Eq.~\eqref{eq:galObliqueDisc}) holds for the Galerkin ROM.
If the LSPG ROM employs $\weightingMatrix=\matI$, then the LSPG ROM solution $\stateRedP^{n}$ satisfies
Eq.~\eqref{eq:discreteOptLinear2} with $\weightingMatrix=\matI$, which yields
Eq.~\eqref{eq:lspgObliqueDisc2}.
\end{proof}
\begin{corollary}[Discrete v.~continuous residual
minimization]\label{cor:discreteContBetter}
If $\beta_j^n=0$, $j\geq 1$ (e.g., backward differentiation formulas), the LSPG ROM employs $\weightingMatrix = \matI$,
and $\stateRedP^{n-\ell} = \stateRedG^{n-\ell}$, $\ell=1,\ldots,k$, then 
 \begin{align} \label{eq:LSPGbetter}
 \begin{split} 
&\min_{\vec
y\in\range{\podstate}}\|\resDiscTotal{\timestepit}(\vec y)\|_2
=
\abs{\beta_0^n}\dt\|(\matI - \bbP^n)\f \del{\stateInitialNo + \podstate
\stateRedP^{n}, t^{n}}\|_{2}\\
&\leq
\abs{\beta_0^n}\dt\|(\matI - \bbV)\f \del{\stateInitialNo + \podstate \stateRedG^{n}, t^{n}}\|_{2} 
= \abs{\beta_0^n}\dt\min_{\vec
y\in\range{\podstate}}\|\vec y - \f \del{\stateInitialNo +
\podstate  \stateRedG^{n}, t^{n}}\|_2.
 \end{split} 
  \end{align} 
\end{corollary}
\begin{proof}
Under the stated conditions, $\resGalTotal{n} = \resDiscTotal{n}$ and the
optimality result Eq.~\eqref{eq:lspgObliqueDisc2} holds, yielding the desired
result.
\end{proof}

Corollary \ref{cor:discreteContBetter} shows that discrete-residual
minimization (i.e., LSPG projection) rather than continuous-residual
minimization (i.e., Galerkin projection) can produce a smaller value of a term
that appears in the \textit{a posteriori} error bounds; for example, this term
appears on the right-hand side of Eqs.~\eqref{eq:error_FOM_ROM_G_multi_stepfirst2}--\eqref{eq:error_FOM_ROM_P_multi_stepfirst2}. This can be interpreted as arising from
the fact that discrete-residual minimization computes the LSPG ROM solution that
minimizes the \textit{entire} normed quantity, while continuous-residual
minimization performs orthogonal projection of the velocity \textit{given} the
Galerkin ROM solution.
}

\reviewerCNewer{
\begin{corollary}[LSPG can produce lower local \textit{a posteriori} error
bounds than Galerkin]\label{cor:discreteOptBeatsGal}
Under the assumptions of Corollary \ref{cor:discreteContBetter} and Theorem
\ref{theorem:localAPosterioriLMM}, the local \textit{a posteriori}
error bound for the LSPG ROM \eqref{eq:error_FOM_ROM_P_multi_stepfirst2} is
smaller than that for the Galerkin ROM
\eqref{eq:error_FOM_ROM_G_multi_stepfirst2}.
\end{corollary}	
\begin{proof}
Under the conditions of Theorem \ref{theorem:localAPosterioriLMM},
Eqs.~\eqref{eq:error_FOM_ROM_P_multi_stepfirst2} and
\eqref{eq:error_FOM_ROM_G_multi_stepfirst2} are valid local \textit{a
posteriori} error bounds. If $\beta_j^n=0$, $j\geq 1$, and $\stateRedP^{n-\ell} =
\stateRedG^{n-\ell}$, $\ell=1,\ldots,k$, then these
bounds simplify to 
\begin{align}
    \normtwo{\deltastateRedG^{n}} 
&\le \frac{|\beta_0^n|\dt}{h^m}\normtwo{ \del{ \matI  - \bbV }
\f \del{\stateInitialNo + \podstate  \stateRedG^{n-\ell}} }  
    + \sum_{\ell=1}^k \secondCoeffMultin{\ell} \normtwo{\deltastateRedP^{n-\ell}}
		\label{eq:error_FOM_ROM_G_multi_stepfirst2simple}\\
    \normtwo{\deltastateRedP^{n}} 
&\le  \frac{|\beta_0^n|\dt}{h^m} \normtwo{ \del{ \matI  - \bbP^{n}
} \f \del{\stateInitialNo + \podstate  \stateRedP^{n-\ell}} }  
    + \sum_{\ell=1}^k \secondCoeffMultin{\ell} \normtwo{\deltastateRedP^{n-\ell}}
		\label{eq:error_FOM_ROM_P_multi_stepfirst2simple},
\end{align}
respectively. Under the conditions of Corollary \ref{cor:discreteContBetter},
inequality \eqref{eq:LSPGbetter} holds and the right-hand side of \eqref{eq:error_FOM_ROM_P_multi_stepfirst2simple}
will be smaller than the right-hand side of \eqref{eq:error_FOM_ROM_G_multi_stepfirst2simple}.
\end{proof}

This is an interesting result, as it provides conditions under which the LSPG
ROM produces a smaller local \textit{a posteriori} error bound than the
Galerkin ROM. It also provides some theoretical justification
\reviewerA{for the numerical
experiments in Section \ref{s:numerics}, which use the three-point
backward-differentiation formula, wherein the \reviewerA{LSPG} uniformly outperforms the
Galerkin ROM.}
	}
\reviewerCNewer{
We now extend to results of Theorem \ref{theorem:localAPosterioriLMM} to obtain global
\textit{a posteriori} error bounds.
}
\begin{theorem}[Global \reviewerA{\textit{a posteriori} error bounds: linear multistep schemes}] \label{eq:error_multi_step}
\reviewerCNewer{Under the assumptions of Theorem
\ref{theorem:localAPosterioriLMM}, we have}
\reviewerA{
\begin{align}
    \normtwo{\deltastateRedG^{n}} 
&\le\sum_{j= 0}^{n-1}\sum_{\ell =0}^{\min(k,j)}\left[
\indicator_{\{0\}}(j-\ell) +
\sum_{\multiIndexSet\in\mathcal A(j-\ell )} \prod_{i=1}^{|\multiIndexSet|}\secondCoeffMulti{\multiIndex{i}}^{n -
\sum_{m=1}^{i-1}\multiIndex{m}}\right]\firstCoeffMulti{\ell }^{n-j + \ell }
\normtwo{ \del{ \matI  -
\bbV } \f \del{\stateInitialNo + \podstate  \stateRedG^{n-j}} }
		\label{eq:error_FOM_ROM_G_multi_step}\\
    \normtwo{\deltastateRedP^{n}} 
&\le\sum_{j= 0}^{n-1}\sum_{\ell =0}^{\min(k,j)}\left[
\indicator_{\{0\}}(j-\ell) +
\sum_{\multiIndexSet\in\mathcal A(j-\ell )} \prod_{i=1}^{|\multiIndexSet|}\secondCoeffMulti{\multiIndex{i}}^{n -
\sum_{m=1}^{i-1}\multiIndex{m}}\right]\firstCoeffMulti{\ell }^{n-j + \ell }
\normtwo{ \del{ \matI  -
\bbP^{n-j + \ell } } \f \del{\stateInitialNo + \podstate  \stateRedP^{n-j}} }
		\label{eq:error_FOM_ROM_P_multi_step}.
\end{align}
Here,
 $\indicator_A(x)$ denotes the indicator function, $\mathcal A(p)\defeq\{(\eta_i)\ |\
\eta_i\innat{k},\ \sum_i\eta_i = p\}$, and $|\multiIndexSet|$ denotes the
length of the tuple $\multiIndexSet$.
} 
\end{theorem}
\begin{proof}
Notice that the term $\normtwo{ \del{ \matI  -
\bbP^{i} } \f \del{\stateInitialNo + \podstate  \stateRedP^{i-j}} }$ in
inequality \eqref{eq:error_FOM_ROM_P_multi_stepfirst2} corresponds to the
error introduced at time \reviewerCNew{instance} $i\innat{n}$ from the state
at time \reviewerCNew{instance} $i-j$
with $j\innat{k}$; this term always
appears in the time-local error bound with coefficient
$\firstCoeffMulti{j}^i$. Further, it 
contributes to the error at a given time \reviewerCNew{instance} $n > i$ through appropriate products of
$\secondCoeffMulti{\ell}^m$. For example, the product
$\gamma_1^n\gamma_2^{n-1}\gamma_1^{n-3}$ provides one possible path for
`traversing' the time-local error bounds from time \reviewerCNew{instance} $n$ to an earlier error contribution at
time \reviewerCNew{instance} $n-4$. Applying this notion more generally and 
using $\deltastateRedP^{0}=\zero$, the error can be bounded
by induction according to inequality \eqref{eq:error_FOM_ROM_P_multi_step}.
\end{proof}

\reviewerA{
The bounds in Theorem~\ref{eq:error_multi_step} can be considered \textit{a
posteriori} error bounds, as they depend on the ROM solutions $\stateRedG$ and
$\stateRedP$ and can thus be computed \textit{a posteriori} if
\reviewerCNewer{the Lipschitz constant} $\lipschitzConstant$ can be
estimated. Note that the rightmost term in the Galerkin bound corresponds to
the orthogonal projection error of 
$\f $ onto $\range{\podstate}$, while the LSPG bound entails an oblique
projector that depends on the ROM solution. Because this oblique projection
\reviewerCNewer{can} associate with a discrete residual-minimization property
\reviewerCNewer{(i.e., Corollary \ref{cor:discreteContBetter})}, the LSPG ROM can yield
smaller error bounds \reviewerCNewer{(i.e., Corollary
\ref{cor:discreteOptBeatsGal})}.
Also,
the first term within square brackets corresponds to errors incurred at the
current time \reviewerCNew{instance} $n$ (i.e., via the leftmost term on the right-hand side of inequality
\eqref{eq:error_FOM_ROM_P_multi_stepfirst2}), while the
second term corresponds to all possible `paths' from current time
\reviewerCNew{instance}
$n$ to the error contribution at previous time \reviewerCNew{instances} (i.e., the rightmost term
on the right-hand side of
inequality \eqref{eq:error_FOM_ROM_P_multi_stepfirst2}).
\reviewerANew{We also note the importance of the time-step condition 
$\dt < \abs{\alpha^{\reviewerA{j}}_0}/(\abs{\beta^{\reviewerA{j}}_0}
\lipschitzConstant )$,
as the stability constants appearing in bounds
\eqref{eq:error_FOM_ROM_G_multi_step}--\eqref{eq:error_FOM_ROM_P_multi_step}
exhibit unbounded growth as the time step
$\dt$ approaches its upper limit, i.e., 
$$\lim\limits_{\dt\rightarrow\abs{\alpha^{\reviewerA{k}}_0}/(\abs{\beta^{\reviewerA{k}}_0}
\lipschitzConstant )}h^k = 0,\quad \forall k\innat{n},$$
which implies that 
$$\lim\limits_{\dt\rightarrow\abs{\alpha^{\reviewerA{k}}_0}/(\abs{\beta^{\reviewerA{k}}_0}
\lipschitzConstant )}\firstCoeffMulti{\ell}^k=\infty\quad\text{and}\quad
\lim\limits_{\dt\rightarrow\abs{\alpha^{\reviewerA{k}}_0}/(\abs{\beta^{\reviewerA{k}}_0}
\lipschitzConstant )}\secondCoeffMulti{\ell}^k=\infty,\quad\forall
k\innat{n}.$$}
\reviewerANewer{
We now derive a simplified variant of these bounds.
\begin{theorem}[\reviewerA{Simplified global \textit{a posteriori} error bounds: linear
multistep schemes}] \label{thm:simplifiedGlobalAposterioriLMM}
Under the assumptions of Theorem \ref{theorem:localAPosterioriLMM}, if
additionally
$\dt \leq
|\alpha_0^\star|(1-\threshold)/(\lipschitzConstant|\beta_0^\star|)$ with
$0 < \threshold < 1$. Then,
\begin{align}
\label{eq:aposterioriLMMboundSimplerGal}
\normtwo{\deltastateRedG^{n}} 
&\leq(k+1)\abs{\betamax}\dt\left(\frac{k|\alpha^\star|}{|\alpha_0^\star|}\right)^n\frac{\exp\left(t^n\lipschitzConstant\threshold^{-1}\left(|\beta^\star|/|\alpha^\star|
+|\beta_0^\star|/|\alpha_0^\star|\right)\right)
-1}{(k|\alpha^\star|  -
|\alpha_0^\star|) +
(k|\beta^\star|+|\beta_0^\star|)\lipschitzConstant\dt}
\max_{j\in\nat{n}}\normtwo{ \del{ \matI  - \bbV } 
\f\del{\stateInitialNo + \podstate 
\stateRedG^{j}} }\\
\label{eq:aposterioriLMMboundSimplerPG}\normtwo{\deltastateRedP^{n}} 
&\leq(k+1)\abs{\betamax}\dt\left(\frac{k|\alpha^\star|}{|\alpha_0^\star|}\right)^n\frac{\exp\left(t^n\lipschitzConstant\threshold^{-1}\left(|\beta^\star|/|\alpha^\star|
+|\beta_0^\star|/|\alpha_0^\star|\right)\right)
-1}{(k|\alpha^\star|  -
|\alpha_0^\star|) +
(k|\beta^\star|+|\beta_0^\star|)\lipschitzConstant\dt}
\max_{\substack{j\in\nat{n}\\ \ell\in\natZero{\elltwostarmax}}}\normtwo{ \del{ \matI  - \bbP^{j} } 
\f\del{\stateInitialNo + \podstate 
\stateRedP^{j-\ell}} }.
\end{align}
Here, we have defined
\begin{gather}
 \abs{\alpha_0^\star} -
\abs{\beta_0^\star}\lipschitzConstant\dt\defeq \min_{j\in\nat{n} }\abs{\alpha_0^j} -
\abs{\beta_0^j}\lipschitzConstant\dt \\
 \abs{\alpha^\star} +
\abs{\beta^\star}\lipschitzConstant\dt\defeq\max_{j\in\nat{n},\ \ell\in\nat{k}}\abs{\alpha_\ell^j} +
\abs{\beta_\ell^j}\lipschitzConstant\dt\\
\abs{\betamax}\defeq \max_{j\in\nat{n},\ \ell\in\natZero{k}}|\beta_\ell^j|\\
k\geq\elltwostarmax\defeq\max_{j\in\nat{n}}\elltwostar{j},\quad\quad
\elltwostar{n}\defeq\underset{\ell\in\natZero{k}}{\arg\max}\
\firstCoeffMultin{\ell} \normtwo{ \del{ \matI  - \bbP^{n} } \f \del{
\stateInitialNo + \podstate  \stateRedP^{n-\ell}} }.
\end{gather}
\end{theorem}
\begin{proof}
We proceed by proving bound \eqref{eq:aposterioriLMMboundSimplerPG}; the proof
for bound \eqref{eq:aposterioriLMMboundSimplerGal} is similar.
First, we define
\begin{gather}
\ellstar{n}\defeq\underset{\ell\in\natZero{k}}{\arg\max}\
\secondCoeffMultin{\ell}
\normtwo{\deltastateRedP^{n-\ell}},
\end{gather}
as well as 
a `path' from time instance $n$ backward to the initial time by
defining $\timestepset{0} = n$ with
\begin{equation}
\timestepset{j+1} = \timestepset{j}-\ellstar{\timestepset{j}},\
j=0,\ldots,\sizetimestepset-1,
\end{equation}
with $ \sizetimestepset \leq n$ and $\timestepset{\sizetimestepset} = 0$.
Then, from local \textit{a posteriori} error bound~\eqref{eq:error_FOM_ROM_P_multi_stepfirst2}, we have
\begin{align}
    \normtwo{\deltastateRedP^{n}} 
&\le 
(k+1)\firstCoeffMultin{\elltwostar{n}} \normtwo{ \del{ \matI  - \bbP^{n} } 
\f\del{\stateInitialNo + \podstate  \stateRedP^{n-\elltwostar{n}}} }  
    + k\secondCoeffMultin{\ellstar{n}}
		\normtwo{\deltastateRedP^{n-\ellstar{n}}},
\end{align}
where $\firstCoeffMultin{\elltwostar{n}}\defeq\abs{\betamax}\dt/(\abs{\alpha_0^\star} -
\abs{\beta_0^\star}\lipschitzConstant\dt)$ and
$\secondCoeffMultin{\ellstar{n}}\defeq(\abs{\alpha^\star} +
\abs{\beta^\star}\lipschitzConstant\dt)/(\abs{\alpha_0^\star} -
\abs{\beta_0^\star}\lipschitzConstant\dt)$.
Via recursion, we then obtain
\begin{align}
    \normtwo{\deltastateRedP^{n}} &\leq \sum_{j=0}^{\sizetimestepset-1}
		\left(k^j
	\prod_{m=0}^{j-1}\secondCoeffMultinTwoarg{\ellstar{\timestepset{m}}}{\timestepset{m}}
		\right)
(k+1)\firstCoeffMultinTwoarg{\elltwostar{\timestepset{j}}}{\timestepset{j}} \normtwo{ \del{ \matI  - \bbP^{\timestepset{j}} } 
\f\del{\stateInitialNo + \podstate 
\stateRedP^{\timestepset{j}-\elltwostar{\timestepset{j}}}} }\\
&\leq (k+1)\abs{\betamax}\sum_{j=0}^{\sizetimestepset-1}
k^j\left(\frac{|\alpha^\star| + |\beta^\star|\lipschitzConstant\dt}{|\alpha_0^\star| -
|\beta_0^\star|\lipschitzConstant\dt}\right)^j
\frac{\dt}{|\alpha_0^\star| -
|\beta_0^\star|\lipschitzConstant\dt}
 \normtwo{ \del{ \matI  - \bbP^{\timestepset{j}} } 
\f\del{\stateInitialNo + \podstate 
\stateRedP^{\timestepset{j}-\elltwostar{\timestepset{j}}}} }\\
&\leq
(k+1)\abs{\betamax}\left(\sum_{m=0}^{n-1}
k^m\left(\frac{|\alpha^\star| + |\beta^\star|\lipschitzConstant\dt}{|\alpha_0^\star| -
|\beta_0^\star|\lipschitzConstant\dt}\right)^m
\frac{\dt}{|\alpha_0^\star| -
|\beta_0^\star|\lipschitzConstant\dt}\right)\max_{\substack{j\in\nat{n} \\
\ell\in\natZero{\elltwostarmax}}}\normtwo{ \del{ \matI  - \bbP^{j} } 
\f\del{\stateInitialNo + \podstate 
\stateRedP^{j-\ell}} }\\
&=
(k+1)\abs{\betamax}\dt
\frac{\zSymb^n(\frac{\aSymb + \bSymb\kSymb\tSymb}{\dSymb-\fSymb\kSymb
\tSymb})^{n}-1}{\zSymb(\aSymb + \bSymb\kSymb\tSymb) - \dSymb +
\fSymb\kSymb\tSymb}
\max_{\substack{j\in\nat{n} \\ \ell\in\natZero{\elltwostarmax}}}\normtwo{ \del{ \matI  - \bbP^{j} } 
\f\del{\stateInitialNo + \podstate 
\stateRedP^{j-\ell}} }\\
&\leq
(k+1)\abs{\betamax}\dt\left(\frac{k\aSymb}{\dSymb}\right)^n
\frac{\exp\left(t^n\lipschitzConstant\threshold^{-1}\left(|\beta^\star|/|\alpha^\star|
+|\beta_0^\star|/|\alpha_0^\star|\right)\right)
-1}
{(\zSymb\aSymb - \dSymb) + (\zSymb\bSymb  +
\fSymb)\kSymb\dt}
\max_{\substack{j\in\nat{n}\\ \ell\in\natZero{\elltwostarmax}}}\normtwo{ \del{ \matI  - \bbP^{j} } 
\f\del{\stateInitialNo + \podstate 
\stateRedP^{j-\ell}} },
\end{align}
where we have used 
$$\left(\frac{\aSymb + \bSymb\kSymb\tSymb}{\dSymb-\fSymb\kSymb
\tSymb}\right)^n 
= \left(\frac{\aSymb}{\dSymb}\right)^n\left(\frac{1 + \kSymb\tSymb\bSymb/\aSymb}{1-\kSymb
\tSymb\fSymb/\dSymb}\right)^n 
= \left(\frac{\aSymb}{\dSymb}\right)^n\left(1 +
\frac{\kSymb\tSymb(\bSymb/\aSymb + \fSymb/\dSymb)}{1-\kSymb
\tSymb\fSymb/\dSymb}\right)^n,
$$
the relation $(1 + x)^n\leq \exp(nx)$, and the following result (with $x = 1 +
\kSymb\tSymb\bSymb/\aSymb$ and $y=1-\kSymb
\tSymb\fSymb/\dSymb$): if $x\geq y$, then $(x-y)/y\leq \threshold^{-1} (x-y)$
if and only if
$y \geq \threshold > 0$.
\end{proof}

\noindent We note that due to the optimality property associated with orthogonal
projection, we can write bound \eqref{eq:aposterioriLMMboundSimplerGal} equivalently as
\begin{equation} \label{eq:galGlobalTimeOpt}
\normtwo{\deltastateRedG^{n}} \leq
\abs{\betamax}\dt\left(\frac{k|\alpha^\star|}{|\alpha_0^\star|}\right)^n\frac{\exp\left(t^n\lipschitzConstant\threshold^{-1}\left(|\beta^\star|/|\alpha^\star|
+|\beta_0^\star|/|\alpha_0^\star|\right)\right)
-1}{(k|\alpha^\star|  -
|\alpha_0^\star|) +
(k|\beta^\star|+|\beta_0^\star|)\lipschitzConstant\dt}
\max_{j\in\nat{n}}\min_{\vec
y\in\range{\podstate}}\normtwo{ \vec y - \f\del{\stateInitialNo + \podstate 
\stateRedG^{j}} }.
\end{equation} 
 
\noindent We now prove conditions under which the \textit{a posteriori} error bound is
 independent of the time step $\dt$ and total number of time instances $n$;
 the bound is fixed for a given time $t^n$.
\begin{corollary}[Time-step-independent global \textit{a posteriori} error
bounds: linear multistep schemes]\label{cor:timestepIndependenceAPosteriori}
Under the assumptions of Theorem \ref{thm:simplifiedGlobalAposterioriLMM}, if additionally $k|\alpha^\star|
= |\alpha_0^\star|$ (e.g., backward Euler, where $k=1$ and $|\alpha^\star| =
|\alpha_0^\star| = 1$) then the global \textit{a posteriori} error bounds
\eqref{eq:aposterioriLMMboundSimplerGal}--\eqref{eq:aposterioriLMMboundSimplerPG}
 are independent
of the time step and simplify to 
\begin{align}
\label{eq:aposterioriLMMboundSimplerGalDtIndep}
\normtwo{\deltastateRedG^{n}} 
&\leq\frac{(k+1)\abs{\betamax}}{(k|\beta^\star|+|\beta_0^\star|)\lipschitzConstant}\left(\exp\left(t^n\lipschitzConstant\threshold^{-1}\left(|\beta^\star|/|\alpha^\star|
+|\beta_0^\star|/|\alpha_0^\star|\right)\right)
-1\right)
\max_{j\in\nat{n}}\normtwo{ \del{ \matI  - \bbV } 
\f\del{\stateInitialNo + \podstate 
\stateRedG^{j}} }\\
\label{eq:aposterioriLMMboundSimplerPGDtIndep}
\normtwo{\deltastateRedP^{n}} 
&\leq\frac{(k+1)\abs{\betamax}}{
(k|\beta^\star|+|\beta_0^\star|)\lipschitzConstant}\left(\exp\left(t^n\lipschitzConstant\threshold^{-1}\left(|\beta^\star|/|\alpha^\star|
+|\beta_0^\star|/|\alpha_0^\star|\right)\right)
-1\right)
\max_{\substack{j\in\nat{n}\\ \ell\in\natZero{\elltwostarmax}}}\normtwo{ \del{ \matI  - \bbP^{j} } 
\f\del{\stateInitialNo + \podstate 
\stateRedP^{j-\ell}} }.
\end{align}
\end{corollary}
\begin{proof}
The result can be derived by substituting $k|\alpha^\star|
= |\alpha_0^\star|$ into inequalities \eqref{eq:aposterioriLMMboundSimplerGal}
and \eqref{eq:aposterioriLMMboundSimplerPG}.
\end{proof}

We now prove conditions under which \textit{a posteriori} error bounds
\eqref{eq:aposterioriLMMboundSimplerGal}--\eqref{eq:aposterioriLMMboundSimplerPG} can
be written in `residual form', i.e., in terms of the discrete residual
arising at each time step. This will enable the respective optimality
properties of the Galerkin and LSPG ROMs to be compared in Remark \ref{rem:optimalityCompare}.

\begin{corollary}[\reviewerA{Simplified global \textit{a posteriori} error
bounds in residual form: linear
multistep schemes}] \label{eq:error_multi_step_twoResidual}
Under the assumptions of Theorem \ref{thm:simplifiedGlobalAposterioriLMM}, if additionally 
 $\beta_j^m=0$,
$j\geq 1$,
$\forall m\innat{n}$
(e.g., backward differentiation formulas). Then, 
\begin{gather}
\label{eq:aposterioriLMMboundSimplerGalRes}
\normtwo{\deltastateRedG^{n}} 
\leq(k+1)\left(\frac{k|\alpha^\star|}{|\alpha_0^\star|}\right)^n\frac{\exp\left(t^n\lipschitzConstant\threshold^{-1}\left(|\beta^\star|/|\alpha^\star|
+|\beta_0^\star|/|\alpha_0^\star|\right)\right)
-1}{(k|\alpha^\star|  -
|\alpha_0^\star|) +
(k|\beta^\star|+|\beta_0^\star|)\lipschitzConstant\dt}
\max_{j\in\nat{n}}\|\resGalTotal{j}(
\podstate \stateRedG^{j})\|_2\\
\label{eq:aposterioriLMMboundSimplerPGRes}\normtwo{\deltastateRedP^{n}} 
\leq(k+1)\left(\frac{k|\alpha^\star|}{|\alpha_0^\star|}\right)^n\frac{\exp\left(t^n\lipschitzConstant\threshold^{-1}\left(|\beta^\star|/|\alpha^\star|
+|\beta_0^\star|/|\alpha_0^\star|\right)\right)
-1}{(k|\alpha^\star|  -
|\alpha_0^\star|) +
(k|\beta^\star|+|\beta_0^\star|)\lipschitzConstant\dt}
\max_{j\in\nat{n}}\|\resDiscTotal{j}(
\podstate \stateRedP^j)\|_2.
\end{gather}
If additionally $\weightingMatrix = \matI$, then
\begin{gather}
\label{eq:aposterioriLMMboundSimplerPGResOpt}\normtwo{\deltastateRedP^{n}} 
\leq(k+1)\left(\frac{k|\alpha^\star|}{|\alpha_0^\star|}\right)^n\frac{\exp\left(t^n\lipschitzConstant\threshold^{-1}\left(|\beta^\star|/|\alpha^\star|
+|\beta_0^\star|/|\alpha_0^\star|\right)\right)
-1}{(k|\alpha^\star|  -
|\alpha_0^\star|) +
(k|\beta^\star|+|\beta_0^\star|)\lipschitzConstant\dt}
\max_{j\in\nat{n}}\min_{\vec
y\in\range{\podstate}}\|\resDiscTotal{j}(
\vec y)\|_2.
\end{gather}
\end{corollary}
\begin{proof}
Under the stated assumptions $\betamax = \beta_0$, $\elltwostarmax = 0$, and Lemma
\ref{lemm:obliqueDiscreteResMin} holds, yielding the desired result.
\end{proof}
}

\reviewerANewer{
	\begin{remark}[Optimality in \textit{a posteriori} error
	bounds]\label{rem:optimalityCompare}
Comparing inequalities \eqref{eq:galGlobalTimeOpt} and
\eqref{eq:aposterioriLMMboundSimplerPGResOpt} highlights the differences
in how optimality affects the Galerkin and LSPG ROM error bounds. Writing these expressions more
compactly under the conditions of Corollary \ref{eq:error_multi_step_twoResidual} yields
\begin{gather}
\label{eq:aposterioriLMMboundSimplerGaloptimality}
\normtwo{\deltastateRedG^{n}} 
\leq\continuousTerm\beta_0\dt
\max_{j\in\nat{n}}\min_{\vec y\in\range{\Phi}}\|\vec y - \f\del{\stateInitialNo + \podstate 
\stateRedG^{j}}\|_2\\
\normtwo{\deltastateRedP^{n}} 
\leq\continuousTerm
\max_{j\in\nat{n}}\min_{\optVec\in\stateInitialNo +
 \range{\podstate}}\normtwo{\resDiscTotal{j} \left(\optVec\right)}.
\end{gather}
where 
$$\continuousTerm\defeq(k+1)\left(\frac{k|\alpha^\star|}{|\alpha_0^\star|}\right)^n\frac{\exp\left(t^n\lipschitzConstant\threshold^{-1}\left(|\beta^\star|/|\alpha^\star|
+|\beta_0^\star|/|\alpha_0^\star|\right)\right)
-1}{(k|\alpha^\star|  -
|\alpha_0^\star|) +
(k|\beta^\star|+|\beta_0^\star|)\lipschitzConstant\dt}.$$
	\end{remark}
}

\subsubsection{\textit{A priori} error bounds}
We now derive \textit{a priori} error bounds by slightly modifying
the steps in the above proof\reviewerANewer{s}. \reviewerCNewer{The most significant
difference in the subsequent results is that the oblique projection associated
with the LSPG ROM no longer associates with residual minimization, as the
argument of the operators corresponds to the full-order-model solution.}

\reviewerCNewer{
\begin{corollary}[Local {\textit{a priori} error bounds: linear multistep
schemes}]\label{theorem:localAPrioriLMM}
If $\del{\bf A_1}$ holds,  
$\dt < \abs{\alpha^{j}_0}/(\abs{\beta^{j}_0}
\lipschitzConstant )$, $\forall j\innat{n}$ for the Galerkin ROM, and $\dt <
\abs{\alpha^{j}_0}/(\abs{\beta^{j}_0} \lipschitzConstant \| \bbP^n \|)$, $\forall j\innat{n}$ for
the LSPG ROM, then
\begin{align}
    \normtwo{\deltastateRedG^{n}} 
&\le \sum_{\ell=0}^k \firstCoeffMultin{\ell} \normtwo{ \del{ \matI  - \bbV }
\f \del{\stateInitialNo + \state_\star^{n-\ell}} }  
    + \sum_{\ell=1}^k \secondCoeffMultin{\ell} \normtwo{\deltastateRedG^{n-\ell}}
		\label{eq:error_FOM_ROM_G_multi_stepfirst2priori}\\
\normtwo{\deltastateRedP^{n}} 
&\le \sum_{\ell=0}^k \firstCoeffMultiD{\ell}^n \normtwo{ \del{ \matI  - \bbP^{n}
} \f \del{\stateInitialNo + \state_\star^{n-\ell}} }  
    + \sum_{\ell=1}^k \secondCoeffMultiD{\ell}^n \normtwo{\deltastateRedP^{n-\ell}}
		\label{eq:error_FOM_ROM_P_multi_stepfirst2priori}
\end{align}
where we have defined $\firstCoeffMultiD{\ell}^{m}\defeq
\abs{\beta^{m}_\ell} \dt
/\bar h^{m}$
$\secondCoeffMultiD{\ell}^{m}\defeq
(\reviewerCNew{\abs{\alpha^{m}_\ell}+\abs{\beta^{m}_\ell} \lipschitzConstant
		\dt \| \bbP^{m} \|_2 })/\bar h^{m}$, 
		$\bar h^{m} \defeq \abs{\alpha^{m}_0} -
		\abs{\beta^{m}_0} \lipschitzConstant  \dt \|
		\bbP^{m} \|_2$. Other quantities
		are defined in Theorem \ref{eq:error_multi_step}.
\end{corollary}
\begin{proof}
Instead of adding and subtracting $\f \del{\stateInitialNo + \podstate 
\stateRedP^{n}}$ between Eqs.~\eqref{eq:beforeAddSubtract} and
\eqref{eq:afterAddSubtract} and between Eqs.~\eqref{eq:beforeAddSubtract2} and
\eqref{eq:error_FOM_ROM_P_multi_step_b} in the proof of
Theorem~\ref{eq:error_multi_step}, adding and subtracting $\bbP^n 
\f\del{\stateInitialNo + \state_\star^{n}}$ and using $\|\bbV\|_2 = 1$ yields the
stated result. 
\end{proof}
}
}

Note that bound
\eqref{eq:error_FOM_ROM_P_multi_stepfirst2priori} is not quite an \textit{a
priori} bound, as the operator $\bbP^{\reviewerCNewer{n}}$ depends on the LSPG ROM solution
$\stateRedP^\reviewerCNewer{n}$; while this dependence could be removed, the bound in its
current form facilitates comparison with the Galerkin bound.
\reviewerCNewer{The same dependence persists for the remaining
\textit{a priori} error bounds in this section.}

\begin{corollary}[\reviewerCNewer{Global} \textit{a priori} error bounds: linear multistep schemes] \label{thm:error_multi_step_apriori}
\reviewerCNewer{Under the assumptions of Corollary
\ref{theorem:localAPrioriLMM}, we have}
\begin{align}
    \normtwo{\deltastateRedG^{n}} 
&\le\sum_{j= 0}^{n-1}\sum_{\ell =0}^{\min(k,j)}\left[
\indicator_{\{0\}}(j-\ell) +
\sum_{\multiIndexSet\in\mathcal A(j-\ell )} \prod_{i=1}^{|\multiIndexSet|}\secondCoeffMulti{\multiIndex{i}}^{n -
\sum_{m=1}^{i-1}\multiIndex{m}}\right]\firstCoeffMulti{\ell }^{n-j + \ell }
\normtwo{ \del{ \matI  -
\bbV } \f \del{\stateInitialNo + \state_\star^{n-j}} }
		\label{eq:error_FOM_ROM_G_multi_step_apriori}\\
    \normtwo{\deltastateRedP^{n}} 
&\le\sum_{j= 0}^{n-1}\sum_{\ell =0}^{\min(k,j)}\left[
\indicator_{\{0\}}(j-\ell) +
\sum_{\multiIndexSet\in\mathcal A(j-\ell )}
\prod_{i=1}^{|\multiIndexSet|}\secondCoeffMultiD{\multiIndex{i}}^{n -
\sum_{m=1}^{i-1}\multiIndex{m}}\right]\firstCoeffMultiD{\ell }^{n-j + \ell }
\normtwo{ \del{ \matI  -
\bbP^{n-j + \ell } } \f \del{\stateInitialNo + \state_\star^{n-j}} }
		\label{eq:error_FOM_ROM_P_multi_step_apriori},
\end{align}
\end{corollary}
\begin{proof}
\reviewerCNewer{The result can be derived by following the same steps as Theorem \ref{eq:error_multi_step} based on the
local bounds in \reviewerCNewer{Corollary} \ref{theorem:localAPrioriLMM}.}
\end{proof}

\reviewerCNewer{
\begin{corollary}[\reviewerA{Simplified global \textit{a priori} error bounds: linear
multistep schemes}] \label{eq:error_multi_step_twoApriori}
Under the assumptions of Corollary \ref{theorem:localAPrioriLMM}, if
additionally
$\dt \leq
|\alpha_0^\star|(1-\threshold)/(\lipschitzConstant|\beta_0^\star|)$ with
$0 < \threshold < 1$ for the Galerkin ROM, and 
$\dt \leq
|\bar\alpha_0^\star|(1-\thresholdLSPG)/(\lipschitzConstant|\bar\beta_0^\star|\|\|\bbP_0^\star\|_2)$ with
$0 < \thresholdLSPG < 1$ for the LSPG ROM, then
\begin{align}
\label{eq:aprioriLMMboundSimplerGal}
\normtwo{\deltastateRedG^{n}} 
\leq&(k+1)\abs{\betamax}\dt\left(\frac{k|\alpha^\star|}{|\alpha_0^\star|}\right)^n\frac{\exp\left(t^n\lipschitzConstant\threshold^{-1}\left(|\beta^\star|/|\alpha^\star|
+|\beta_0^\star|/|\alpha_0^\star|\right)\right)
-1}{(k|\alpha^\star|  -
|\alpha_0^\star|) +
(k|\beta^\star|+|\beta_0^\star|)\lipschitzConstant\dt}
\max_{j\in\nat{n}}\normtwo{ \del{ \matI  - \bbV } 
\f\del{\stateInitialNo + \state_\star^{j}} }\\
\begin{split}
\label{eq:aprioriLMMboundSimplerPG}\normtwo{\deltastateRedP^{n}} 
\leq&(k+1)\abs{\betamax}\dt\left(\frac{k|\bar \alpha^\star|}{|\bar
\alpha_0^\star|}\right)^n\frac{\exp\left(t^n\lipschitzConstant\thresholdLSPG^{-1}\left(|\bar
\beta^\star|\|\bbP^\star\|_2/|\bar \alpha^\star|
+|\bar \beta_0^\star|\|\|\bbP_0^\star\|_2/|\bar \alpha_0^\star|\right)\right)
-1}{(k|\bar \alpha^\star|  -
|\bar \alpha_0^\star|) +
(k|\bar \beta^\star|\|\bbP^\star\|_2+|\bar
\beta_0^\star|\|\bbP_0^\star\|_2)\lipschitzConstant\dt}\cdot\\
&\max_{\substack{j\in\nat{n}\\ \ell\in\natZero{\barelltwostarmax}}}\normtwo{ \del{ \matI  - \bbP^{j} } 
\f\del{\stateInitialNo + 
\state_\star^{j-\ell}} }.
\end{split}
\end{align}
Here, we have defined
\begin{gather*}
 \abs{\bar\alpha_0^\star} -
\abs{\bar\beta_0^\star}\lipschitzConstant\dt\|\bbP_0^\star\|_2\defeq \min_{j\in\nat{n} }\abs{\alpha_0^j} -
\abs{\beta_0^j}\lipschitzConstant\dt\|\bbP^j\|_2 \\
 \abs{\bar\alpha^\star} +
\abs{\bar\beta^\star}\lipschitzConstant\dt\|\bbP^\star\|_2\defeq\max_{\
j\in\nat{n},\ \ell\in\nat{k}}\abs{\alpha_\ell^j} +
\abs{\beta_\ell^j}\lipschitzConstant\dt\|\bbP^j\|_2\\
k\geq\barelltwostarmax\defeq\max_{j\in\nat{n}}\barelltwostar{j},\quad\quad
\barelltwostar{n}\defeq\underset{\ell\in\natZero{k}}{\arg\max}\
\firstCoeffMultiD{\ell}^n \normtwo{ \del{ \matI  - \bbP^{n} } \f \del{
\stateInitialNo + \state_\star^{n-\ell}} }.
\end{gather*}
\end{corollary}
\begin{proof}
The result can be derived by following the same steps as Theorem \ref{thm:simplifiedGlobalAposterioriLMM} based on the
local bounds in \reviewerCNewer{Corollary} \ref{theorem:localAPrioriLMM}.
\end{proof}

\noindent We now demonstrate conditions under which the \textit{a priori} error bound is
independent of the time step $\dt$.
\begin{corollary}[\reviewerA{Time-step-independent global \textit{a priori} error bounds: linear
multistep schemes}] \label{cor:timestepIndependenceAPriori}
Under the assumptions of Corollary \ref{eq:error_multi_step_twoApriori}, if additionally 
$k|\alpha^\star|
= |\alpha_0^\star|$ (e.g., backward Euler)
for the Galerkin ROM, and $k|\bar \alpha^\star| = |\bar \alpha_0^\star|$ (e.g.,
backward Euler) for the LSPG ROM, then the \textit{a priori} error bounds
\eqref{eq:aprioriLMMboundSimplerGal}--\eqref{eq:aprioriLMMboundSimplerPG} are
independent of the time step and simplify to 
\begin{align}
\label{eq:aprioriLMMboundSimplerGalTsindep}
\normtwo{\deltastateRedG^{n}} 
\leq&\frac{(k+1)\abs{\betamax}}{
(k|\beta^\star|+|\beta_0^\star|)\lipschitzConstant}\left(\exp\left(t^n\lipschitzConstant\threshold^{-1}\left(|\beta^\star|/|\alpha^\star|
+|\beta_0^\star|/|\alpha_0^\star|\right)\right)
-1\right)
\max_{j\in\nat{n}}\normtwo{ \del{ \matI  - \bbV } 
\f\del{\stateInitialNo + \state_\star^{j}} }\\
\begin{split}
\label{eq:aprioriLMMboundSimplerPGTsindep}\normtwo{\deltastateRedP^{n}} 
\leq&\frac{(k+1)\abs{\betamax}}{
(k|\bar \beta^\star|\|\bbP^\star\|_2+|\bar
\beta_0^\star|\|\bbP_0^\star\|_2)\lipschitzConstant}\left(\exp\left(t^n\lipschitzConstant\thresholdLSPG^{-1}\left(|\bar
\beta^\star|\|\bbP^\star\|_2/|\bar \alpha^\star|
+|\bar \beta_0^\star|\|\|\bbP_0^\star\|_2/|\bar \alpha_0^\star|\right)\right)
-1\right)\cdot\\
&\max_{\substack{j\in\nat{n}\\ \ell\in\natZero{\barelltwostarmax}}}\normtwo{ \del{ \matI  - \bbP^{j} } 
\f\del{\stateInitialNo + 
\state_\star^{j-\ell}} }.
\end{split}
\end{align}
\end{corollary}
\begin{proof}
As in Corollary \ref{cor:timestepIndependenceAPosteriori}, the result can be
shown by substitution of the appropriate assumptions (i.e., $k|\alpha^\star|
= |\alpha_0^\star|$ for the Galerkin ROM,  $k|\bar \alpha^\star| = |\bar
\alpha_0^\star|$ for the LSPG ROM) into 
inequalities \eqref{eq:aprioriLMMboundSimplerGal} and \eqref{eq:aprioriLMMboundSimplerPG}.
\end{proof}
}
\reviewerCNewer{
\begin{remark}[Optimality in \textit{a priori} error bounds]
Because the argument of $\f$ in the \textit{a priori} error bounds corresponds
to the full-order-model solution, it is not possible to relate such quantities
to ROM residuals as was done in the case of \textit{a posteriori} error bounds
in Lemma \ref{lemm:obliqueDiscreteResMin}.
As such, it is
not possible to associate the oblique projection of the LSPG ROM with
minimizing any component of the \textit{a priori} error bounds, as was shown
for \textit{a posteriori} error bounds in Corollary \ref{cor:discreteContBetter}. However,
it is possible to associate Galerkin projection with minimizing terms in the
\textit{a priori} error bounds, i.e., the following optimality result holds
under the conditions of Corollary \ref{eq:error_multi_step_twoApriori}:
\begin{equation}
\label{eq:aprioriLMMboundSimplerGaloptimality}
\normtwo{\deltastateRedG^{n}} 
\leq\bar\continuousTerm
\max_{j\in\nat{n}}\min_{\vec y\in\range{\Phi}}\|\vec y - \f\del{\stateInitialNo +
\state_\star^{j}}\|_2.
\end{equation}
This is analogous to inequality
\eqref{eq:aposterioriLMMboundSimplerGaloptimality} for \textit{a
posteriori} error bounds,
where
$$
\bar\continuousTerm\defeq(k+1)\abs{\betamax}\dt\left(\frac{k|\alpha^\star|}{|\alpha_0^\star|}\right)^n\frac{\exp\left(t^n\lipschitzConstant\threshold^{-1}\left(|\beta^\star|/|\alpha^\star|
+|\beta_0^\star|/|\alpha_0^\star|\right)\right)
-1}{(k|\alpha^\star|  -
|\alpha_0^\star|) +
(k|\beta^\star|+|\beta_0^\star|)\lipschitzConstant\dt}.
$$
\end{remark}
}


\subsection{Backward Euler}\label{sec:errorBackEuler}
\reviewerANew{In practice, error bounds for a specific time integrator can be
derived by substituting the appropriate values of
$\alpha_\ell^\reviewerCNewer{m}$ and
$\beta_\ell^\reviewerCNewer{m}$ into
\reviewerCNewer{the appropriate error bounds.};
doing so can lend additional insight into the error bound. We now perform this
exercise for the backward Euler scheme---which is a single-step method that can be characterized by
Eq.~\eqref{eq:linearMultistepDef} with $k=1$, $\alpha_0 = 1$, $\alpha_1 = -1$,
$\beta_0 = 1$, and $\beta_1 = 0$---\reviewerCNewer{and selected bounds}.}

\reviewerCNewer{We first note that the backward Euler scheme satisfies $\beta_j^n=0$, $j\geq 1$
in Corollary \ref{cor:discreteOptBeatsGal}. Thus, that result provides
conditions under which the LSPG ROM has a lower local \textit{a posteriori}
error bound than the Galerkin ROM.}
\reviewerCNewer{Next,} we \reviewerA{specialize the \reviewerCNewer{global
\textit{a posteriori} error bounds} in Theorem~\ref{eq:error_multi_step}
to the case of the backward Euler scheme.}
\begin{corollary}[\reviewerA{\textit{a posteriori} error bounds:} backward Euler]\label{cor:backward_euler}
Under the assumptions of \thmref{theorem:localAPosterioriLMM}, for the backward Euler
scheme we obtain 
\begin{align}
    \normtwo{\deltastateRedG^{n}} 
&\le \dt  \sum_{j=0}^{n-1} \frac{1}{(h)^{j+1}} 
    \normtwo{ \del{ \matI  - \bbV } \f \del{\stateInitialNo + \podstate  \stateRedG^{n-j}} }  \label{eq:error_FOM_ROM_G_Backward_Euler} \\ 
    \normtwo{\deltastateRedP^{n}} 
&\le \dt  \sum_{j=0}^{n-1} \frac{1}{(h)^{j+1}} 
    \normtwo{ \del{ \matI  - \bbP^{n-j} } \f \del{\stateInitialNo + \podstate 
		\stateRedP^{n-j}} }
		\label{eq:error_FOM_ROM_P_Backward_Euler}\reviewerC{,}
\end{align}
where $h \defeq 1 - \lipschitzConstant  \dt $. \reviewerCNewer{Note that the time-step condition 
corresponds to $\dt < 1/
\lipschitzConstant$ in this case.}
\end{corollary}
\begin{proof}
 
\reviewerA{
Because it is a single-step method, the \reviewerCNewer{linear multistep} coefficients do not vary between
time \reviewerCNew{instances}; as a result, the constants appearing in Theorem
\ref{eq:error_multi_step} are also time-step independent and are $h=
1-\lipschitzConstant \dt$, $\firstCoeffMulti{0}
=
\dt/h$, $\firstCoeffMulti{1}=0$,
$\secondCoeffMulti{0} = (\lipschitzConstant \dt +
1)/h$, and $\secondCoeffMulti{1} =1/h$.
}
Substituting these values into error bound
\eqref{eq:error_FOM_ROM_P_multi_step} \reviewerA{
and noting that 
\begin{align*}
\mathcal A(j) =\begin{cases}
\{\left(1,\ldots,1\right)\in\RR{j}\}\quad &j>0\\
\emptyset\quad &\text{otherwise},
\end{cases}
\end{align*}
}
	yields
\reviewerA{
\begin{align}
    \normtwo{\deltastateRedP^{n}} 
&\le\frac{\dt}{h}\sum_{j= 0}^{n-1}\underbrace{\left[
\indicator_{\{0\}}(j) +(1/h)^{j}\indicator_{\nat{n-1}}(j)\right]}_{(1/h)^{j}}
\normtwo{ \del{ \matI  -
\bbP^{n-j } } \f \del{\stateInitialNo + \podstate  \stateRedP^{n-j}} },
\end{align}
	which simplifies to} bound \eqref{eq:error_FOM_ROM_P_Backward_Euler}.
	Derivation of bound \eqref{eq:error_FOM_ROM_G_Backward_Euler} is identical
	and is thus omitted.
\end{proof}

\reviewerANewer{We now specialize the results of the time-step-independent
global
\textit{a priori} error bounds in Corollary \ref{cor:timestepIndependenceAPriori} to the backward Euler scheme.
\begin{corollary}[Time-step-independent global \textit{a priori} error bounds:
backward Euler]\label{cor:backward_euler_exp}
Under the assumptions of Corollary
\ref{cor:timestepIndependenceAPriori}---which can be satisfied by backward
Euler as $k|\alpha^\star|
= |\alpha_0^\star|$ for this scheme---we obtain the following for the backward
Euler scheme:
\begin{align}
    \normtwo{\deltastateRedG^{n}} 
&\le 2\frac{\exp(t^n\lipschitzConstant\threshold^{-1} )-1}{\lipschitzConstant} 
      \max_{j\in\nat{n}} 
     \normtwo{ \del{ \matI  - \bbV } \f \del{\stateInitialNo + \state_\star^{j}} }  
     \label{eq:error_FOM_ROM_G_Backward_Euler_expCorollarypriori} \\ 
    \normtwo{\deltastateRedP^{n}} 
&\le 2\frac{\exp( t^n\lipschitzConstant\thresholdLSPG^{-1}\|\bbP_0^\star\|_2)-1}{\|\bbP_0^\star\|_2\lipschitzConstant} 
      \max_{j\in\nat{n}}
    \normtwo{ \del{ \matI  - \bbP^{j} } \f \del{\stateInitialNo + \state_\star
		^j } }.
		\label{eq:error_FOM_ROM_P_Backward_Euler_expCorollary}
\end{align}
Note that the time step conditions correspond to $\dt \leq
(1-\threshold)/\lipschitzConstant$ for the Galerkin ROM and
$\dt \leq (1-\thresholdLSPG)/(\lipschitzConstant\|\bbP_0^\star\|_2)$ for the LSPG ROM in this case.
\end{corollary}
\begin{proof}
For backward Euler, 
$k=|\alpha^\star|=|\alpha_0^\star|=|\beta_0^\star|=1$ and
$\barelltwostarmax=|\beta^\star|=0$.
Substituting these quantities into bounds
\eqref{eq:aprioriLMMboundSimplerGalTsindep} and
\eqref{eq:aprioriLMMboundSimplerPGTsindep} produces the desired result.
\end{proof}
}

\reviewerCNewer{We now derive a result that highlights the (surprising) role the time step
$\dt$ plays in the LSPG error bound. As will be shown in the numerical
experiments in Section \ref{s:numerics}, this theoretical results can have an
important effect on the performance of the LSPG ROM in practice.}
\begin{corollary}\label{corr:auxiliaryProblem}
If $\bar{\vec x}$ solves \reviewerA{an} auxiliary problem \reviewerA{that
computes the full-space solution increment} centered on the
\reviewerC{LSPG} ROM trajectory
\begin{equation}
\label{eq:auxODeltaE}
\vec {\bar x}^{j} = \dt  \f \del{\stateInitialNo +\vec {\bar x}^{j}} +
\podstate \stateRedP^{j-1},\quad j\in\nat{\timestepit},
\end{equation}
then the following holds:
\begin{align}\label{eq:auxToPlot}
\normtwo{\deltastateRedP^{n}} 
&\le (1 + \lipschitzConstant \dt)\sum_{j=0}^{n-1} \frac{\mu^{n-j}}{(h)^{j+1}}\\
\label{eq:auxToPlot2}
&=
\dt(1 + \lipschitzConstant \dt)\sum_{j=0}^{n-1}
\frac{\bar\mu^{n-j}}{(h)^{j+1}}\|\f(\bar\state^{n-j})\|_\reviewerA{2} 
.
\end{align}
Here, $\mu^{j} \defeq \normtwo{\podstate  \Delta  \stateRedP^{j} - \Delta \bar{\vec
x}^{j} }$ denotes the difference in solution increments at time instance $j$,
where 
$\Delta  \stateRedP^{j} \defeq \stateRedP^{j} - \stateRedP^{j-1}$ and 
$\Delta \bar{\vec x}^{j} \defeq \vec {\bar x}^{j} -
\podstate \stateRedP^{j-1}$.
We denote the relative solution increment at time instance $j$ by $\bar\mu^{j}
\defeq \mu^{j}/ \|\Delta \bar{\vec x}^{j} \|_\reviewerA{2}$.
\end{corollary}
\begin{proof}
Eq.~\eqref{eq:error_FOM_ROM_P_Backward_Euler} in conjunction with
\reviewerCNewer{\eqref{eq:lspgObliqueDisc} (with $\abs{\beta_0^n}=1$ and the
appropriate form of the discrete residual for backward Euler)} implies 
\begin{align}\label{eq:error_estimate_first}
\normtwo{\deltastateRedP^{n}} 
&\le \sum_{j=0}^{n-1} \frac{1}{(h)^{j+1}} 
    \normtwo{\podstate  \Delta  \stateRedP^{n-j} - \dt  \f \del{\stateInitialNo
		+\podstate  \stateRedP^{n-j}}  }.
\end{align}
We can also write the auxiliary equation \eqref{eq:auxODeltaE} as $ \Delta
\bar{\vec x}^{j} = \dt  \f \del{\vec {\bar x}^{j}}$,
$j\in\nat{\timestepit}$, which allows us to rewrite bound \eqref{eq:error_estimate_first} as 
\begin{align*}
\normtwo{\deltastateRedP^{n}} 
\le \sum_{j=0}^{n-1} \frac{1}{(h)^{j+1}}\cdot
    \Big\| &\del{\podstate  \Delta  \stateRedP^{n-j} - \Delta \bar{\vec
		x}^{n-j} } -\\
		&\dt \del{ \f \del{\stateInitialNo +\podstate  \Delta \stateRedP^{n-j} + \vec
		\Phi \hat{\vec x}_P^{n-j-1}} - \f \del{\stateInitialNo + \Delta \vec {\bar
		x}^{n-j} + \podstate \hat{\vec x}_P^{n-j-1}} } \Big\|_\reviewerA{2} .
\end{align*}
Lipschitz continuity of $\f $ leads to the bound \eqref{eq:auxToPlot}. To
obtain Eq.~\eqref{eq:auxToPlot2}, we multiply and divide by $\|\Delta
\bar{\vec x}^{n-j} \|_\reviewerA{2}$ for each term in the summation and use $ \Delta
\bar{\vec x}^{n-j} = \dt  \f \del{\vec {\bar x}^{n-j}}$.
\end{proof}

\reviewerA{Corollary \ref{corr:auxiliaryProblem} is useful in that it
expresses the \reviewerC{LSPG} ROM error in terms of the (time-local)
single-step errors incurred by projection along the \reviewerC{LSPG} ROM
trajectory. In addition, this result highlights the critical role of the time
step $\dt$ in the performance of the \reviewerC{LSPG} ROM; the following
remark provides this discussion.}

\begin{remark}\label{rem:modestTimestep} \rm
The time step $\dt$ in the error bound \eqref{eq:auxToPlot2} for the
\reviewerC{LSPG} ROM solution plays an important role. In particular,
decreasing the time step produces both beneficial effects (bound decrease)
and deleterious effects (bound increase), which we denote by `+' and `-',
respectively as follows:
\begin{enumerate} 
\item[+] The time-discretization error decreases (this does not appear in the
time-discrete error analysis above).
\item[-] The number of overall time \reviewerCNew{instances} $n$ increases,  so there are more
terms in the summation.
\item[+] The terms $\dt(1 + \lipschitzConstant \dt)$ and $1/{(h)^{j+1}}$ decrease.
\item[?] The term $\bar\mu^{n-j}$ may increase or decrease, depending on the
spectral content of the basis $\podstate$.
\end{enumerate}
We now discuss this final ambiguous effect.
The term $\bar \mu^n$ can be interpreted as the \emph{relative error in solution
increment over $\left[(n-1)\dt ,n\dt \right]$}.
Clearly, the ability of the \reviewerC{LSPG} ROM to make $\bar \mu^n$ small
depends on the spectral content of the basis $\podstate$: if the basis
only captures modes that evolve over long time scales, then $\bar \mu^n$
will be large (i.e., close to one), as the basis does not contain the `fast
evolving' solution components 
that change over a single time step.
This suggests that the time step should be `matched' to the spectral content of the
reduced basis $\podstate$. In Section \ref{sec:optimalDt} of the experiments, we explore this issue
numerically, and demonstrate that the error bound is minimized for an
intermediate value of the time step $\dt$.

We note that the above arguments do not hold for the Galerkin ROM, which is
simply an ODE that does not depend on the time step. Instead, decreasing
the time step should increase accuracy, as it has the effect of reducing
the time-discretization error.

\end{remark}

\subsection{Runge--Kutta schemes}\label{sec:errorRK}

We now derive Runge--Kutta error bounds for 
the Galerkin ROM~\eqref{eq:romResRK} and
the \reviewerC{LSPG} ROM~\eqref{eq:LSPGRK}.
\reviewerC{For notational simplicity, we define
\reviewerCNewer{$\f_i^n:\statevarf\mapsto\f(\statevarf,
t^{n-1}+c_i\dt)$}, $i\innat{s}$, $n\innat{T/\dt}$.} \reviewerCNewer{Rather
than present the full collection of results as was done for linear multistep
schemes in Section \ref{sec:errorLinear}, for compactness we instead focus
only on the most important results for Runge--Kutta schemes: global \textit{a
posteriori} and \textit{a priori} error bounds, as well as
time-step-independent \textit{a priori} error bounds.}

We rewrite Eqs.~\eqref{eq:resRKSolve}, 
\eqref{eq:romResRKSolve}, and \eqref{eq:LSPGRK} as 
\begin{align}
 \unknown^{n}_{\star,i}
&= \f ^n_i\Big(\stateInitialNo + \state_\star^{n-1}+\dt\sum_{j=1}^s a_{ij}\unknown_{\star,j}^n\Big),\quad
i\in\nat{s}\label{eq:fom_RK} \\
 \unknownRed^{n}_{G,i}
&= \podstate^T\f ^n_i\Big(\stateInitialNo + \podstate\stateRedG^{n-1}+\dt\sum_{j=1}^s a_{ij}\podstate\unknownRed_{G,j}^n \Big),\quad
i\in\nat{s}\label{eq:rom_RK}  \\
 \unknownRed^{n}_{P,i}
&= \del{(\testBasis_{ii}^n)^T \podstate}^{-1} (\testBasis_{ii}^n)^T \f ^n_i\Big(\stateInitialNo + \podstate\stateRedP^{n-1}+\dt\sum_{j=1}^s a_{ij}\podstate\unknownRed_{P,j}^n \Big)
\nonumber \\
& \ - \del{(\testBasis_{ii}^n)^T \podstate}^{-1} \sum_{e=1, e\neq i}^s
(\testBasis_{ie}^n)^T \bigg(\podstate \unknownRed^{n}_{P,e} - \f ^n_e\Big(\stateInitialNo + \podstate\stateRedP^{n-1}+\dt\sum_{j=1}^s a_{ej}\podstate\unknownRed_{P,j}^n \Big) \bigg) ,
\quad 
i\in\nat{s} \label{eq:rom_RK_D}
\end{align}
\reviewerCNewer{and set $\state_\star^0 = \state_G^0=\state_P^0=\zero$}.

\subsubsection{\reviewerCNewer{\textit{A posteriori} error bounds}}
\reviewerCNewer{We first derive \textit{a posteriori} error
bounds.}
\begin{theorem}[\reviewerA{Global \textit{a posteriori} error bounds: Runge--Kutta schemes}]\label{thm:RKerrors}
If $\del{\bf A_1}$ holds and $\dt$ is such that
\begin{enumerate}[$(a)$]
\item the matrix $\D \in\RR{s\times s}$ with entries $d_{ij}\defeq
\delta_{ij} - \lipschitzConstant \dt |a_{ij}|$ is invertible, and
\item for every $\vec x, \vec y \ge 0$, if $\D  \vec x \le \vec y$ then
$\vec x \le \D ^{-1} \vec y$,
\end{enumerate}
then
\begin{align}
\begin{split}
  \| \deltastateRedG^n \|_\reviewerA{2}
	\le \dt \sum_{\ell=0}^{n-1} &\Big(1 + \lipschitzConstant  \dt \sum_{k=1}^s |b_k| \sum_{i=1}^s[
	\D^{-1}]_{ki} \Big)^\ell\cdot\\
	&\Big( \sum_{k=1}^s |b_k| \sum_{i=1}^s [
	\D^{-1}]_{ki} 
  \Big\| (\matI  - \mathbb{V}) \f ^{n\reviewerC{-\ell}}_i\Big(\stateInitialNo +
	\podstate\stateRedG^{n-\ell-1} + \dt\sum_{j=1}^s
	a_{ij}\podstate\unknownRed_{G,j}^{n-\ell}\Big) \Big\|_\reviewerA{2}\Big)  \label{eq:RKG_est}  .
\end{split}
\end{align}
and 
\begin{multline}
  \| \deltastateRedP^n \|_2
	\le \dt \sum_{\ell=0}^{n-1} \Big(1 + \lipschitzConstant  \dt \sum_{k=1}^s |b_k| \sum_{i=1}^s[
	\D^{-1}]_{ki} \Big)^\ell\cdot\\
	\Bigg( \sum_{k=1}^s |b_k| \sum_{i=1}^s [
	\D^{-1}]_{ki} 
  \Big\| (\matI  - \bbP^{n-\ell}_i) \f ^{n-\ell}_i\Big(\stateInitialNo +
	\podstate\stateRedP^{n-\ell-1} + \dt\sum_{j=1}^s
	a_{ij}\podstate\unknownRed_{P,j}^{n-\ell}\Big) \Big\|_2 \\
   + \sum_{k=1}^s \abs{b_k} \sum_{i=1}^s [\D ^{-1}]_{ki} \bigg\| \podstate
	 \del{(\testBasis_{ii}^{n-\ell})^T \podstate}^{-1} \sum_{e=1, e\neq i}^s
	 (\testBasis_{ie}^{n-\ell})^T \bigg(\podstate \unknownRed^{n-\ell}_{P,e} -
	 \f ^{n-\ell}_e\Big(\stateInitialNo + \podstate\stateRedP^{n-\ell-1}+\dt\sum_{j=1}^s a_{ej}\podstate\unknownRed_{P,j}^{n-\ell} \Big) \bigg) \bigg\|_2
  \Bigg)  \label{eq:RKG_est_P}  ,
\end{multline}	
where $\bbP_i^n \defeq \podstate \del{(\testBasis_{ii}^n)^T \podstate}^{-1} (\testBasis_{ii}^n)^T$. 
\reviewerA{Here, inequalities applied to vectors hold entrywise.}
\end{theorem}
\begin{proof}
\underline{Galerkin ROM}.
First we will show bound \eqref{eq:RKG_est}. 
Subtracting Eq.~\eqref{eq:rom_RK} from Eq.~\eqref{eq:fom_RK} and applying the triangle
inequality yields
\[
  \| \delta \unknown_{G,i}^n \|_\reviewerA{2}
  \le \Big\| \f ^n_i\Big(\stateInitialNo + \state_\star^{n-1}+\dt\sum_{j=1}^s
	a_{ij}\unknown_{\star,j}^n\Big) - \mathbb{V} \f ^n_i\Big(\stateInitialNo +
	\podstate\stateRedG^{n-1}+\dt\sum_{j=1}^s a_{ij}\podstate\unknownRed_{G,j}^n \Big)
	\Big\|_\reviewerA{2},
  \qquad i \in \mathbb{N}(s) ,
\]
where $\delta \unknown_{G,i}^n \defeq \unknown_{\star,i}^n - \podstate\unknownRed_{G,i}^n$. 
Adding and subtracting $\f ^n_i\Big(\stateInitialNo +
\podstate\stateRedG^{n-1}+\dt\sum_{j=1}^s a_{ij}\podstate\unknownRed_{G,j}^n\Big)$
and invoking assumption $\del{\bf A_1}$, we deduce 
\[
  \| \delta \unknown_{G,i}^n \|_\reviewerA{2}
  - \lipschitzConstant  \dt \sum_{j=1}^s |a_{ij}| \| \delta \unknown^n_{G,j}
	\|_\reviewerA{2} 
  \le \Big\| (\matI  - \mathbb{V}) \f ^n_i\Big(\stateInitialNo +
	\podstate\stateRedG^{n-1}+\dt\sum_{j=1}^s
	a_{ij}\podstate\unknownRed_{G,j}^n\Big) \Big\|_\reviewerA{2}
  + \lipschitzConstant  \| \deltastateRedG^{n-1} \|_\reviewerA{2},
    \qquad i \in \mathbb{N}(s) . 
\]  
Selecting $\dt$ small enough such that $(a)$ and $(b)$ hold yields
\begin{equation}\label{eq:etaUpperBound}
  \| \delta \unknown_{G,k}^{n} \|_\reviewerA{2}  
  \le \sum_{i=1}^s [\D ^{-1}]_{ki} \Big\| (\matI  - \mathbb{V}) 
	\f^n_i\Big(\stateInitialNo + \podstate\stateRedG^{n-1}+\dt\sum_{j=1}^s
	a_{ij}\podstate\unknownRed_{G,j}^n\Big) \Big\|_\reviewerA{2}
  + \lipschitzConstant  \| \deltastateRedG^{n-1} \|_\reviewerA{2}\sum_{i=1}^s[\D ^{-1}]_{ki} ,
\end{equation}
where $[\cdot]_{ij}$ denotes entry $(i,j)$ of the argument.
From explicit state updates \eqref{eq:fom_update} and
\eqref{eq:romG_RK_update}, we obtain 
\[
  \| \deltastateRedG^n \|_\reviewerA{2} \le \| \deltastateRedG^{n-1}
	\|_\reviewerA{2}
  + \dt \sum_{k=1}^s |b_k| \| \delta \unknown_{G,k}^n \|_\reviewerA{2}.
\]
Using upper bound \reviewerCNewer{\eqref{eq:etaUpperBound}}  yields
\begin{align*}
 \| \deltastateRedG^n \|_\reviewerA{2} \le 
&  \Big(1 + \lipschitzConstant  \dt \sum_{k=1}^s |b_k|
\sum_{i=1}^s[\D ^{-1}]_{ki} \Big) \| \deltastateRedG^{n-1}\|_\reviewerA{2}
\\
	&+ \dt \sum_{k=1}^s |b_k| \sum_{i=1}^s [\D ^{-1}]_{ki} \Big\| (\matI  -
	\mathbb{V}) \f ^n_i\Big(\stateInitialNo +
	\podstate\stateRedG^{n-1}+\dt\sum_{j=1}^s
	a_{ij}\podstate\unknownRed_{G,j}^n\Big) \Big\|_\reviewerA{2}.
\end{align*}
Finally, an induction argument produces the desired result \eqref{eq:RKG_est}.

\noindent
\underline{LSPG ROM}. We now prove bound \eqref{eq:RKG_est_P}.
Subtracting \eqref{eq:rom_RK_D} from
\eqref{eq:fom_RK} and applying the triangle inequality yields
\begin{align*}
  \| \delta \unknown_{P,i}^n \|_2
  &\le \Big\| \f ^n_i\Big(\stateInitialNo + \state_\star^{n-1}+\dt\sum_{j=1}^s
	a_{ij}\unknown_{\star,j}^n\Big) - \bbP^n_i \f ^n_i\Big(\stateInitialNo +
	\podstate\stateRedP^{n-1}+\dt\sum_{j=1}^s a_{ij}\podstate\unknownRed_{P,j}^n \Big)
	\Big\|_2 \\
  &\quad + \bigg\| \podstate \del{(\testBasis_{ii}^n)^T \podstate}^{-1}
	\sum_{e=1, e\neq i}^s (\testBasis_{ie}^n)^T \bigg(\podstate
	\unknownRed^{n}_{P,e} - \f ^n_e\Big(\stateInitialNo + \podstate\stateRedP^{n-1}+\dt\sum_{j=1}^s a_{ej}\podstate\unknownRed_{P,j}^n \Big) \bigg) \bigg\|_2 ,
  \qquad i \in \mathbb{N}(s) .
\end{align*}
Adding and subtracting $\f ^n_i\Big(\stateInitialNo + 
\podstate\stateRedP^{n-1}+\dt\sum_{j=1}^s a_{ij}\podstate\unknownRed_{P,j}^n\Big)$
and invoking assumption $\del{\bf A_1}$, for $i \in \mathbb{N}(s)$ we deduce 
\begin{multline*}
\| \delta \unknown_{P,i}^n \|_2
  - \lipschitzConstant  \dt \sum_{j=1}^s |a_{ij}| \| \delta \unknown^n_{P,j}
	\|_2 
  \le \Big\| (\matI  - \bbP^n_i) \f ^n_i\Big(\stateInitialNo +
	\podstate\stateRedP^{n-1}+\dt\sum_{j=1}^s
	a_{ij}\podstate\unknownRed_{P,j}^n\Big) \Big\|_2
  + \lipschitzConstant  \| \deltastateRedP^{n-1} \|_2 \\
   + \bigg\| \podstate \del{(\testBasis_{ii}^n)^T \podstate}^{-1} \sum_{e=1,
	 e\neq i}^s (\testBasis_{ie}^n)^T \bigg(\podstate \unknownRed^{n}_{P,e} -
	 \f ^n_e\Big(\stateInitialNo + \podstate\stateRedP^{n-1}+\dt\sum_{j=1}^s a_{ej}\podstate\unknownRed_{P,j}^n \Big) \bigg) \bigg\|_2 . 
\end{multline*}
Again selecting $\dt$ small enough such that $(a)$ and $(b)$ hold yields
\begin{multline}\label{eq:zetaUpperBound}
  \| \delta \unknown_{P,k}^{n} \|_2  
  \le \sum_{i=1}^s [\D ^{-1}]_{ki} \Big\| (\matI  - \bbP^n_i) 
	\f^n_i\Big(\stateInitialNo + \podstate\stateRedP^{n-1}+\dt\sum_{j=1}^s
	a_{ij}\podstate\unknownRed_{P,j}^n\Big) \Big\|_2
  + \lipschitzConstant  \| \deltastateRedP^{n-1} \|_2\sum_{i=1}^s[\D ^{-1}]_{ki} \\
  + \sum_{i=1}^s [\D ^{-1}]_{ki} \bigg\| \podstate
	\del{(\testBasis_{ii}^n)^T \podstate}^{-1} \sum_{e=1, e\neq i}^s
	(\testBasis_{ie}^n)^T \bigg(\podstate \unknownRed^{n}_{P,e} - \f ^n_e\Big(\stateInitialNo + \podstate\stateRedP^{n-1}+\dt\sum_{j=1}^s a_{ej}\podstate\unknownRed_{P,j}^n \Big) \bigg) \bigg\|_2 . 
\end{multline}
The explicit state updates from \eqref{eq:fom_update} and
\eqref{eq:romG_RK_update} and the bound
\reviewerCNewer{\eqref{eq:zetaUpperBound}} yield
\begin{align*}
\begin{split}
 \| \deltastateRedP^n \|_2 \le 
&\Big(1 + \lipschitzConstant  \dt \sum_{k=1}^s |b_k|
\sum_{i=1}^s[\D ^{-1}]_{ki} \Big) \| \deltastateRedP^{n-1}\|_2 \\
	&+\dt \sum_{k=1}^s |b_k| \sum_{i=1}^s [\D ^{-1}]_{ki} \Big\| (\matI  -
	\bbP^n_i) \f ^n_i\Big(\stateInitialNo +
	\podstate\stateRedP^{n-1}+\dt\sum_{j=1}^s
	a_{ij}\podstate\unknownRed_{P,j}^n\Big) \Big\|_2   \\
 &+ \dt \sum_{k=1}^s \abs{b_k} \sum_{i=1}^s [\D ^{-1}]_{ki} \cdot\\
 &\quad \bigg\|
 \podstate \del{(\testBasis_{ii}^n)^T \podstate}^{-1} \sum_{e=1, e\neq i}^s
 (\testBasis_{ie}^n)^T \bigg(\podstate \unknownRed^{n}_{P,e} - \f ^n_e\Big(\stateInitialNo + \podstate\stateRedP^{n-1}+\dt\sum_{j=1}^s a_{ej}\podstate\unknownRed_{P,j}^n \Big) \bigg) \bigg\|_2 . 
\end{split}
\end{align*}
An induction argument yields the bound \eqref{eq:RKG_est_P}.
\end{proof}

As with linear multistep schemes, the error bound for the Galerkin ROM in
\eqref{eq:RKG_est} depends on the orthogonal projection error of $\f ^n_i$
onto $\range{\podstate}$, while the LSPG ROM error bound depends on an oblique
projection; however, because the oblique projector depends on the LSPG ROM
solution, this bound can be smaller than the Galerkin bound (as was
demonstrated in Corollary \ref{cor:discreteOptBeatsGal}---note that backward
Euler is also a Runge--Kutta scheme). Further, notice that the complex `path
traversing' that appears in the linear multistep error bounds is not present for
Runge--Kutta schemes; this is a consequence of the fact that previous time
\reviewerCNew{instances} only influence the error at the current time step through induction for
Runge--Kutta schemes.
In addition, note that the LSPG ROM error bound is more complex than the Galerkin
bound; the final line in bound \eqref{eq:RKG_est_P} is a consequence of the
test-basis coupling \eqref{eq:discreteOptRKTest} for
general implicit Runge--Kutta schemes. Finally, we point out that both bounds
grow exponentially in time due to the amplification factor. We now present a simpler version of
this error bound for explicit Runge--Kutta and DIRK schemes.

\reviewerA{
\begin{corollary}[\reviewerCNewer{Global} \textit{a posteriori}
\reviewerCNewer{LSPG ROM} error bound: explicit Runge--Kutta and DIRK]\label{cor:DIRK_error}
Under the assumptions of Theorem~\ref{thm:RKerrors} for explicit RK $(\theta =
1)$ and DIRK $(\theta = 0)$ schemes, we have
\begin{multline}
  \| \deltastateRedP^n \|_2
	\le \dt \sum_{\ell=0}^{n-1} \Big(1 + \lipschitzConstant  \dt \sum_{k=1}^s |b_k| \sum_{i=1}^s[
	\D^{-1}]_{ki} \Big)^\ell\cdot\\
	\Bigg( \sum_{k=1}^s |b_k| \sum_{i=1}^s [
	\D^{-1}]_{ki} 
  \Big\| (\matI  - \bbP^{n-\ell}_i) \f ^{n-\ell}_i\Big(\stateInitialNo +
	\podstate\stateRedP^{n-\ell-1} +  \dt\sum_{j=1}^{i-\theta}
	a_{ij}\podstate\unknownRed_{P,j}^{n-\ell}\Big) \Big\|_2  \Bigg). \label{eq:RKG_est_P_DIRK}
\end{multline}	
\end{corollary}
\begin{proof}
For explicit and diagonally implicit Runge--Kutta schemes
\eqref{eq:LSPGRKexplicit}, we have $\testBasis_{ij}^n = 0$ when $i \neq j$.
The proof is then an immediate consequence of \eqref{eq:RKG_est_P}. 
\end{proof}
}

\reviewerA{
Note that the LSPG error bound \eqref{eq:RKG_est_P_DIRK} resembles the
Galerkin error bound \eqref{eq:RKG_est} much more closely than the previous
LSPG bound \eqref{eq:RKG_est_P}, as explicit Runge--Kutta and DIRK schemes
remove the coupling of the test basis across stages.
}

\subsubsection{\reviewerCNewer{\textit{A priori} error bounds}}
We now state the \textit{a priori} versions of the Galerkin 
Runge--Kutta schemes \eqref{eq:RKG_est} and the LSPG Runge--Kutta schemes \eqref{eq:RKG_est_P}.
\begin{theorem}[\reviewerCNewer{Global} \textit{a priori} error bounds: Runge--Kutta schemes]\label{thm:RKerrors_apriori}
If $\del{\bf A_1}$ holds and $\dt$ is such that
\begin{enumerate}[$(a)$]
\item the matrices $\D \in\RR{s\times s}$ (defined in
Theorem~\ref{thm:RKerrors}) and $\bar{\D }^n\in\RR{s\times
s}$ with entries $[\bar{\D }^n]_{ij}\defeq
\delta_{ij} - \lipschitzConstant \dt |a_{ij}| \| \bbP^n_i \|_2$ are invertible, and
\item for every $\vec x, \vec y \ge 0$, if $\D  \vec x \le \vec y$ then
$\vec x \le \D ^{-1} \vec y$ and 
if $\bar{\D }^m \vec x \le \vec y$ then
$\vec x \le [\bar{\D }^m]^{-1} \vec y$, $m\innat{n}$
,
\end{enumerate}
then
\begin{align}
\begin{split}
  \| \deltastateRedG^n \|_2
	\le \dt \sum_{\ell=0}^{n-1} &\Big(1 + \lipschitzConstant  \dt \sum_{k=1}^s |b_k| \sum_{i=1}^s[
	\D^{-1}]_{ki} \Big)^\ell\cdot\\
	&\Big( \sum_{k=1}^s |b_k| \sum_{i=1}^s [
	\D^{-1}]_{ki} 
  \Big\| (\matI  - \mathbb{V}) \f ^{n-\ell}_i\Big(\stateInitialNo +
	\state_\star^{n-\ell-1} + \dt\sum_{j=1}^s
	a_{ij}\unknown_{\star,j}^{n-\ell}\Big) \Big\|_2\Big)  \label{eq:RKG_est_apriori}  .
\end{split}
\end{align}
and 
\begin{align}
\begin{split}
  \| \deltastateRedP^n \|_2
	\le &\dt \sum_{\ell=0}^{n-1} \prod_{m = 0}^{\ell-1}\Big(1 +
	\lipschitzConstant  \dt \sum_{k=1}^s |b_k| \sum_{i=1}^s[[\bar{
	\D}^{n-m}]^{-1}]_{ki} \Big)\cdot\\
	&\Bigg( \sum_{k=1}^s |b_k| \sum_{i=1}^s [[\bar{
	\D}^{n-\ell}]^{-1}]_{ki} 
  \Big\| (\matI  - \bbP^{n-\ell}_i) \f ^{n-\ell}_i\Big(\stateInitialNo +
	\state_\star^{n-\ell-1} + \dt\sum_{j=1}^s
	a_{ij}\unknown_{\star,j}^{n-\ell}\Big) \Big\|_2 \\
   &+ \sum_{k=1}^s \abs{b_k} \sum_{i=1}^s [[\bar{
	 \D}^{n-\ell}]^{-1}]_{ki}\cdot\\
	 &\bigg\| \podstate
	 \del{(\testBasis_{ii}^{n-\ell})^T \podstate}^{-1} \sum_{e=1, e\neq i}^s
	 (\testBasis_{ie}^{n-\ell})^T \bigg(\podstate \unknownRed^{n-\ell}_{P,e} -
	 \f ^{n-\ell}_e\Big(\stateInitialNo + \podstate\stateRedP^{n-\ell-1}+\dt\sum_{j=1}^s a_{ej}\podstate\unknownRed_{P,j}^{n-\ell} \Big) \bigg) \bigg\|_2
  \Bigg)  \label{eq:RKG_apriori},
\end{split}	
\end{align}	
where we have used the convention that the empty product is equal to one.
\end{theorem}
\begin{proof}
The proof of \eqref{eq:RKG_est_apriori} follows the Galerkin-ROM derivation in Theorem~\ref{thm:RKerrors}, wherein the quantity $\bbV\f ^n_i\Big(
\stateInitialNo +
	\state_\star^{n-1} + 
	\dt\sum_{j=1}^{s}
	a_{ij}\unknown_{\star,j}^{n}\Big)$ is added and
subtracted rather than $\f ^n_i\Big(\stateInitialNo + 
\podstate\stateRedG^{n-1}+\dt\sum_{j=1}^s
a_{ij}\podstate\unknownRed_{G,j}^n\Big)$ and $\|\bbV\|_2 = 1$ is used. On the
other hand, \eqref{eq:RKG_apriori} follows the LSPG-ROM derivation by
adding and subtracting $\bbP^n_i\f ^n_i\Big(
\stateInitialNo +
	\state_\star^{n-1} + 
	\dt\sum_{j=1}^{s}
	a_{ij}\unknown_{\star,j}^{n}\Big)$
instead of 
$\f ^n_i\Big(\stateInitialNo + 
\podstate\stateRedP^{n-1}+\dt\sum_{j=1}^s a_{ij}\podstate\unknownRed_{P,j}^n\Big)$.
\end{proof}

\reviewerANewer{
Similarly to Corollary~\ref{cor:backward_euler_exp}, we now derive a
time-step-independent variant of the error bound \eqref{eq:RKG_est_apriori}. A
similar result can be shown for the LSPG bound \eqref{eq:RKG_apriori};
however, we omit this result for simplicity and instead will provide (more
readily interpretable) \textit{a priori} LSPG ROM error bounds for explicit Runge--Kutta
and DIRK schemes in Corollary \ref{thm:LSPGDirkRK_apriori}.%
\begin{corollary}[Time-step-independent global \textit{a priori} Galerkin ROM error bound:
RK schemes]\label{thm:GalRK_apriori}
Under the assumptions of Theorem~\ref{thm:RKerrors_apriori}, assume
additionally $\dt
\leq(1-\thresholdRK)/(\lipschitzConstant\astar)$  with
$\astar\defeq\|[|a_{ij}|]_{ij}\|_2$ and $0 < \thresholdRK < 1$.
Then,
\begin{equation}\label{eq:RKG_est_apriori_exp}
\begin{aligned}
   \| \deltastateRedG^n \|_2 \le 
    \bigg(\frac{\exp( t^n\lipschitzConstant \thresholdRK^{-1}\bmax s^{3/2}) -
		1}{\lipschitzConstant} 
    \bigg)
    \Big(\max_{\substack{\ell\in\natZero{n-1}\\ i\in\nat{s}}} 
    \Big\| (\matI  - \mathbb{V}) \f ^{n-\ell}_i\Big(\stateInitialNo +
	\state_\star^{n-\ell-1} + \dt\sum_{j=1}^s
	a_{ij}\unknown_{\star,j}^{n-\ell}\Big) \Big\|_2 \Big),
\end{aligned}	
\end{equation}
where 
$\bmax \defeq \max_{i\in\nat{s}}|b_i|$.
\end{corollary}
\begin{proof}
As in Corollary~\ref{cor:backward_euler_exp}, by using an upper bound for the right hand side in \eqref{eq:RKG_est_apriori} we obtain
\begin{equation}\label{eq:RKG_est_apriori_aux1Kev}
\begin{aligned}
    \| \deltastateRedG^n \|_2 \le \dt \sum_{\ell=0}^{n-1} &\Big(1 + \lipschitzConstant  \dt \sum_{k=1}^s |b_k| 
    \sum_{i=1}^s[\D ^{-1}]_{ki} \Big)^\ell\cdot 
	\Big(  \sum_{k=1}^s |b_k| \sum_{i=1}^s [\D^{-1}]_{ki} \Big) \cdot \\
    &\Big(\max_{\substack{\ell\in\natZero{n-1}\\i\in\nat{s}}} 
    \Big\| (\matI  - \mathbb{V}) \f ^{n-\ell}_i\Big(\stateInitialNo +
	\state_\star^{n-\ell-1} + \dt\sum_{j=1}^s
	a_{ij}\unknown_{\star,j}^{n-\ell}\Big) \Big\|_2 \Big)
\end{aligned}	
\end{equation}
We now derive a bound for the term $\sum\limits_{k=1}^s |b_k| \sum\limits_{i=1}^s
[\D^{-1}]_{ki}$ as
\begin{align}\label{eq:boundDterm}
\begin{split}
\sum_{k=1}^s |b_k| \sum_{i=1}^s [\D^{-1}]_{ki} 
&\leq \bmax \sum_{k=1}^s\sum_{i=1}^s [\D^{-1}]_{ki}
\leq \bmax \sum_{k=1}^s\sum_{i=1}^s |[\D^{-1}]_{ki}|
\leq\bmax s\|\D^{-1}\|_F\leq\bmax s^{3/2}\|\D^{-1}\|_2\\
&\leq\frac{\bmax s^{3/2}}{\sigma_\text{min}(\matI) -
\sigma_\text{max}(\lipschitzConstant\dt[|a_{ij}|]_{ij})}
=\frac{\bmax s^{3/2}}{1 - 
\lipschitzConstant\dt\astar},
\end{split}
\end{align}
where we have used the vector-norm equivalence relation 
 $\|\state\|_1\leq \sqrt{\ndof}\|\state\|_2$ with $\state\in\RR{\ndof}$, the
matrix-norm equivalence relation 
$\|\weightingMatrix\|_F\leq\sqrt{\ndof}\|\weightingMatrix\|_2$  with
$\weightingMatrix\in\RR{\ndof\times\ndof}$, 
the relation 
$\|\weightingMatrix^{-1}\|_2 = 1/\sigma_\text{min}(\weightingMatrix)$,
$\sigma_\text{min}(\A + \B) \geq \sigma_\text{min}(\A) -
\sigma_\text{max}(\B) = \sigma_\text{min}(\A) -
\|\B\|_2$ (with $\A = \matI $ and $\B =
-\lipschitzConstant\dt[|a_{ij}|]_{ij})$), and the
assumption $\dt \leq(1-\thresholdRK)/(\lipschitzConstant\astar)$.

We can now compute an upper bound on the constant in inequality
\eqref{eq:RKG_est_apriori_aux1Kev}
as 
\begin{align}
    \dt \sum_{\ell=0}^{n-1} \Big(1 + \lipschitzConstant  & \dt \sum_{k=1}^s |b_k| 
    \sum_{i=1}^s[\D ^{-1}]_{ki} \Big)^\ell\cdot 
	\Big( \sum_{k=1}^s |b_k| \sum_{i=1}^s [\D^{-1}]_{ki} \Big) \nonumber \\
	&\le \dt \sum_{\ell=0}^{n-1} \Bigg(1 +  \frac{\lipschitzConstant\dt \bmax
	s^{3/2}}{1 - 
\lipschitzConstant\dt\astar} 
	  \Bigg)^\ell \Big(\frac{\bmax
	s^{3/2}}{1 - 
\lipschitzConstant\dt\astar} \Big) \label{eq:RKG_est_apriori_aux2KevAlmost}\\
	 &
	 = \frac{\bigg(1+\lipschitzConstant\dt(\bmax s^{3/2}-\astar)
	 \bigg)^n (1 - \lipschitzConstant \dt \astar)^{-n} -
	 1}{\lipschitzConstant}
	 \label{eq:RKG_est_apriori_aux2Kev}
\end{align}
Then, setting $t^n = n \dt$ we obtain
\begin{equation}\label{eq:RKG_est_apriori_aux3Kev}
    \bigg( 1 +  \lipschitzConstant \dt \Big(\bmax s^{3/2}-\astar\Big) \bigg)^n
		(1 - \lipschitzConstant \dt \astar)^{-n} 
    \le \exp(t^n\lipschitzConstant \thresholdRK^{-1} \bmax s^{3/2})
\end{equation}
As in Theorem \ref{theorem:localAPosterioriLMM}, we have used
the relation $(1 + x)^n\leq \exp(nx)$, and the result (with $x =  1 +  \lipschitzConstant \dt \Big(\bmax s^{3/2}-\astar\Big) $ and $y=1 - 
\lipschitzConstant\dt\astar$) that states if $x\geq y$, then $(x-y)/y\leq
\thresholdRK^{-1} (x-y)$
if and only if
$y \geq \thresholdRK > 0$.
Finally, substituting in \eqref{eq:RKG_est_apriori_aux3Kev} in
\eqref{eq:RKG_est_apriori_aux2Kev} and combining the resulting expression with
\eqref{eq:RKG_est_apriori_aux1Kev} yields the desired result. 
\end{proof}  
}

As with linear multistep schemes, the rightmost term in the
Galerkin \textit{a priori} bound will always be smaller than that for the LSPG
bound, as the former associates with an orthogonal projection error of a
fixed vector. In addition, the LSPG bound depends on the LSPG ROM solution;
while this dependence could be removed, the bound in its current form
	facilitates comparison with the Galerkin bound.
We again notice the complex structure of the estimator in \eqref{eq:RKG_apriori} compared to \eqref{eq:RKG_est_apriori}. To better understand the behavior the LSPG estimator in \eqref{eq:RKG_apriori} we consider two subcases: explicit Runge--Kutta and DIRK schemes. 

%
\begin{corollary}[\reviewerCNewer{Global} \textit{a priori}
\reviewerCNewer{LSPG ROM} error bounds: explicit Runge--Kutta and DIRK]\label{cor:RKerrors_apriori_explicit}
Under the assumptions of Theorem~\ref{thm:RKerrors_apriori} for explicit RK $(\theta =
1)$ and DIRK $(\theta = 0)$ schemes, we have 
\begin{align}
\begin{split}
	  \| \deltastateRedP^n \|_2
	\le \dt \sum_{\ell=0}^{n-1} &
	\prod_{m = 0}^{\ell-1}\Big(1 +
	\lipschitzConstant  \dt \sum_{k=1}^s |b_k| \sum_{i=1}^s[[\bar{
	\D}^{n-m}]^{-1}]_{ki} \Big)
	\cdot\\
	&\Bigg( \sum_{k=1}^s |b_k| \sum_{i=1}^s [[\bar{
	\D}^{n-\ell}]^{-1}]_{ki} 
  \Big\| (\matI  - \bbP^{n-\ell}_i) \f ^{n-\ell}_i\Big(\stateInitialNo +
	\state_\star^{n-\ell-1} + 
	\dt\sum_{j=1}^{i-\theta}
	a_{ij}\unknown_{\star,j}^{n-\ell}\Big) \Big\|_2  \Bigg)  \label{eq:RKG_est_P_apriori} . 
\end{split}
\end{align}	
\end{corollary}
\begin{proof}
For explicit and diagonally implicit Runge--Kutta schemes
\eqref{eq:LSPGRKexplicit}, we have $\testBasis_{ij}^n = 0$ when $i \neq j$.
The proof is then an immediate consequence of \eqref{eq:RKG_apriori}.  
\end{proof}

Owing to the fact that the explicit Runge--Kutta and DIRK schemes removes the
coupling of the basis, we again notice that the LSPG error bound
\eqref{eq:RKG_est_P_apriori} resembles the Galerkin ROM error bound
\eqref{eq:RKG_est_apriori}. 

\reviewerANewer{
\begin{corollary}[Time-step-independent global \textit{a priori} LSPG ROM error bound:
explicit RK and DIRK]\label{thm:LSPGDirkRK_apriori}
Under the assumptions of Theorem~\ref{thm:RKerrors_apriori} for explicit RK $(\theta =
1)$ and DIRK $(\theta = 0)$ schemes, assume
additionally $\dt
\leq(1-\thresholdRKLSPG)/(\lipschitzConstant\barastar)$  with
$\barastar\defeq\min_{\ell\in\natZero{n-1}}\|[|a_{ij}|\|\bbP_i^\ell\|_2]_{ij}\|_2$
and $0 < \thresholdRKLSPG < 1$.
Then,
\begin{equation}\label{eq:RKG_est_apriori_exp}
\begin{aligned}
   \| \deltastateRedP^n \|_2 \le 
    \bigg(\frac{\exp(t^n\lipschitzConstant \thresholdRKLSPG^{-1}\bmax s^{3/2})
		- 1}{\lipschitzConstant} 
    \bigg)
    \Big(\max_{\substack{\ell\in\natZero{n-1}\\i\in\nat{s}}}
  \Big\| (\matI  - \bbP^{n-\ell}_i) \f ^{n-\ell}_i\Big(\stateInitialNo +
	\state_\star^{n-\ell-1} + 
	\dt\sum_{j=1}^{i-\theta}
	a_{ij}\unknown_{\star,j}^{n-\ell}\Big) \Big\|_2\Big).
\end{aligned}	
\end{equation}
\end{corollary}
\begin{proof}
The proof follows closely that of Corollary \ref{thm:GalRK_apriori}. We begin
by deriving an upper bound for the right hand side in
\eqref{eq:RKG_est_apriori} as
\begin{equation}\label{eq:RKLSPG_est_apriori_aux1Kev}
\begin{aligned}
	  \| \deltastateRedP^n \|_2\leq
\dt \sum_{\ell=0}^{n-1} &
	\prod_{m = 0}^{\ell-1}\Big(1 +
	\lipschitzConstant  \dt \sum_{k=1}^s |b_k| \sum_{i=1}^s[[\bar{
	\D}^{n-m}]^{-1}]_{ki} \Big)
	\Bigg( \sum_{k=1}^s |b_k| \sum_{i=1}^s [[\bar{
	\D}^{n-\ell}]^{-1}]_{ki}   \Bigg)\cdot\\
	&
\Big(\max_{\substack{\ell\in\natZero{n-1}\\i\in\nat{s}}}
  \Big\| (\matI  - \bbP^{n-\ell}_i) \f ^{n-\ell}_i\Big(\stateInitialNo +
	\state_\star^{n-\ell-1} + 
	\dt\sum_{j=1}^{i-\theta}
	a_{ij}\unknown_{\star,j}^{n-\ell}\Big) \Big\|_2\Big). 
\end{aligned}	
\end{equation}
Analogously to the derivation in Eq.~\eqref{eq:boundDterm}, we bound the term 
$\sum_{k=1}^s |b_k| \sum_{i=1}^s [[\bar{
	\D}^{n-\ell}]^{-1}]_{ki} $
as
\begin{align}
\sum_{k=1}^s |b_k| \sum_{i=1}^s [[\bar{
	\D}^{n-\ell}]^{-1}]_{ki} 
\leq\frac{\bmax s^{3/2}}{1 - 
\lipschitzConstant\dt\barastar}.
\end{align}
Finally, we compute an upper bound on the constant in inequality
\eqref{eq:RKLSPG_est_apriori_aux1Kev}
as 
\begin{align}
    \dt \sum_{\ell=0}^{n-1}
	\prod_{m = 0}^{\ell-1}&\Big(1 +
	\lipschitzConstant  \dt \sum_{k=1}^s |b_k| \sum_{i=1}^s[[\bar{
	\D}^{n-m}]^{-1}]_{ki} \Big)
	\cdot
	\Bigg( \sum_{k=1}^s |b_k| \sum_{i=1}^s [[\bar{
	\D}^{n-\ell}]^{-1}]_{ki}   \Bigg)
 \nonumber
	&\le \dt \sum_{\ell=0}^{n-1} \Bigg(1 +  \frac{\lipschitzConstant\dt \bmax
	s^{3/2}}{1 - 
\lipschitzConstant\dt\barastar} 
	  \Bigg)^\ell \Big(\frac{\bmax
	s^{3/2}}{1 - 
\lipschitzConstant\dt\barastar} \Big) \nonumber 
	 \label{eq:RKG_est_apriori_aux2Kev}
\end{align}
As this result is identical to upper bound
\eqref{eq:RKG_est_apriori_aux2KevAlmost} in the proof of Corollary
\ref{thm:GalRK_apriori} with $\barastar$
replacing $\astar$, we obtain the desired result by applying the
remaining steps as Corollary \ref{thm:GalRK_apriori}.
\end{proof}  
	}



\section{Numerical experiments}\label{s:numerics}
This section compares the performance of Galerkin and \reviewerA{LSPG} ROMs
on a \reviewerC{computational-fluid-dynamics (CFD)} application using a basis constructed by proper orthogonal
decomposition. These experiments highlight the importance of the
previous analyses, in particular the limiting equivalence of Galerkin and
\reviewerA{LSPG} ROMs (Theorem~\ref{thm:dtZeroEquiv}), superior accuracy of
the \reviewerA{LSPG} ROM compared with the Galerkin ROM (Corollary~\ref{cor:discreteOptBeatsGal}),
and performance improvement of the \reviewerA{LSPG} ROM when an intermediate time
step is selected (Corollary~\ref{corr:auxiliaryProblem} and
Remark~\ref{rem:modestTimestep}). 

\reviewerB{Note that these experiments could
be carried out on any dynamical system yielding a system of nonlinear ODEs
\eqref{eq:ODE}; we have selected compressible turbulent fluid dynamics due to
both its wide interest and challenging nature: limited progress has been made
to date on developing robust, accurate ROMs for such problems. The numerical
experiments highlight this fact, as standard Galerkin ROMs generate unstable
responses in all cases.}

\subsection{Problem description}\label{sec:probDesc}

The Galerkin and \reviewerA{LSPG} ROMs are implemented in AERO-F
~\citep{geuzaine2003aeroelastic,farhat2003application}, a massively parallel
compressible-flow solver.  AERO-F solves the steady or unsteady compressible
Navier--Stokes equations with various closure models available for turbulent
flow, and employs a second-order node-centered finite-volume scheme. For
model-reduction algorithms, all linear least-squares problems and singular
value decompositions are computed in parallel using the ScaLAPACK library
\citep{scalapack}.  

The full-order model corresponds to an unsteady Navier--Stokes simulation of a
two-dimensional open cavity \reviewerCNew{with a length-to-depth ratio of 4.5} using AERO-F's DES turbulence model (based on the
Spalart--Allmaras one-equation model \reviewerA{\cite{spalartAllmaras}}) and a
wall-function boundary condition applied on solid surface boundaries.  The
fluid domain is discretized by a mesh with 192,816 nodes and 573,840
tetrahedra (Figure~\ref{fig:fluidDomain}).  The two-dimensional geometry is
discretized in three dimensions by considering a slab of thin, but finite
thickness, in the $z$-direction\reviewerA{; the resulting grid is one element
wide and is created by extruding a distance that is 1\% of the cavity length.}
The viscosity is assumed to be constant, and the Reynolds number based on
cavity length is \reviewerCNew{$2.97\times 10^6$}, while the free-stream Mach number is 0.6.
Due to the turbulence model and three-dimensional domain, the number of
conservation equations per node is $6$, and therefore the dimension of the CFD
model is $\ndof=1,156,896$. Roe's scheme is employed to discretize the
convective fluxes, and a linear variation of the solution is assumed within
each control volume, which leads to a second-order space-accurate scheme
\reviewerANew{on general unstructured, multi-dimensional meshes; however, we
employ an extended stencil that gives fifth-order formal order of accuracy
\reviewerA{(with uniform mesh spacing)} on inviscid, one-dimensional
problems.} \reviewerA{The viscous flux is discretized using a centered
Galerkin scheme.}
\begin{figure}[htbp] 
\centering 
\subfigure[Full domain]{
\includegraphics[width=\textwidth]{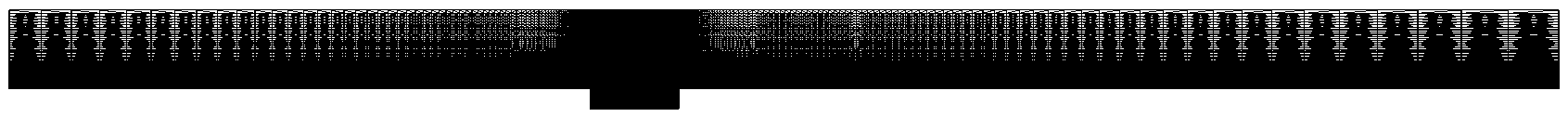} }
\subfigure[Detail around cavity]{
\includegraphics[width=0.6\textwidth]{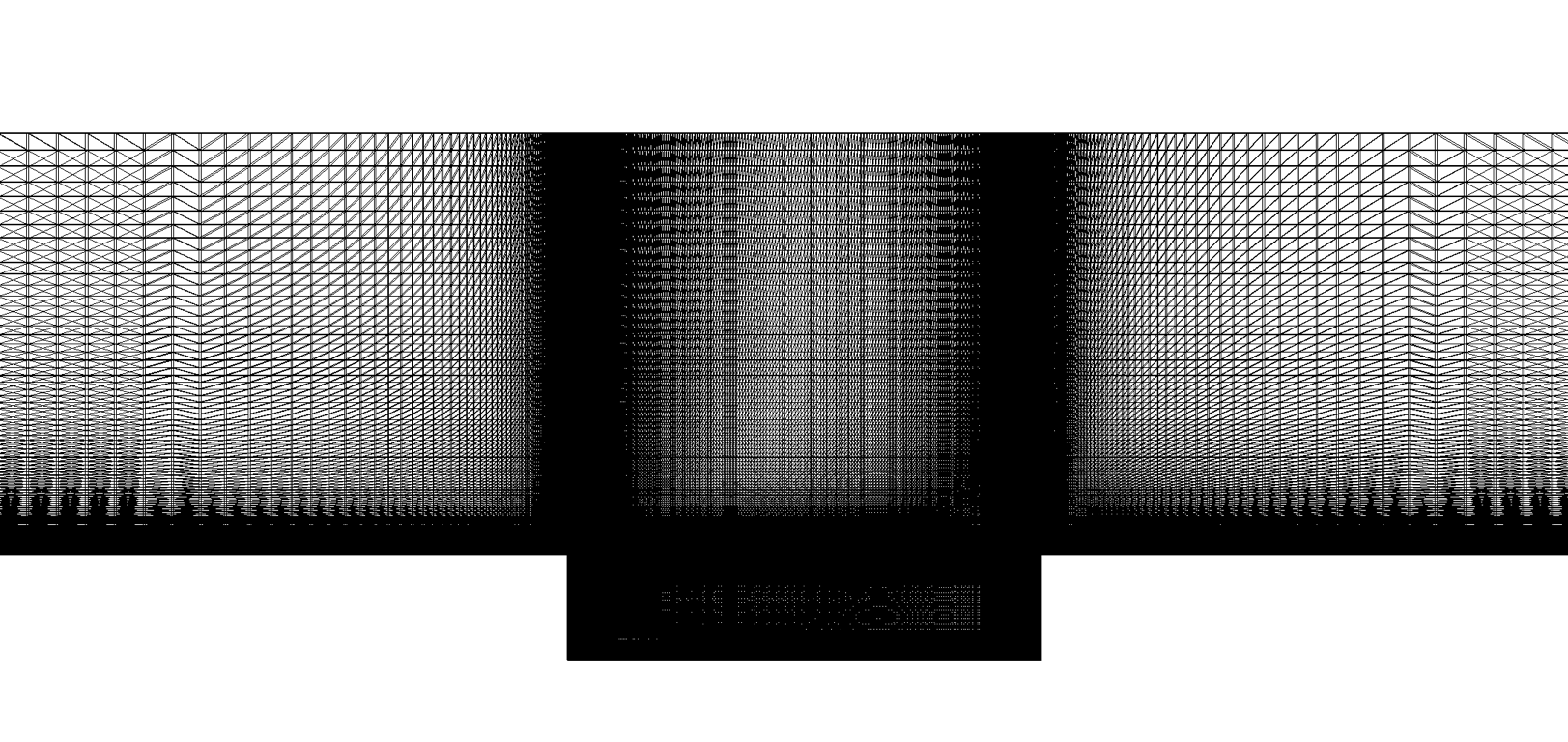} }
\caption{Computational mesh: $x-y$ plane cut.} 
\label{fig:fluidDomain} 
\end{figure}

Flow simulations are performed within a time interval $t\in\timedomain$ with
$\maxT=12.5$ \reviewerANew{time units}.
We employ the
second-order accurate implicit three-point backward \reviewerCNewer{differentiation formula}, which is a linear multistep scheme characterized by $k=2$, $\alpha_0=1$,
$\alpha_1 = -4/3$, $\alpha_2 = 1/3$, $\beta_0 = 2/3$, $\beta_1=\beta_2=0$,
for time integration\reviewerA{; future work will perform numerical
	experiments with Runge--Kutta
	schemes}. The O$\Delta$E \eqref{eq:resLinMultiSolve} arising at
	each time step is solved by a Newton--Krylov method, where GMRES  is employed as
	the iterative linear solver with a restrictive additive Schwarz
	preconditioner (with no fill in) and the previous 50 Krylov vectors are employed
	for orthogonalization.  Convergence is declared when the residual norm is
		reduced to a factor of $10^{-3}$ of its starting value.  All flow
		computations are performed in a non-dimensional setting.

The initial condition $\stateInitialNo$ is provided by first computing a
steady-state solution, and using that solution as an initial guess for an
unsteady `transient' simulation (which captures the initial
transient before the flow reaches a quasi-periodic state) of 7.5 \reviewerANew{time units}. The state at the end of the unsteady transient
simulation is then used as the initial condition for the subsequent simulations.
The steady-state calculation is characterized by the same
parameters as above, except that it employs local time stepping with a maximum
CFL number of 100, it uses the first-order implicit backward Euler time
integration scheme, it assumes a linear variation of the solution within each
control volume, it employs a Spalart--Allmaras turbulence model, and it employs only one Newton iteration per (pseudo) time
step.

The output of interest is the pressure at location (0.0001,-0.0508,0.0025),
which is shown in the bottom row of Figure \ref{fom_fields}. All
errors are reported as the $\ell^2$ relative error in this quantity, i.e.,
\begin{equation*} 
\errorValue(\pressure,\pressuretruth) = \frac{\sqrt{\sum_{\timestepit =
1}^{\maxT/\dttruth}\left(\timeproj(\pressure)(\timestepit\dttruth) -
\pressuretruth(\timestepit\dttruth)\right)^2}}{\sqrt{\sum_{\timestepit =
1}^{\maxT/\dt}\pressuretruth(\timestepit\dttruth)^2}}\reviewerC{,}
\end{equation*} 
where $p:\nat{\maxT/\dt}\rightarrow \RR{}$ is the pressure for the model of
interest, $\pressuretruth:\nat{\maxT/\dttruth}\rightarrow \RR{}$ is this
pressure response of the designated `truth' model (typically the full-order
model), and $\timeproj$ is a linear interpolation of the pressure response
onto the grid based on the truth-model time step $\dttruth$.

All computations are performed in double-precision
arithmetic on a parallel Linux cluster\footnote{The cluster contains
8-core compute nodes that each contain a
2.93 GHz dual socket/quad core Nehalem X5570 processor with 12 GB of memory. The
interconnect is a 3D torus InfiniBand.} using 48 cores across 6 nodes.

\subsection{Time-step verification}

Because this paper considers the time step to be an important parameter in
model reduction, we first perform a time-step verification study to ensure we
employ an appropriate `nominal' time step. Figure
\ref{fig:fomTimestepVerification} reports these results using a
time-step refinement factor of two.  A time step of $\dttruth = 0.0015$
\reviewerANew{time units} yields observed convergence rates in both the instantaneous drag force
on the lower wall
and instantaneous pressure at $t=\maxT$ that are close to the asymptotic rate
of convergence (2.0) of three-point BDF2 scheme.  Further, this value also
leads to sub-2\% errors in both quantities, which we deem to be sufficient for
this set of experiments.

 \begin{figure}[htbp] 
  \centering 
	\subfigure[Drag: $\dttruth  = 0.0015$ yields an
		approximate rate of convergence of 1.94 and an estimated error in the output quantity (computed via
		Richardson extrapolation) of $1.26\times 10^{-2}$. \reviewerA{The red points denote
		the result for $\dttruth  = 0.0015$}. The rightmost plot
		shows the time-dependent response for \reviewerA{the finest tested time step $\dt
		= 1.875\times 10^{-4}$ (black solid) and the converged time step $\dttruth  =
		0.0015$ (red dashed).}]{
	 \includegraphics[width=\textwidth]{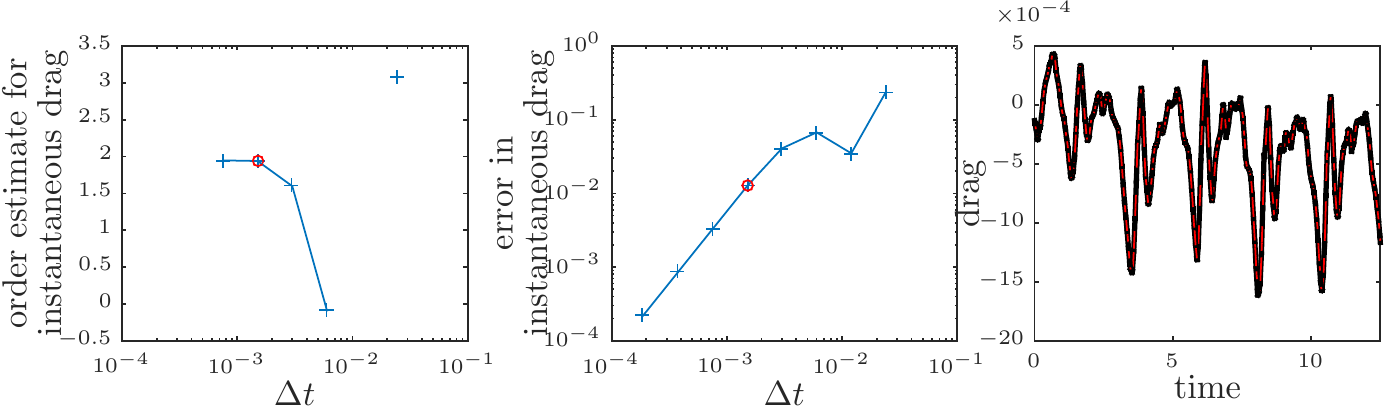}
	}
	\subfigure[Pressure: $\dttruth  = 0.0015$ yields an
		approximate rate of convergence of 1.83 and an estimated error in the output quantity (computed via
		Richardson extrapolation) of $7.68\times 10^{-4}$. \reviewerA{The red points denote
		the result for $\dttruth  = 0.0015$}. The rightmost plot
		shows the time-dependent response for \reviewerA{the finest tested time step $\dt
		= 1.875\times 10^{-4}$ (black solid) and the converged time step $\dttruth  =
		0.0015$ (red dashed).}]{
	 \includegraphics[width=\textwidth]{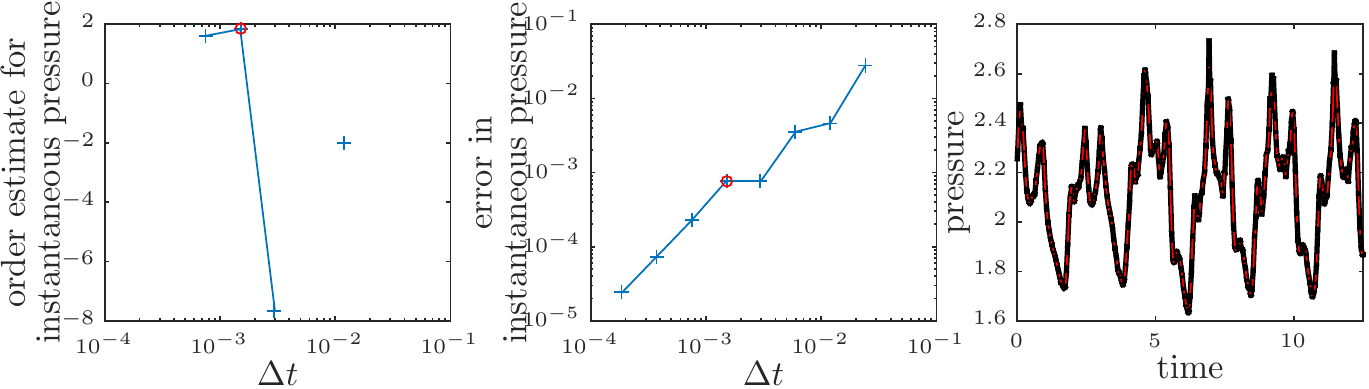}
	 }
	  \caption{Time-step verification study. Note that the approximated
		convergence rates are close to the asymptotic value of 2.0 for the 
		BDF2
		scheme.} 
		 \label{fig:fomTimestepVerification} 
		  \end{figure} 

Figure~\ref{fom_fields} shows several instantaneous snapshots of the
vorticity field and corresponding pressure field generated by the
high-fidelity CFD model.  The flow within the cavity is
quasi-periodic; during one cycle, vorticity is shed from the leading
edge of the cavity, convects downstream, and impinges on the aft edge
of the cavity.  Upon impingement, an acoustic disturbance is generated
which propagates upstream and scatters on the leading edge of the
cavity, generating a new vortical disturbance to initiate the next
oscillation cycle.  The pressure fields in the bottom row of Figure~\ref{fom_fields}
reveal regions of low pressure (blue contours) associated with
vortices, as well as acoustic disturbances both within the cavity and
radiating outside the cavity.  This complex flow is governed by the
interactions of several nonlinear processes, including roll-up of the
shear layer vortices, impingement of the vortices on the aft wall
resulting in sound generation, propagation of nonlinear acoustic
waves, and interaction of these waves with the shear layer vorticity.
\begin{figure}[h]
  \centering
  \subfigure[time = 2.10]
     {\includegraphics[width=0.22\textwidth]{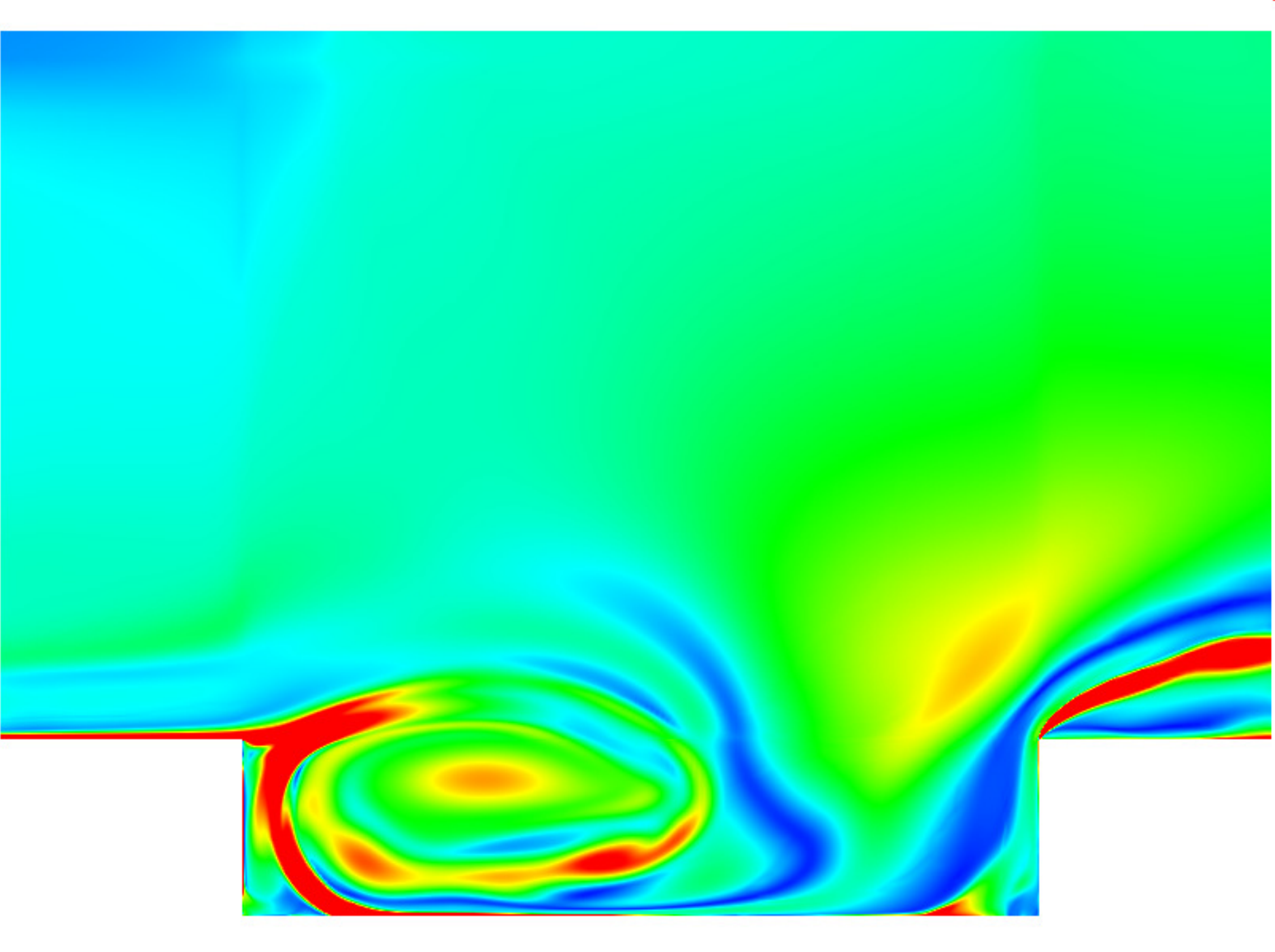}}
  \subfigure[time = 2.61]
     {\includegraphics[width=0.22\textwidth]{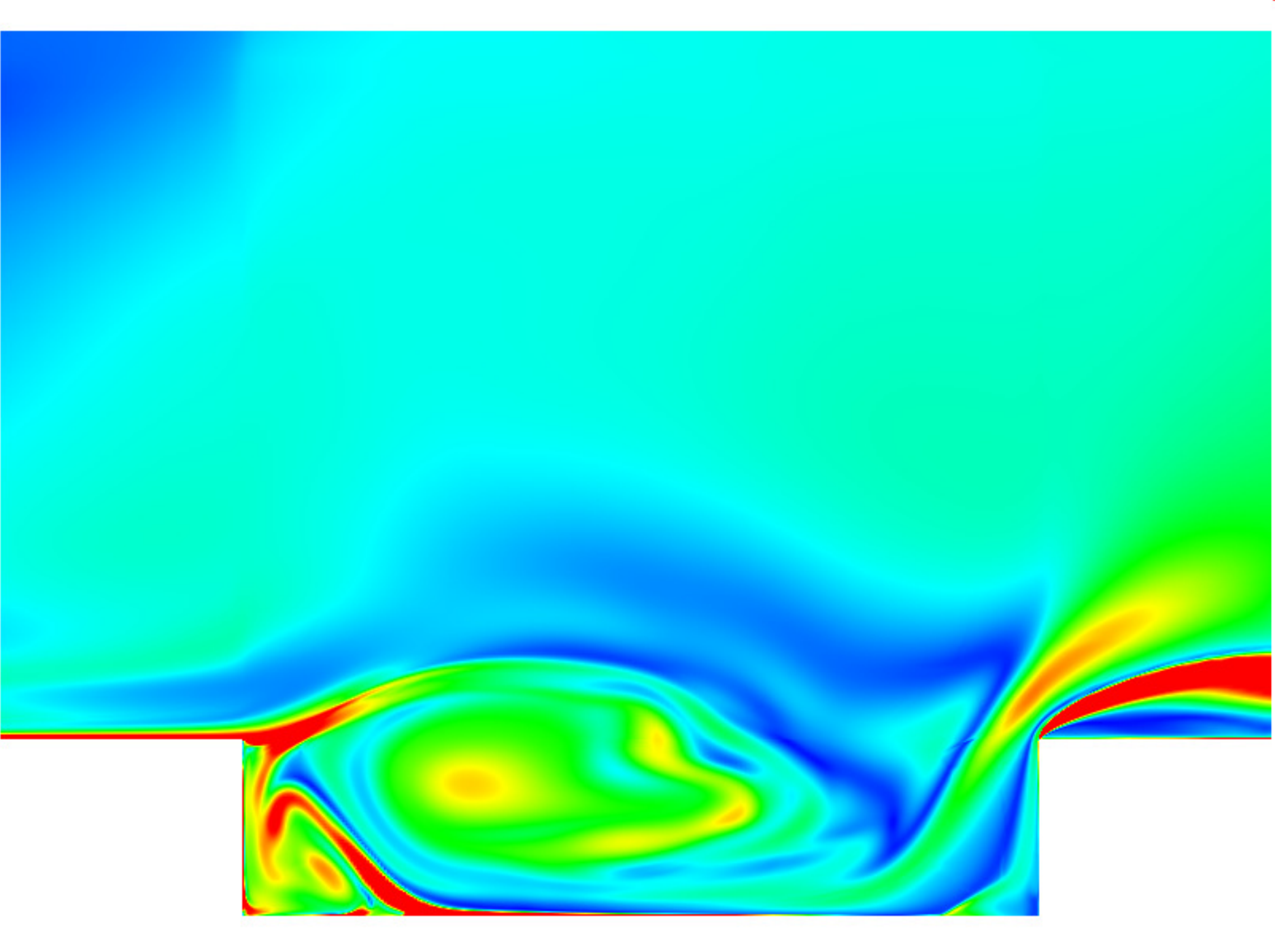}}
  \subfigure[time = 3.12 ]
     {\includegraphics[width=0.22\textwidth]{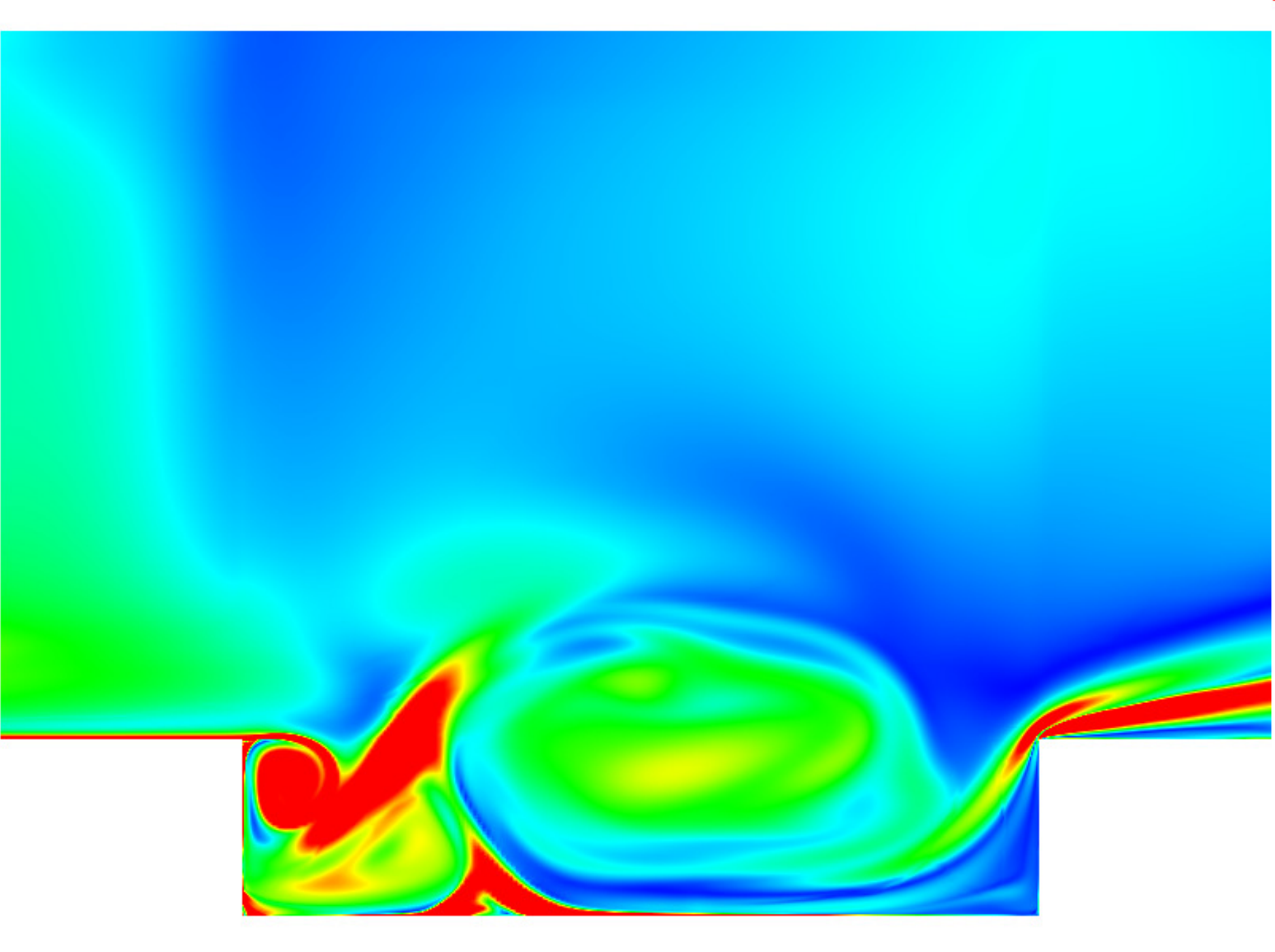}}
  \subfigure[time = 3.63]
     {\includegraphics[width=0.22\textwidth]{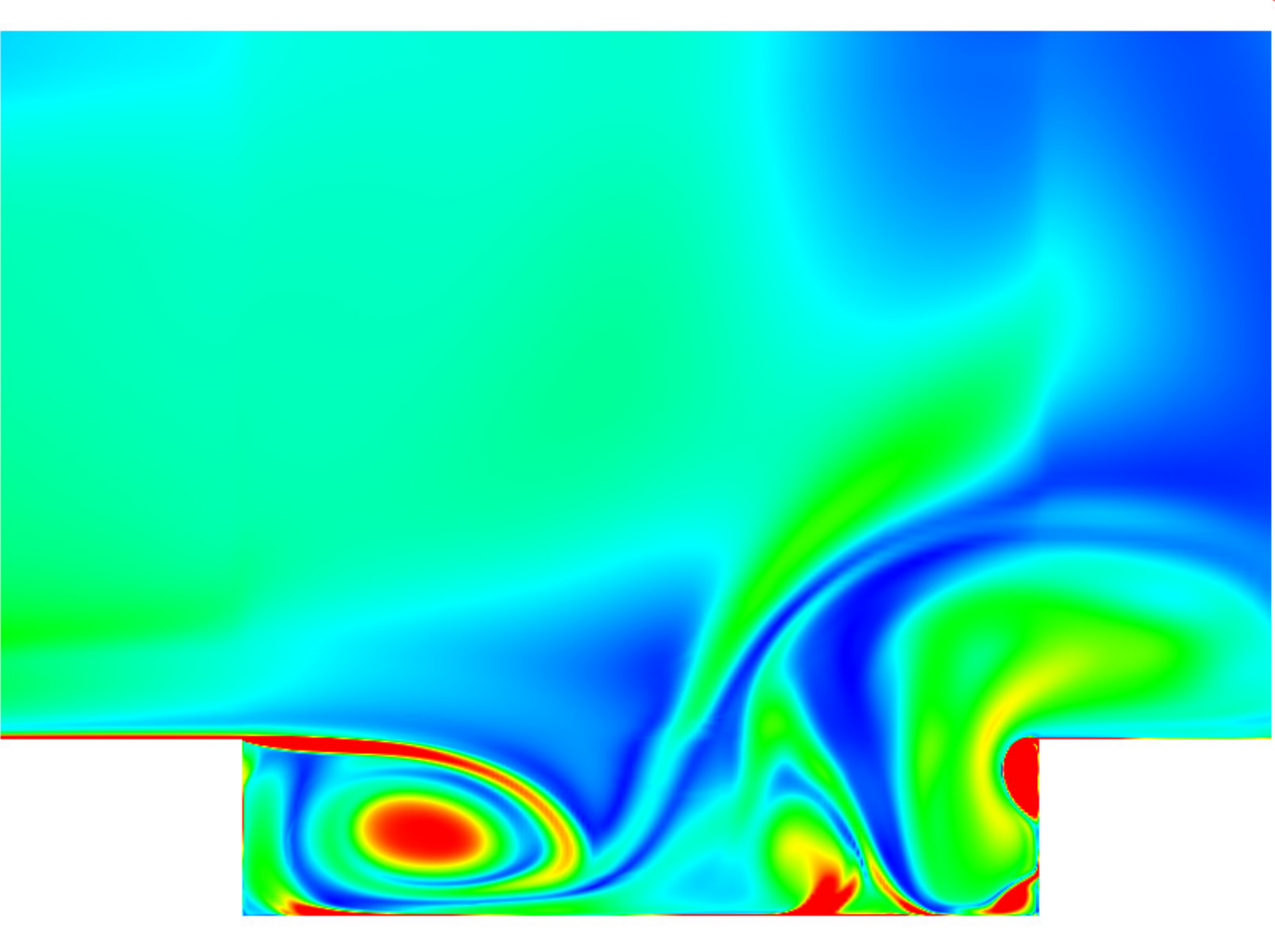}}
  \subfigure[time = 2.10]
     {\includegraphics[width=0.22\textwidth]{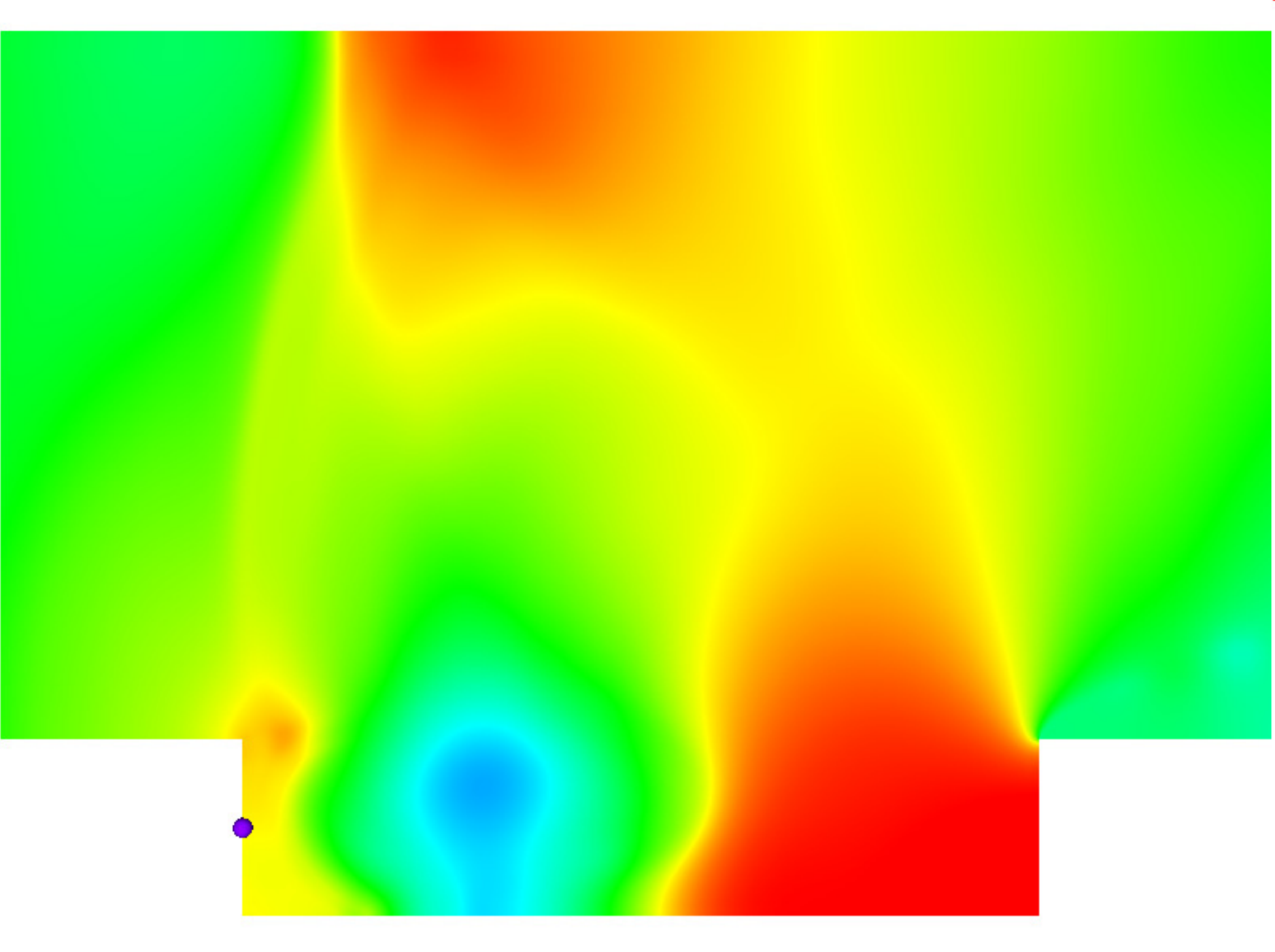}}
  \subfigure[time = 2.61]
     {\includegraphics[width=0.22\textwidth]{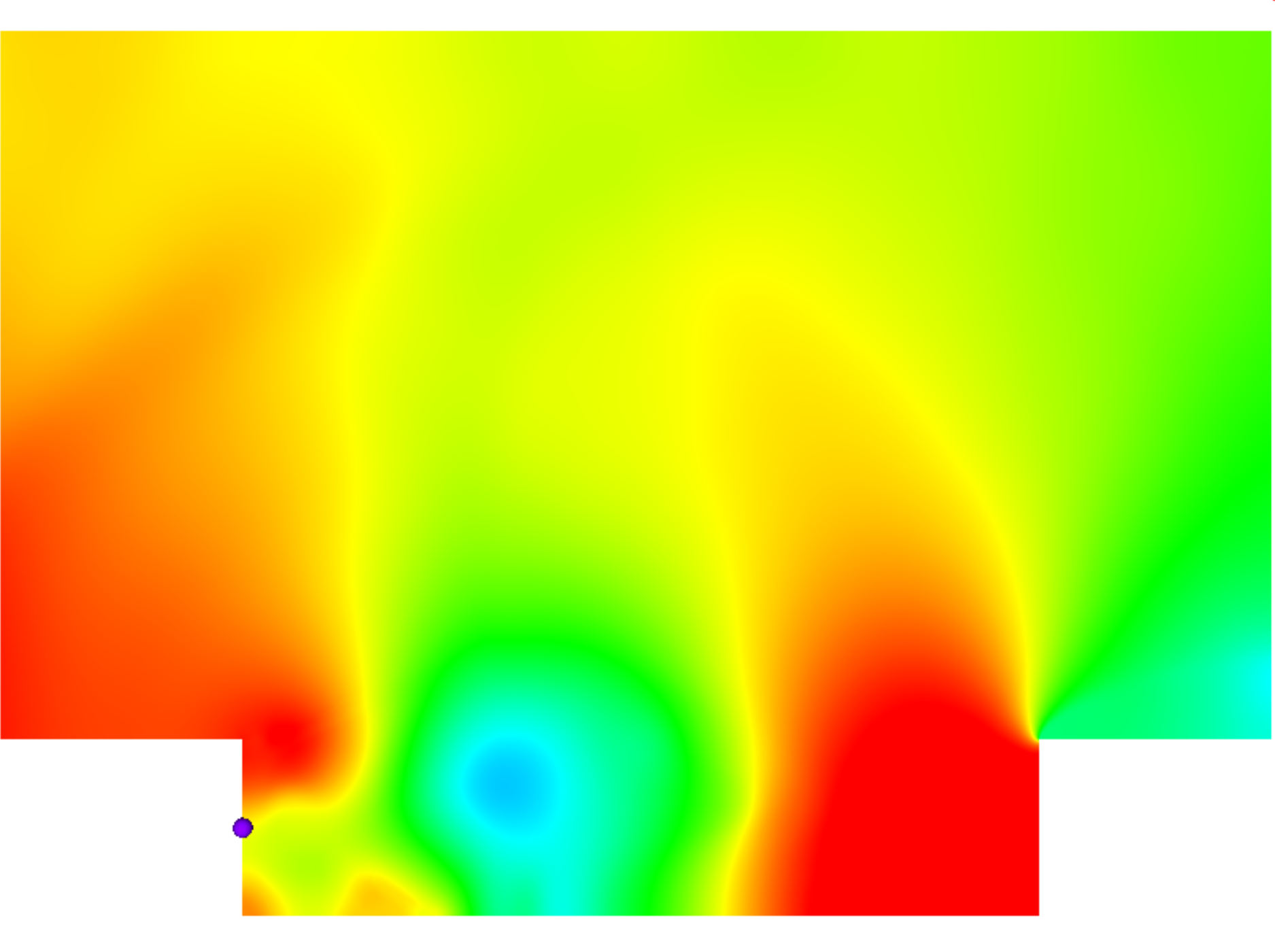}}
  \subfigure[time = 3.12 ]
     {\includegraphics[width=0.22\textwidth]{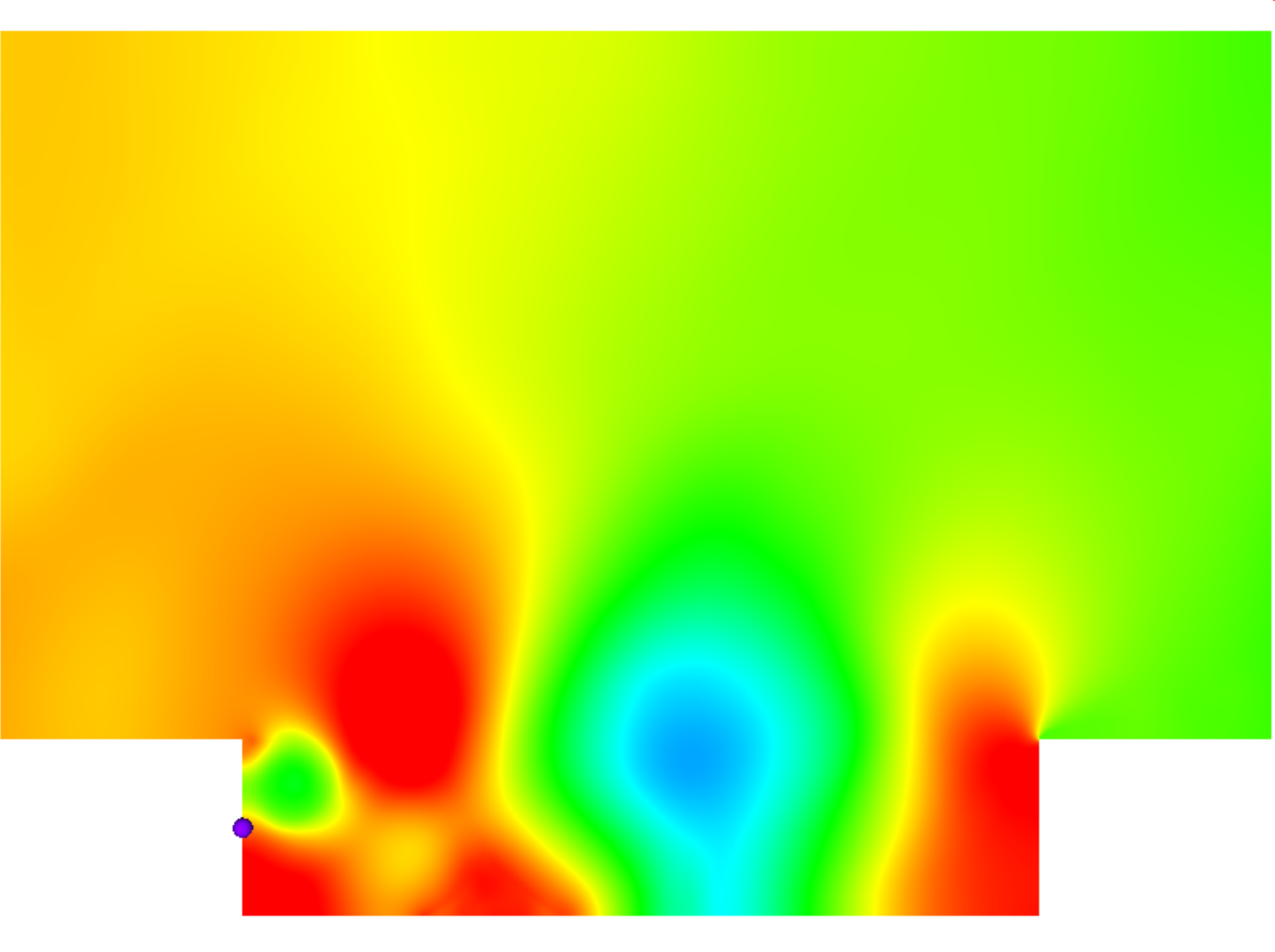}}
  \subfigure[time = 3.63]
     {\includegraphics[width=0.22\textwidth]{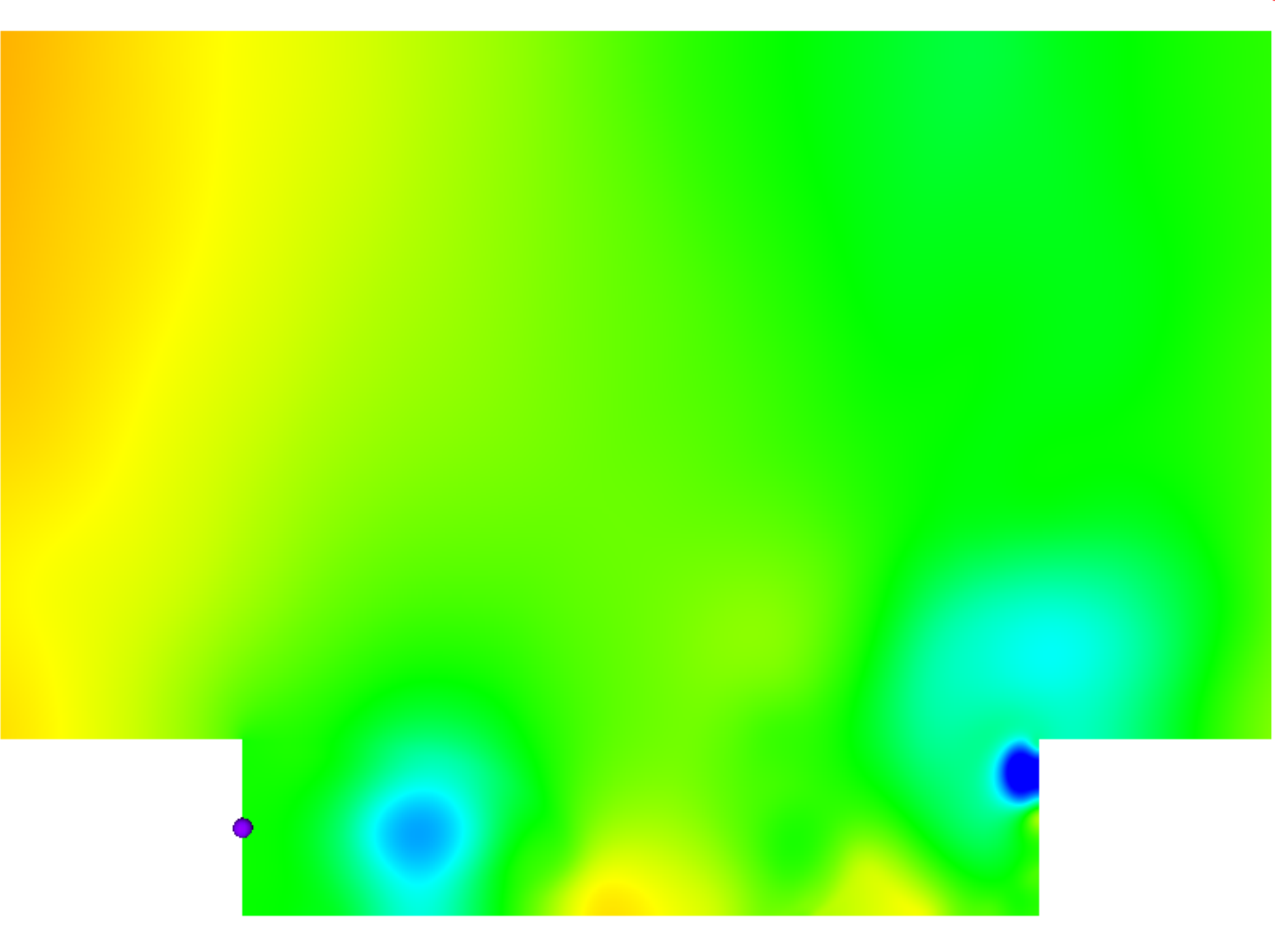}}
  \caption{Instantaneous CFD vorticity field (top) and pressure field
    (bottom) during one oscillation cycle. The dot on the forward wall
    of the cavity indicates the location of the pressure signal
    output. \label{fom_fields}}
\end{figure}

\subsection{Reduced-order models}


To construct both the Galerkin and \reviewerA{LSPG} ROMs, we employ the proper
orthogonal decomposition (POD) technique\reviewerA{; we employ a constant
weighting matrix $\weightingMatrix = \matI$
for the \reviewerA{LSPG} ROM. To construct the POD basis}, we set
$\podstate\leftarrow \podArgs{\snapsNo}{\energyCrit}$, where $\podArgsNo$ is
computed via Algorithm \ref{alg:PODSVD} of the appendix with snapshots consisting of the
initial-condition-centered full-order model states $\snapsNo =
\{\state_\star(k\dttruth)-\stateInitialNo\}_{k=1}^{8334}$, where
$\state_\star$  
denotes the FOM response computed for a time step of $\dttruth \reviewerA{= 0.0015}$. Three values of the energy
criterion $\energyCrit\in[0,1]$ are used during the experiments:
$\energyCrit=1-10^{-4}$ ($\nstate = 204$), $\energyCrit=1-10^{-5}$ ($\nstate =
368$), and $\energyCrit=1-10^{-6}$ ($\nstate = 564$). Figure
\ref{fig:podModes} shows a selection of the energy component of the computed
POD modes. Note that as the mode number increases, the modes capture finer
spatial-scale behavior, which we expect to be associated with finer time-scale
behavior; this will be verified in Section \ref{sec:spectralPOD}.

\begin{figure}[htb] 
\centering 
\subfigure[mode 1]{\includegraphics[width=0.32\textwidth]{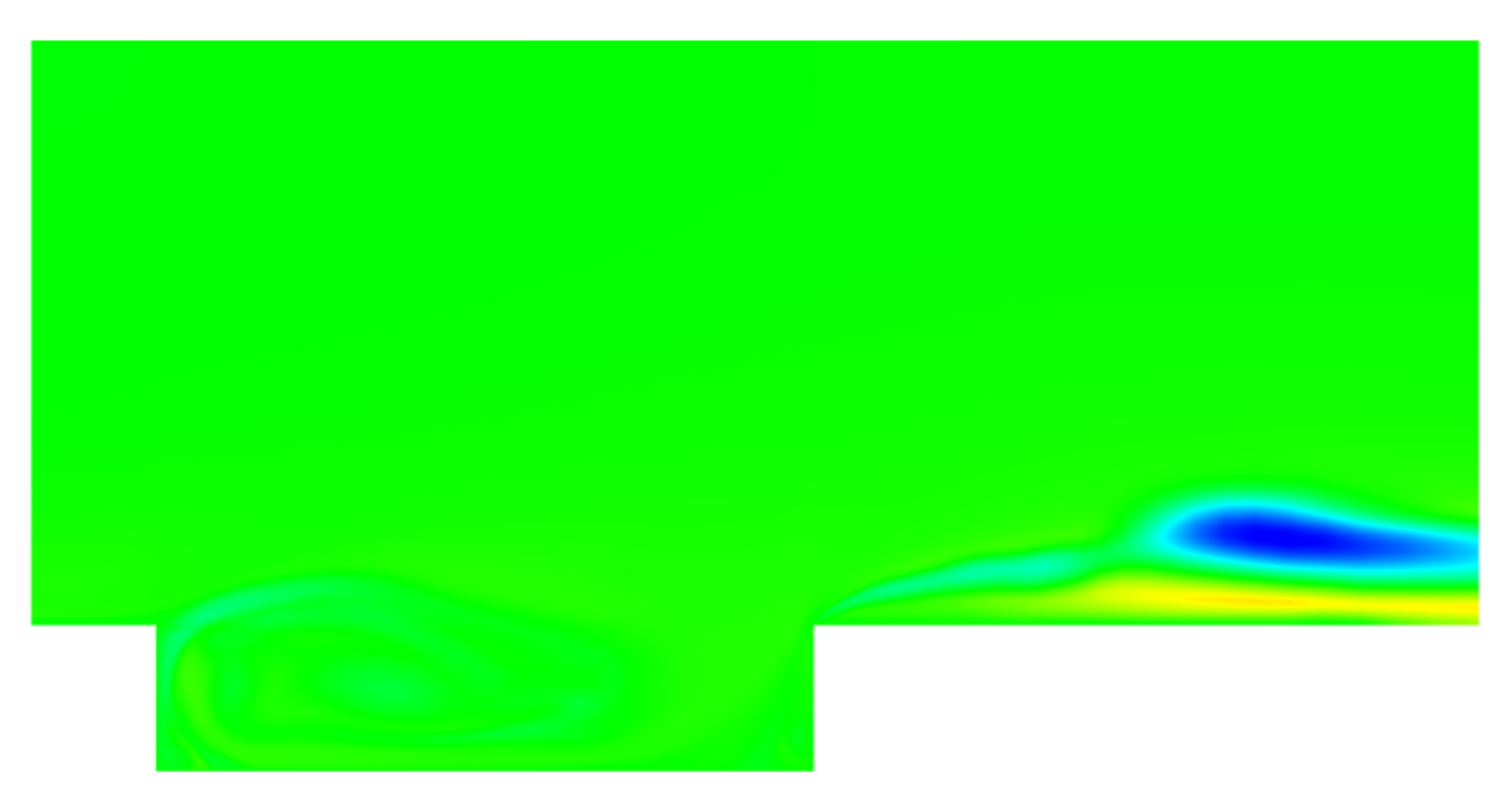}}
\subfigure[mode 21]{\includegraphics[width=0.32\textwidth]{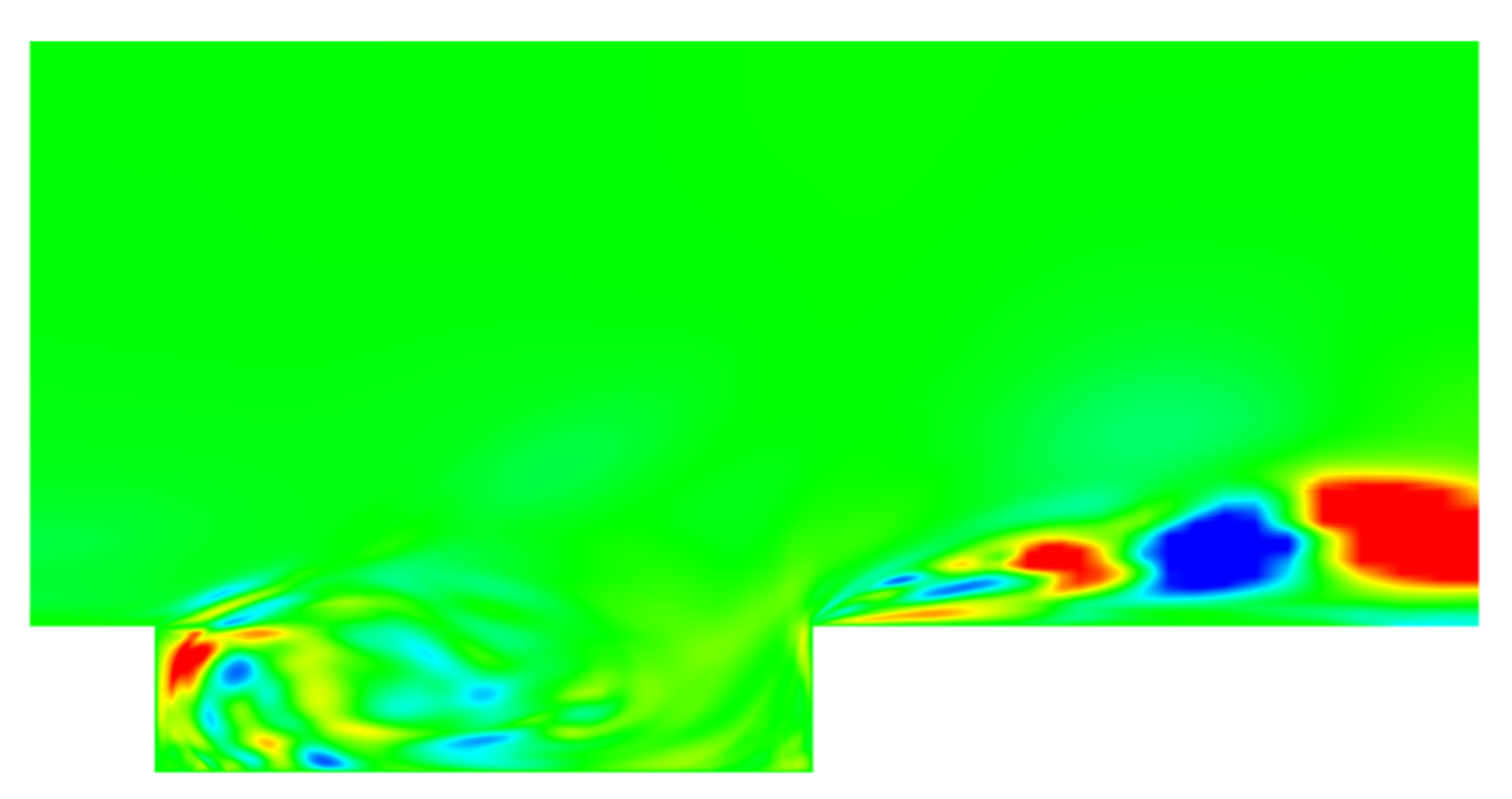}}
\subfigure[mode 101]{\includegraphics[width=0.32\textwidth]{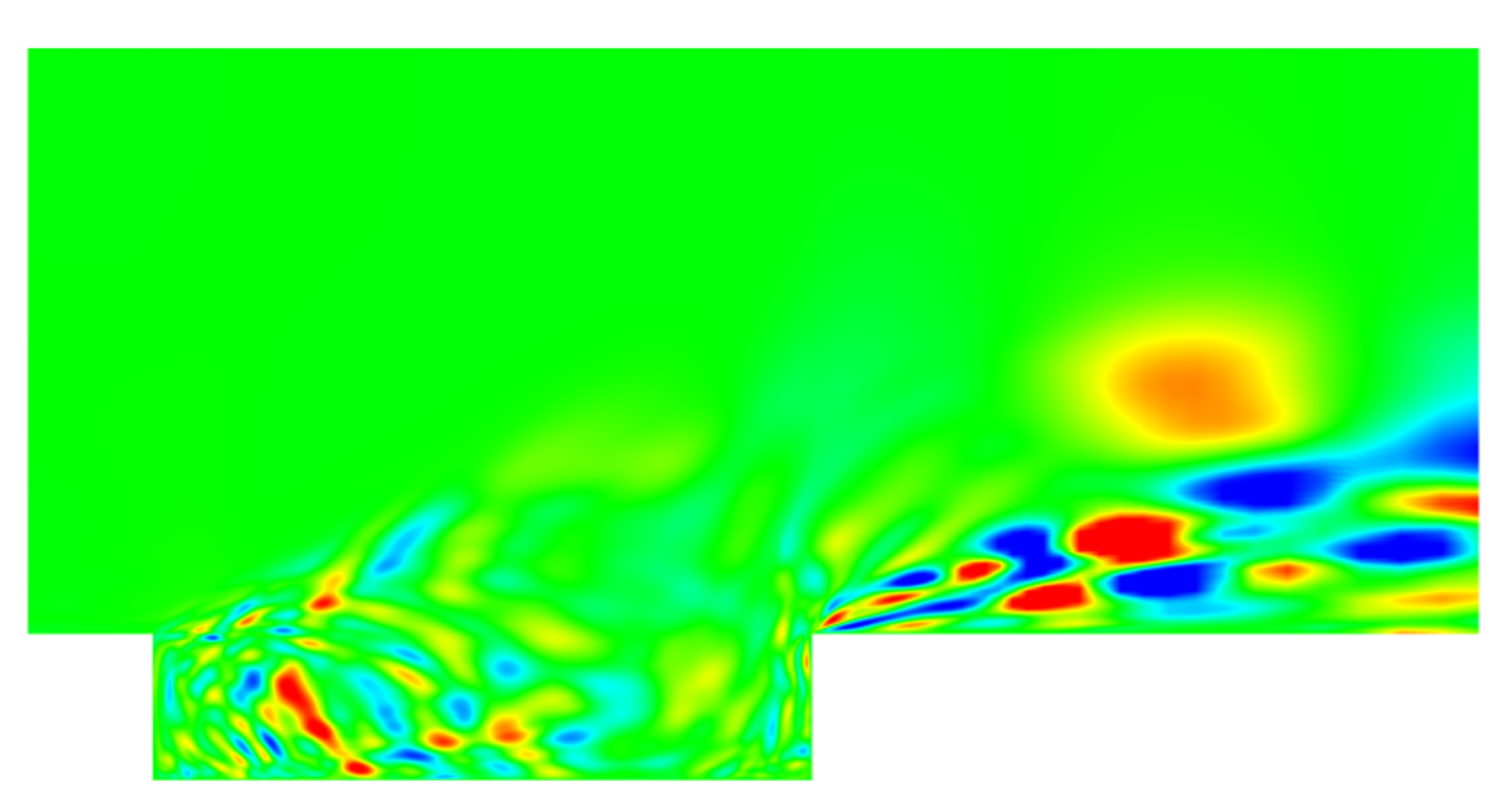}}
\subfigure[mode 201]{\includegraphics[width=0.32\textwidth]{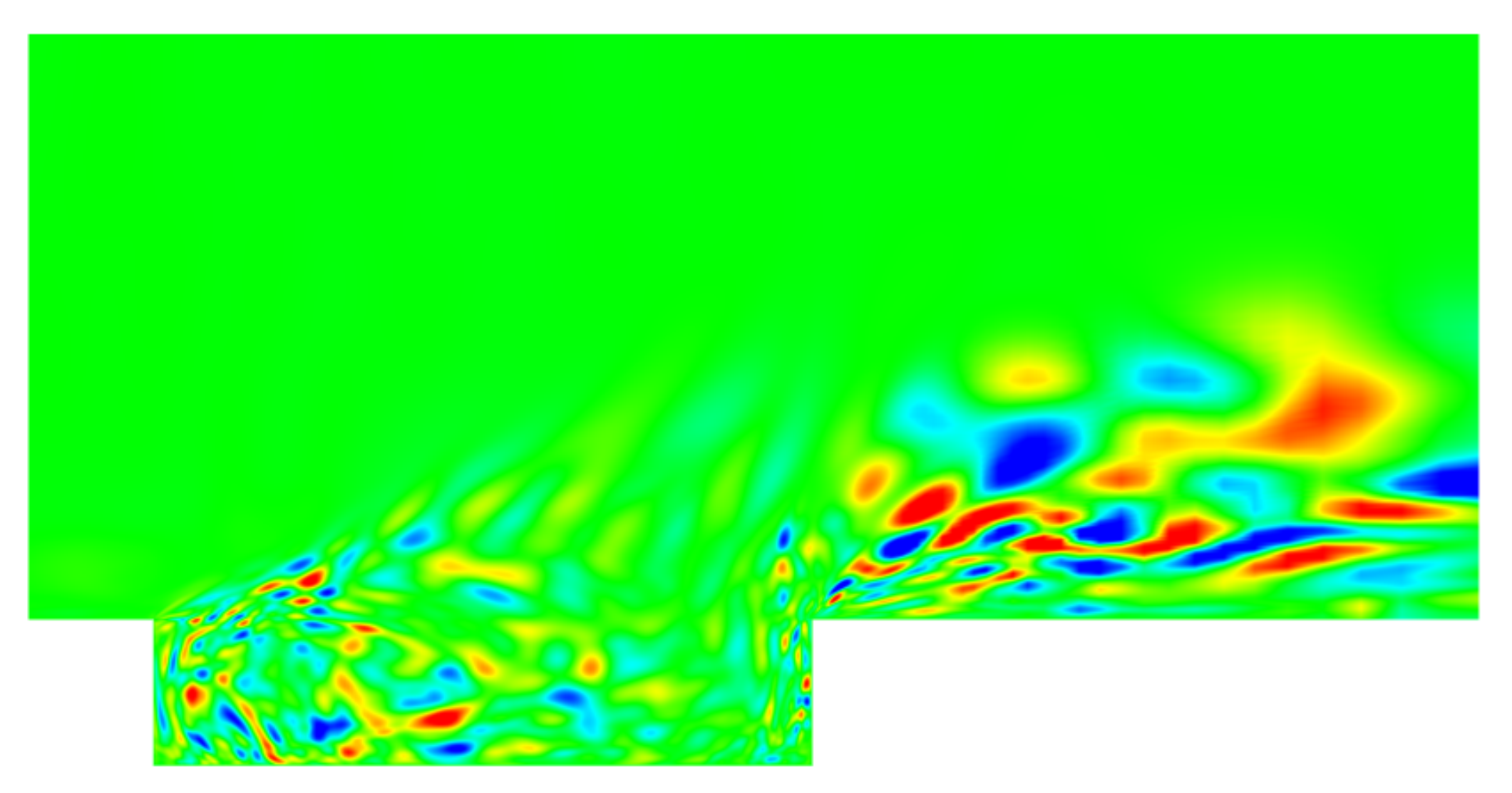}}
\subfigure[mode 401]{\includegraphics[width=0.32\textwidth]{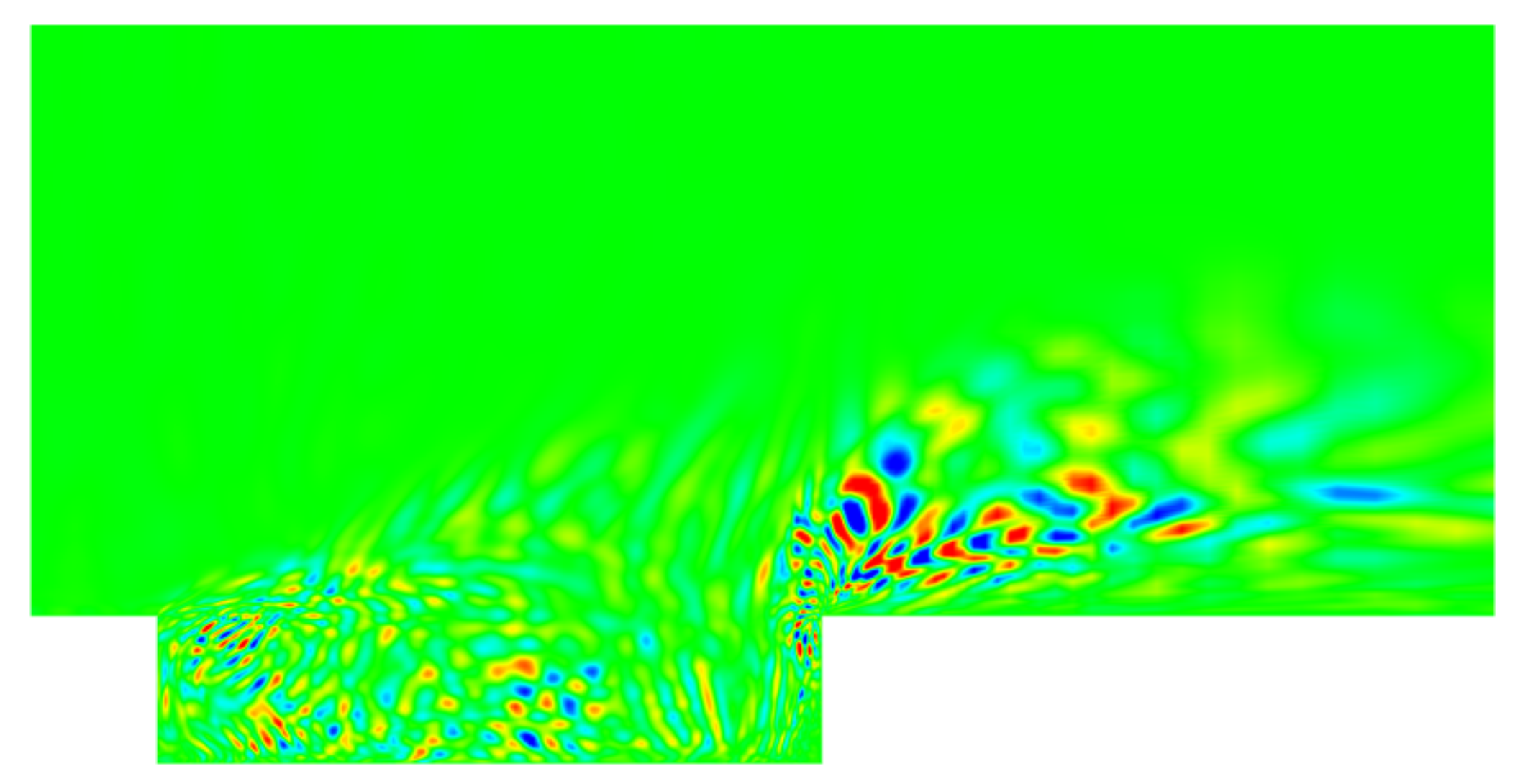}}
\caption{Visualization of the energy component of the POD modes.} 
\label{fig:podModes} 
\end{figure} 

 We first repeat the time-step verification study, but we do so for the
 reduced-order models \reviewerA{(again using the BDF2 scheme)} in the time interval $0\leq t\leq 0.55$, as all
 Galerkin ROMs remain stable in this time interval. Figure
 \ref{fig:romTimestepVerification} reports these results.  First, we note that
 the Galerkin ROM converges an approximated rate of 2.0, which is what we
 expect given that the Galerkin ROM simply associates with a
 time-step-independent ODE \eqref{eq:romODE}.  However, the \reviewerA{LSPG}
 ROM does not exhibit this behavior; in fact the error convergence is not even
 monotonic. This is \reviewerA{likely} due to the fact that the method does not associate with a
 time-step-independent ODE. 

 \begin{figure}[htb] 
  \centering 
	\subfigure[Galerkin reduced-order model]{
	 \includegraphics[width=0.48\textwidth]{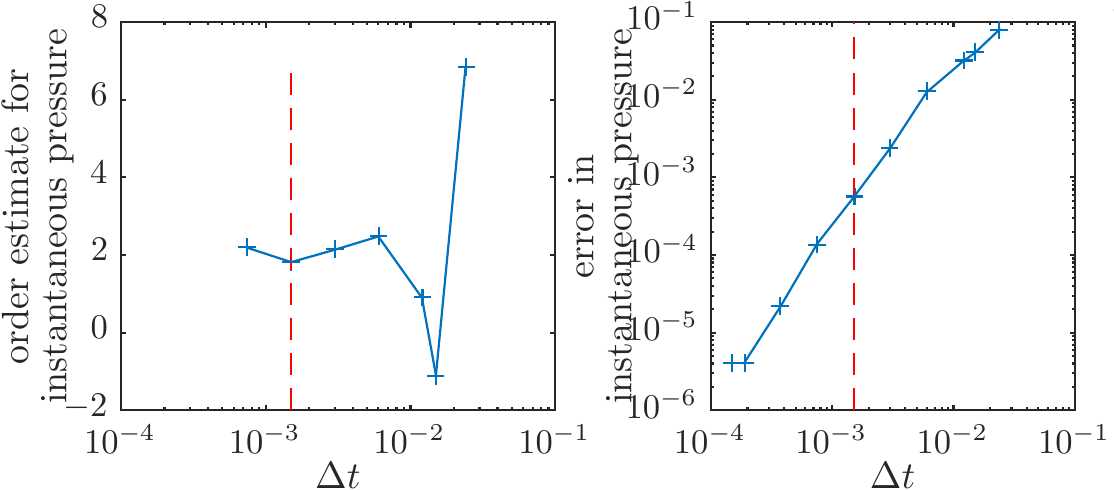}
	}
	\subfigure[\reviewerA{LSPG} reduced-order model]{
	 \includegraphics[width=0.48\textwidth]{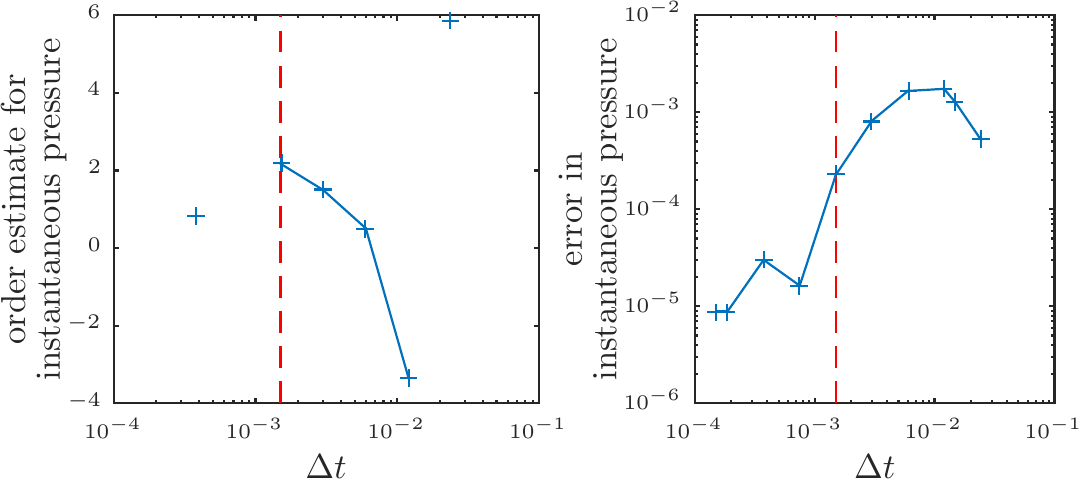}
	 }
	  \caption{Time-step verification study for Galerkin and \reviewerA{LSPG}
		reduced-order models for $\nstate = 368$ and $0\leq t\leq 0.55$. While the
		approximated convergence rate for the Galerkin reduced-order model is
		close to the asymptotic value of 2.0 for the BDF2 scheme, this 
		is not observed for the \reviewerA{LSPG} reduced-order model.
		\reviewerA{Note that the error for each ROM is computed with respect its
		response for finest Richardson extrapolation, and
		a dashed red line indicates the snapshot-collection time step $\dttruth  =
		0.0015$. See Figure \ref{fig:ROMresponses}(c--d) for the associated time-dependent
		responses.}}
		 \label{fig:romTimestepVerification} 
		  \end{figure} 

We next perform simulations for both reduced-order models for all tested basis
dimensions and time steps; Figure \ref{fig:ROMresponses} reports the
time-dependent responses. When a response stops before the end of the time
interval, this indicates that a negative pressure was encountered, which
causes AERO-F to exit the simulation. We interpret this phenomenon as a
non-physical instability.

First, note that the Galerkin ROMs become unstable (i.e., generate a negative
pressure) for all time steps and all basis dimensions. This is consistent
with previously reported results
\cite{carlbergHawaii,carlbergJCP,CarlbergGappy,carlbergThesis} that indicate
Galerkin projection almost always leads to inaccurate responses for
compressible fluid-dynamics problems. In contrast, the \reviewerA{LSPG} ROM
results in many stable, accurate responses for all basis dimensions. Further,
\reviewerA{LSPG} responses exhibit a clear dependence on the time step $\dt$.
Subsequent sections provide a deeper analysis of this dependence.

\begin{figure}[htbp] 
\centering 
\subfigure[Galerkin, $\nstate = 204$]
{\includegraphics[width=8cm]{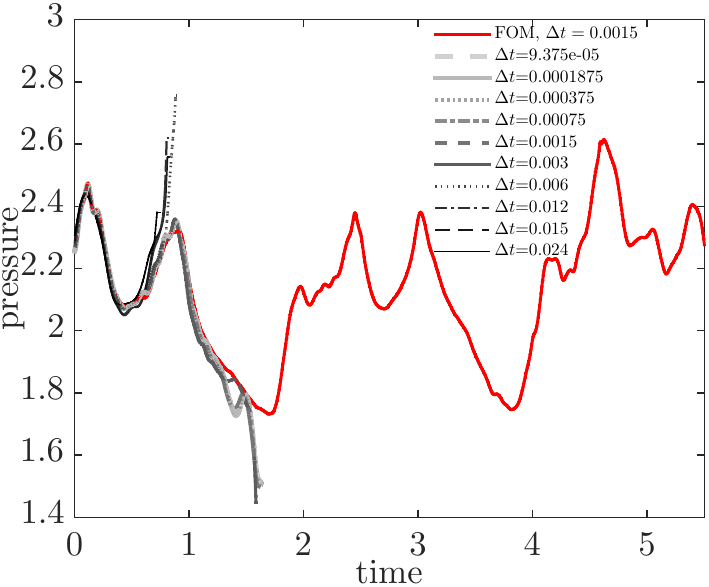}}
\subfigure[\reviewerC{LSPG}, $\nstate = 204$ ]
{\includegraphics[width=8cm]{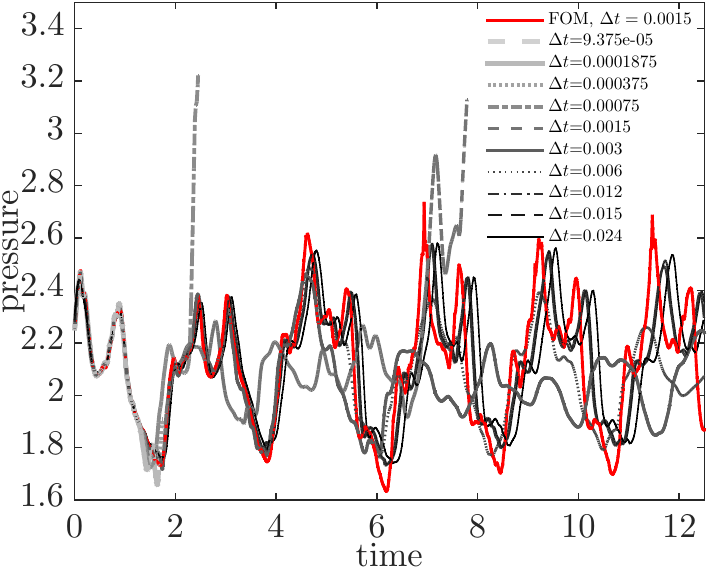}}
\subfigure[Galerkin, $\nstate = 368$ ]
{\includegraphics[width=8cm]{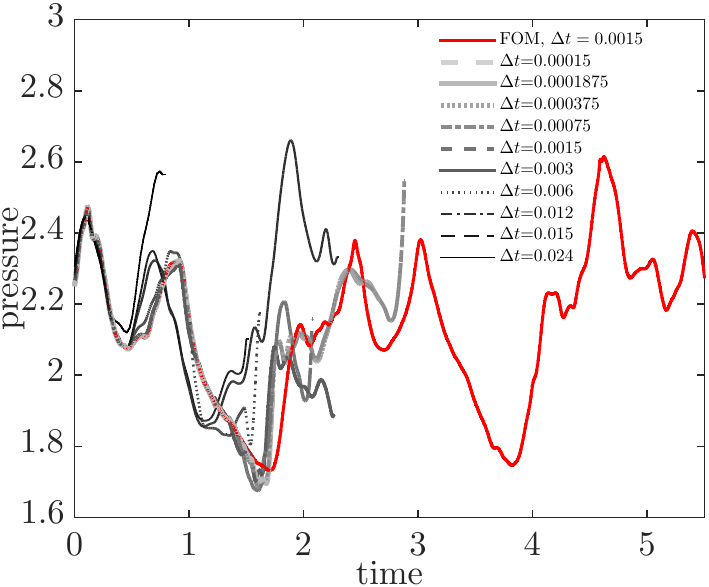}}
\subfigure[\reviewerC{LSPG}, $\nstate = 368$]
{\includegraphics[width=8cm]{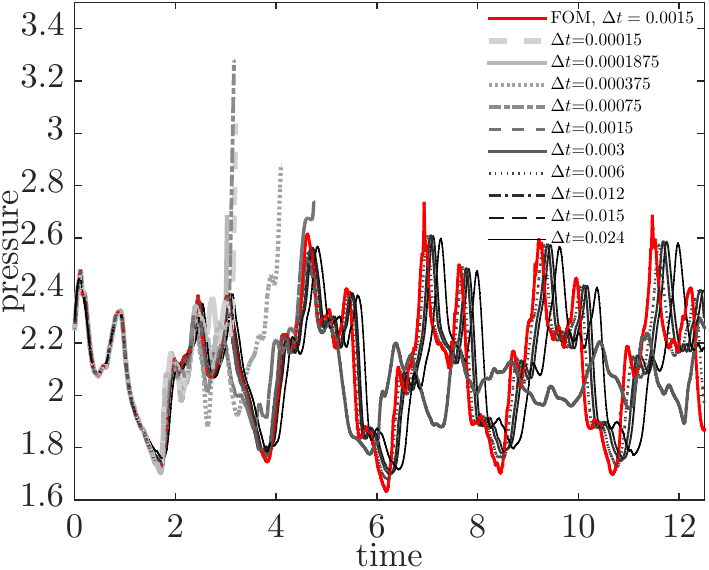}}
\subfigure[Galerkin, $\nstate = 564$]
{\includegraphics[width=8cm]{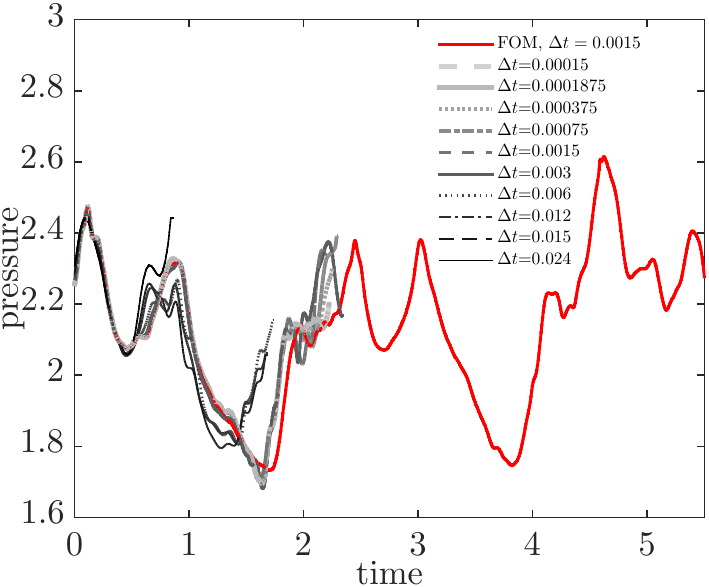}}
\subfigure[\reviewerC{LSPG}, $\nstate = 564$]
{\includegraphics[width=8cm]{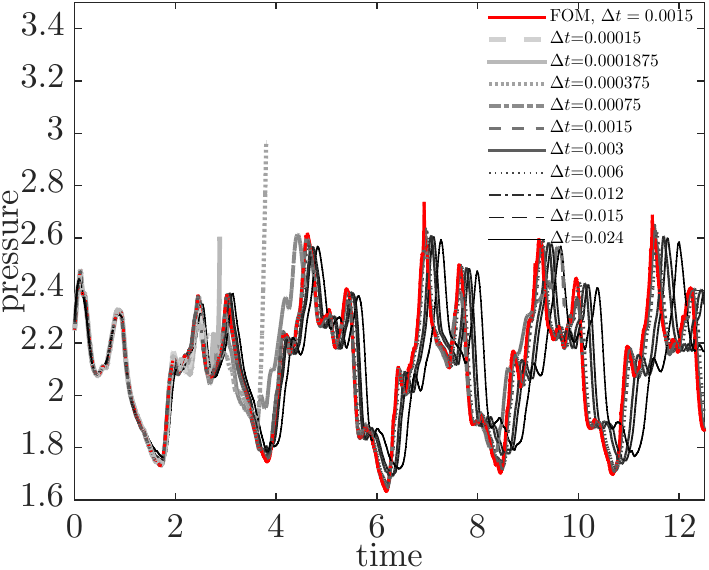}}
\caption{Responses generated by Galerkin and \reviewerC{LSPG ROMs} for different basis sizes
$\nstate$ and time steps $\dt $} 
\label{fig:ROMresponses} 
\end{figure}

\subsection{Limiting case: comparison}

We next compare the responses of the Galerkin and \reviewerA{LSPG} ROMs for
small time windows (when the Galerkin responses remain stable) and small
time steps. Figure \ref{fig:GalDiscOptErr} reports
$\errorValue(\pressure_{\text{discrete~opt.}},\pressure_{\text{Gal}_\star})$---which
is the difference between the \reviewerA{pressure responses generated by the}
\reviewerA{LSPG} ROM \reviewerA{with different time steps and the Galerkin ROM
with a fixed time step} $\dt = 1.875\times 10^{-4}$ (the smallest
tested time step)---\reviewerA{for a time window $0\leq t\leq 1.1$}.  These
responses support an important conclusion (see Theorem
\ref{thm:dtZeroEquiv}): the Galerkin and \reviewerA{LSPG} ROMs are equal in
the limit of $\dt\rightarrow 0$ \reviewerCNewer{for $\weightingMatrix =
1/\sqrt{\alpha_0}\matI$,
which is what we employ for the LSPG ROM (note that
$\alpha_0=1$ for this time integrator)}.\footnote{\reviewerA{Note that in the $\nstate =
564$ case, it is not clear if the difference is converging to zero.
This is likely due to the fact that the time steps are
not sufficiently small to detect convergence to zero in this case. In fact, as the basis
dimension $\nstate$ increases, the basis captures finer temporal behavior
(as will be shown in Figure \ref{fig:spectralContent})
and so the time scale of the ROM response will be smaller; in
turn, smaller time steps $\dt$ will be required to detect convergent behavior.}} This has significant consequences for the
\reviewerA{LSPG} ROM, as decreasing the time step leads to the same 
\emph{unstable} response as Galerkin; larger time steps are needed to ensure
the \reviewerA{LSPG} ROM is stable for the entire time interval.

\begin{figure}[htbp] 
%
\centering 
\subfigure[$\nstate = 204$]
{\includegraphics[width=0.3\textwidth]{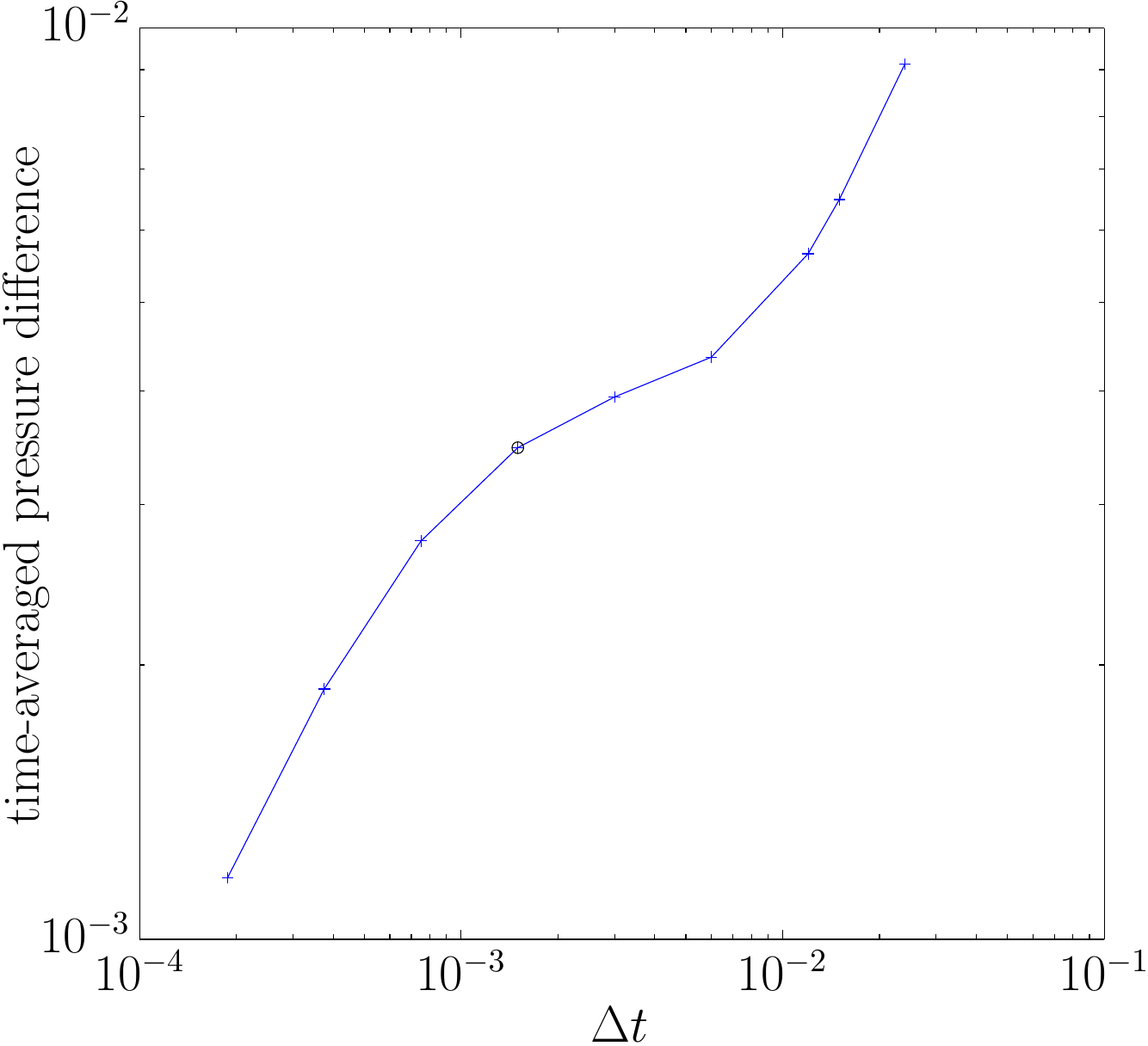}}
\subfigure[$\nstate = 368$]
{\includegraphics[width=0.3\textwidth]{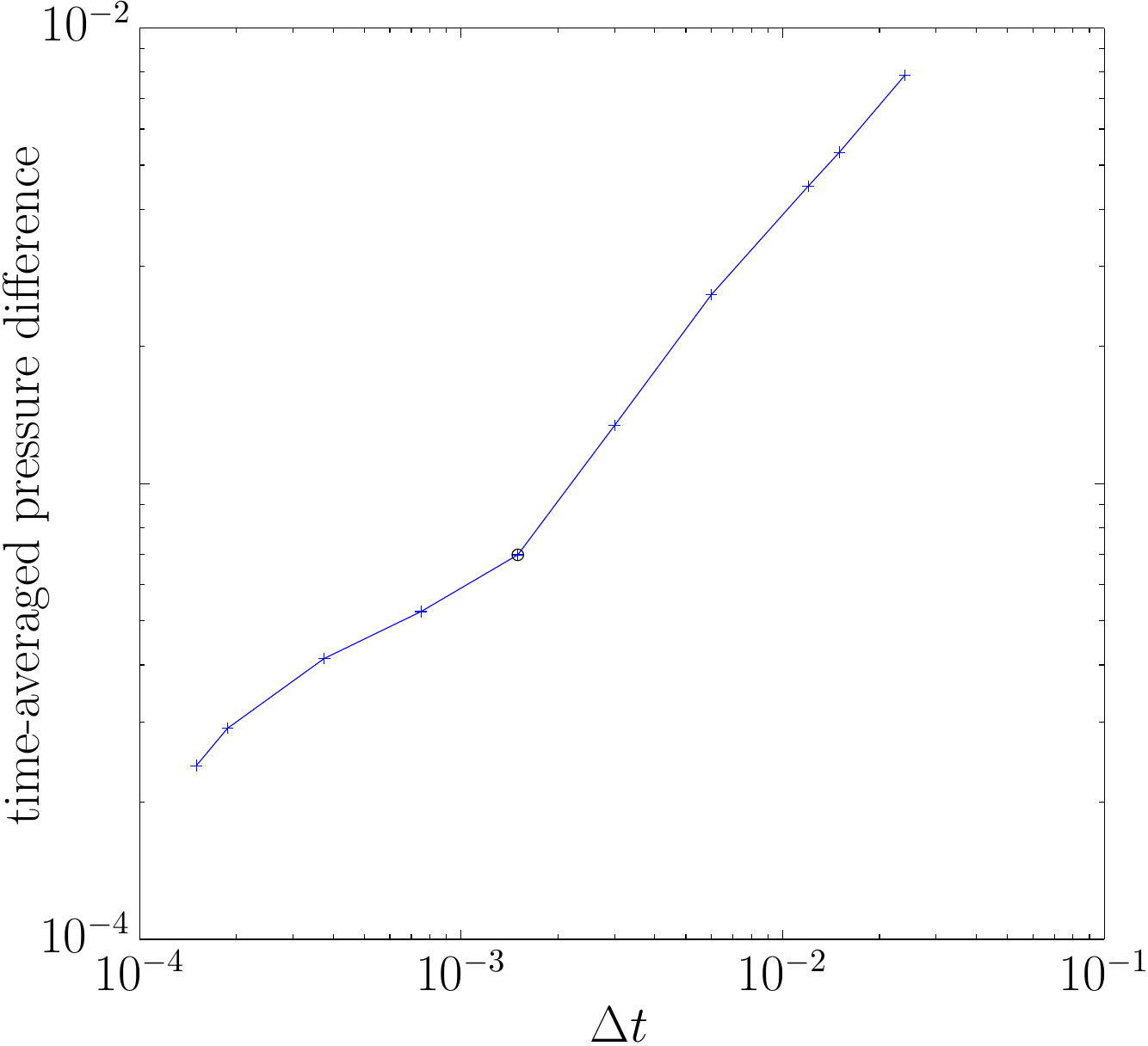}}
\subfigure[$\nstate = 564$]
{\includegraphics[width=0.3\textwidth]{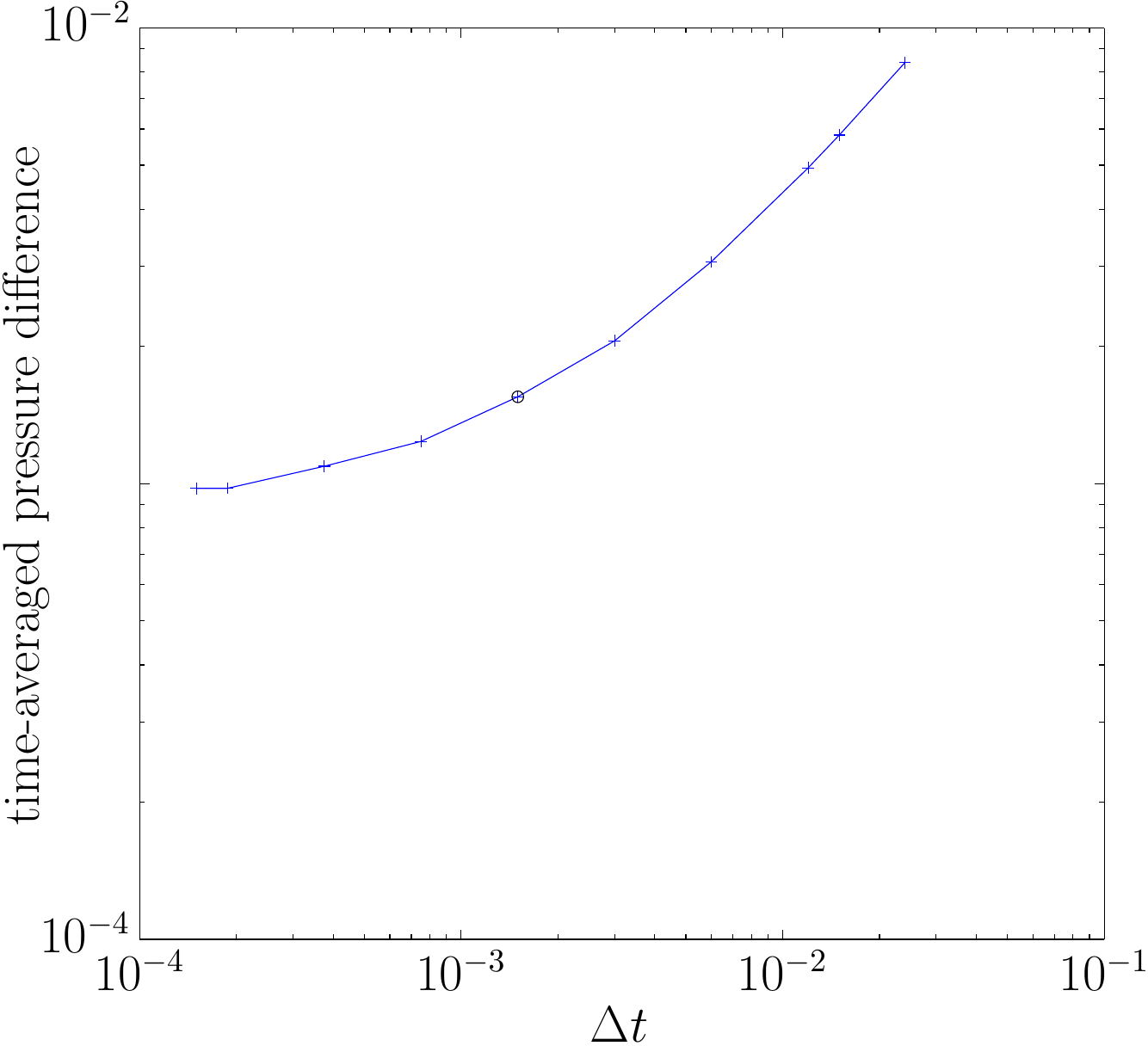}}
\caption{\reviewerA{Difference 
$\errorValue(\pressure_{\text{\reviewerANew{LSPG}}},\pressure_{\text{Gal}_\star})$
between the pressure responses generated by the
\reviewerA{LSPG} ROM with different time steps and the Galerkin ROM with a
fixed time step $\dt
= 1.875\times 10^{-4}$}
in $0\leq t\leq 1.1$. This demonstrates convergence of
	the \reviewerA{LSPG} ROM to Galerkin as $\dt \rightarrow 0$.}
\label{fig:GalDiscOptErr} 
\end{figure} 

Figure \ref{fig:galPgErrors} reports
$\errorValue(\pressure_{\text{discrete~opt.}},\pressure_{\text{FOM}_\star})$
and
$\errorValue(\pressure_{\text{Gal.}},\pressure_{\text{FOM}_\star})$---which
are the differences between the two ROM-generated pressure responses and the
full-order model pressure response for $\dt = 1.875\times 10^{-4}$---as a
function of the time step for all three basis dimensions and three time
intervals. These results highlight a critical observation: the
\reviewerA{LSPG} ROM is \emph{more accurate} for an intermediate time step. This not only supports the result of Corollary
\ref{corr:auxiliaryProblem},
but provides an interesting insight: taking a larger time step not only
leads to better speedups (i.e., the end of the time interval is reached
in fewer time steps), but it also decreases the error, sometimes significantly. This is
further explored in the next section.

\begin{figure}[htbp] 
\centering 
\subfigure[$0\leq t\leq 0.55$, $\nstate = 204$]{
 \includegraphics[width=0.25\textwidth]{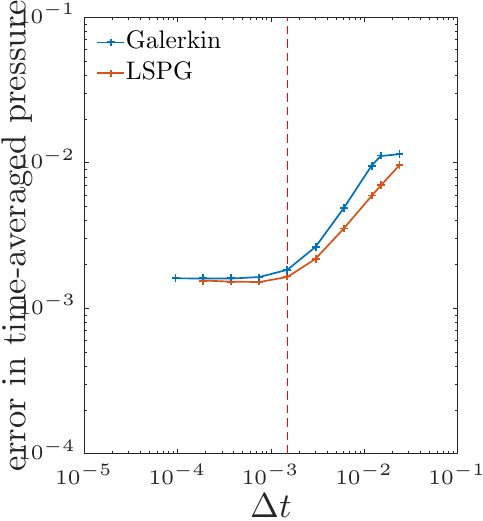}}
\subfigure[$0\leq t\leq 1.1$, $\nstate = 204$]{
	\includegraphics[width=0.25\textwidth]{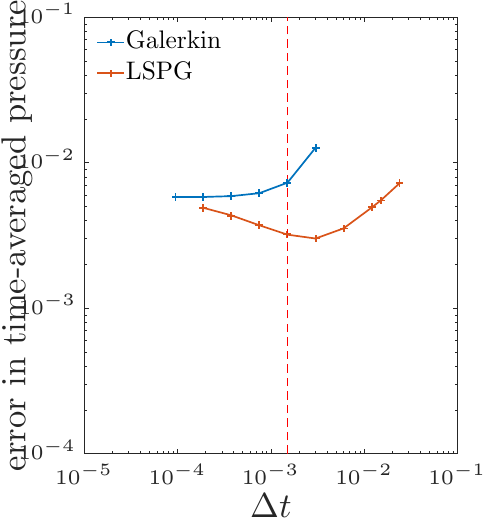}}
\subfigure[$0\leq t\leq 1.54$, $\nstate = 204$]{
	\includegraphics[width=0.25\textwidth]{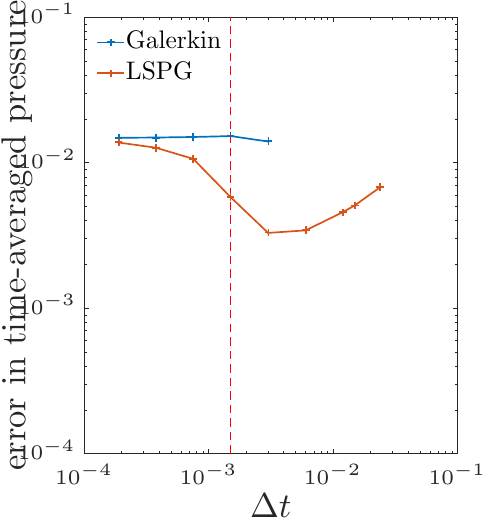}}
\subfigure[$0\leq t\leq 0.55$, $\nstate = 368$]{
	\includegraphics[width=0.25\textwidth]{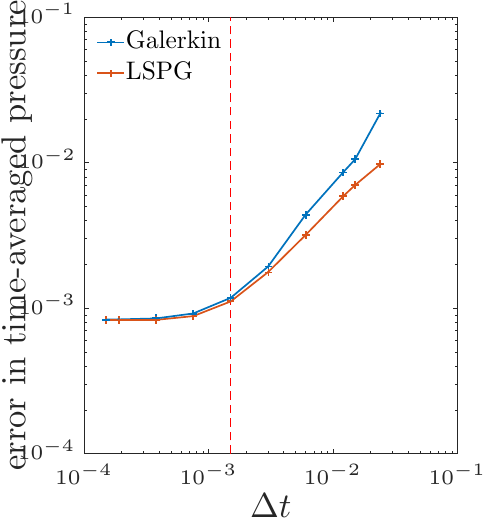}}
\subfigure[$0\leq t\leq 1.1$, $\nstate = 368$]{
	\includegraphics[width=0.25\textwidth]{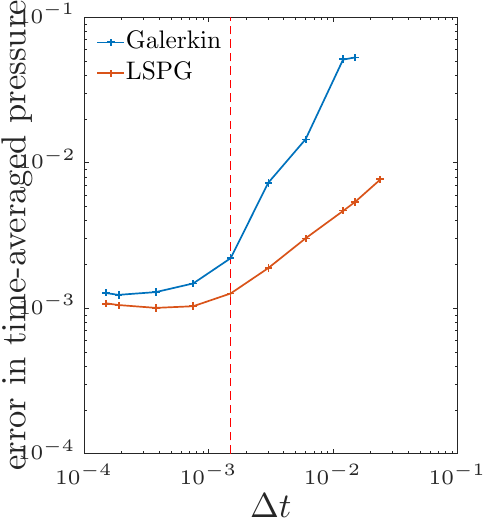}}
\subfigure[$0\leq t\leq 1.65$, $\nstate = 368$]{
	\includegraphics[width=0.25\textwidth]{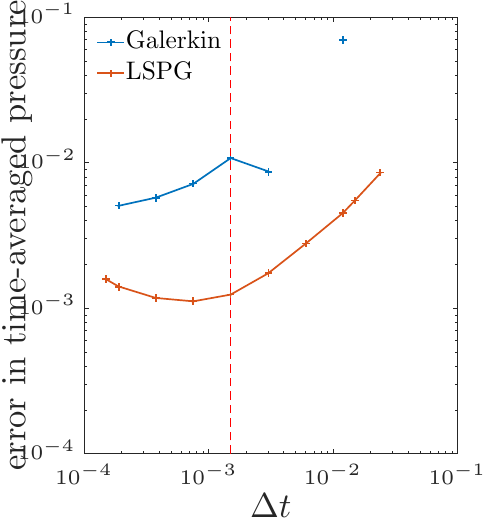}}
\subfigure[$0\leq t\leq 0.55$, $\nstate = 564$]{
	\includegraphics[width=0.25\textwidth]{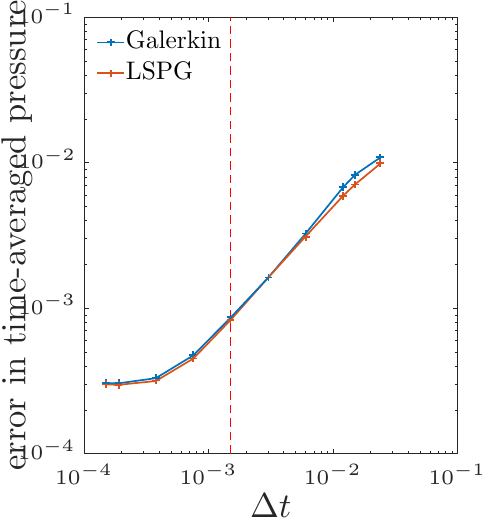}}
\subfigure[$0\leq t\leq 1.1$, $\nstate = 564$]{
	\includegraphics[width=0.25\textwidth]{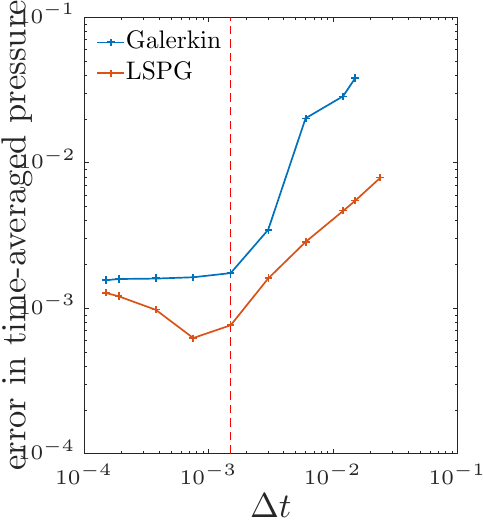}}
\subfigure[$0\leq t\leq 1.65$, $\nstate = 564$]{
	\includegraphics[width=0.25\textwidth]{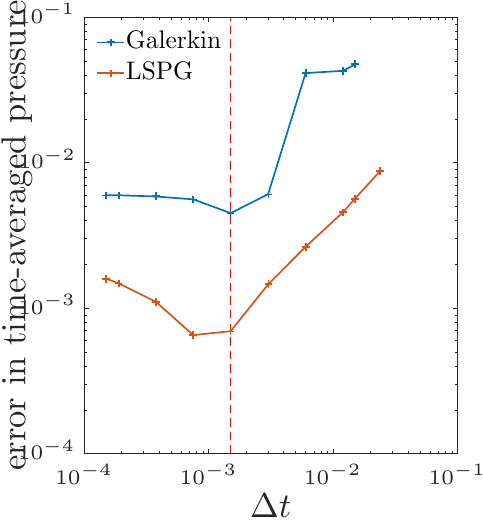}}
\caption{Galerkin errors $\errorValue(\pressure_{\text{\reviewerANew{Gal.}}},\pressure_{\text{FOM}_\star})$
and \reviewerANew{LSPG} errors 
$\errorValue(\pressure_{\text{\reviewerANew{LSPG}}},\pressure_{\text{FOM}_\star})$
over different time intervals, time steps, and basis dimensions.
\reviewerA{For reference, a dashed red line indicates the snapshot-collection
time step $\dttruth  = 0.0015$.}}
\label{fig:galPgErrors} 
\end{figure} 

\subsection{Time-step selection}\label{sec:optimalDt}

Recall from Corollary \ref{corr:auxiliaryProblem} and Remark
\ref{rem:modestTimestep} that decreasing the time step $\dt$ has a non-obvious
effect on the error bound for the \reviewerA{LSPG} ROM.  We now assess these
effects for the current problem.

\subsubsection{Spectral content of POD basis}\label{sec:spectralPOD}
In our interpretation of the error bound \eqref{eq:auxToPlot2} for the
\reviewerA{LSPG} ROM applied to the backward Euler scheme, we noted that
the time step should be `matched' to the spectral content of the trial basis
$\podstate$.  This is of practical importance, as selecting an
appropriate time step for the ROM should take into
account the relevant temporal dynamics associated with the 
basis.  For example, a time step may be too small if the basis has filtered
out modes with a time scale matching that of the time step. 
If we assume that the basis $\podstate$ is computed via POD, then we would
expect the vectors to be naturally ordered such that lower mode numbers are
associated with lower temporal frequencies. Then, including additional modes
has the effect of encoding information at higher frequencies.  It follows
that the time step should be decreased as additional modes are retained in
construction of the ROM.
 
Here we investigate the validity of this assumption by examining the spectral
content of the POD basis vectors for the current cavity-flow problem. We
compute the time histories of the generalized coordinates by projecting the
FOM solution onto the POD basis as $\stateRedFOM(k\dttruth) \defeq
\podstate^T(\state_\star(k\dttruth)-\stateInitialNo)$, $k\in\nat{8334}$.  We then compute
power spectral densities of the generalized coordinates $\stateRedFOM(t)$.
Figure~\ref{state_psds} shows sample spectra, normalized by the total energy
in each signal,\footnote{The energy in a time series within some frequency
range is obtained by integrating the power spectral density over that range.}
for several of the POD modes.  The figure shows that energy shifts to higher
	frequencies as the POD mode number increases, confirming our assumption for
	this example.  This is further quantified by calculating a characteristic
	time-scale $\tau_{95}$ associated with each mode; we define this time scale
	as the inverse of the frequency below which 95 percent of the energy is
	captured for that mode.  Figure~\ref{tau95} plots this time scale versus the
	mode number, showing a clear trend of decreasing time scale with increasing
	mode number.

Thus, at least for the present application problem, we expect the optimal time
step for the \reviewerA{LSPG} ROM to decrease as modes are added to the POD
basis (this will be verified by Figure \ref{fig:pgSelfCompare}).  Note that
systematic calibration could be performed to attempt to automate selection of
the ROM time step as a function of basis dimension.  \reviewerB{While
this would be of clear practical interest, we do not pursue it here, as
optimal-timestep computation would be complicated in practice by nonlinear
interactions arising from the dynamical system, as well as effects from the
spatial-discretization error and POD truncation error.}
\begin{figure}[h]
\centering
\subfigure[Power spectral densities of several generalized coordinate
  time series for the cavity-flow problem. \label{state_psds}]{
\includegraphics[width=0.45\textwidth]{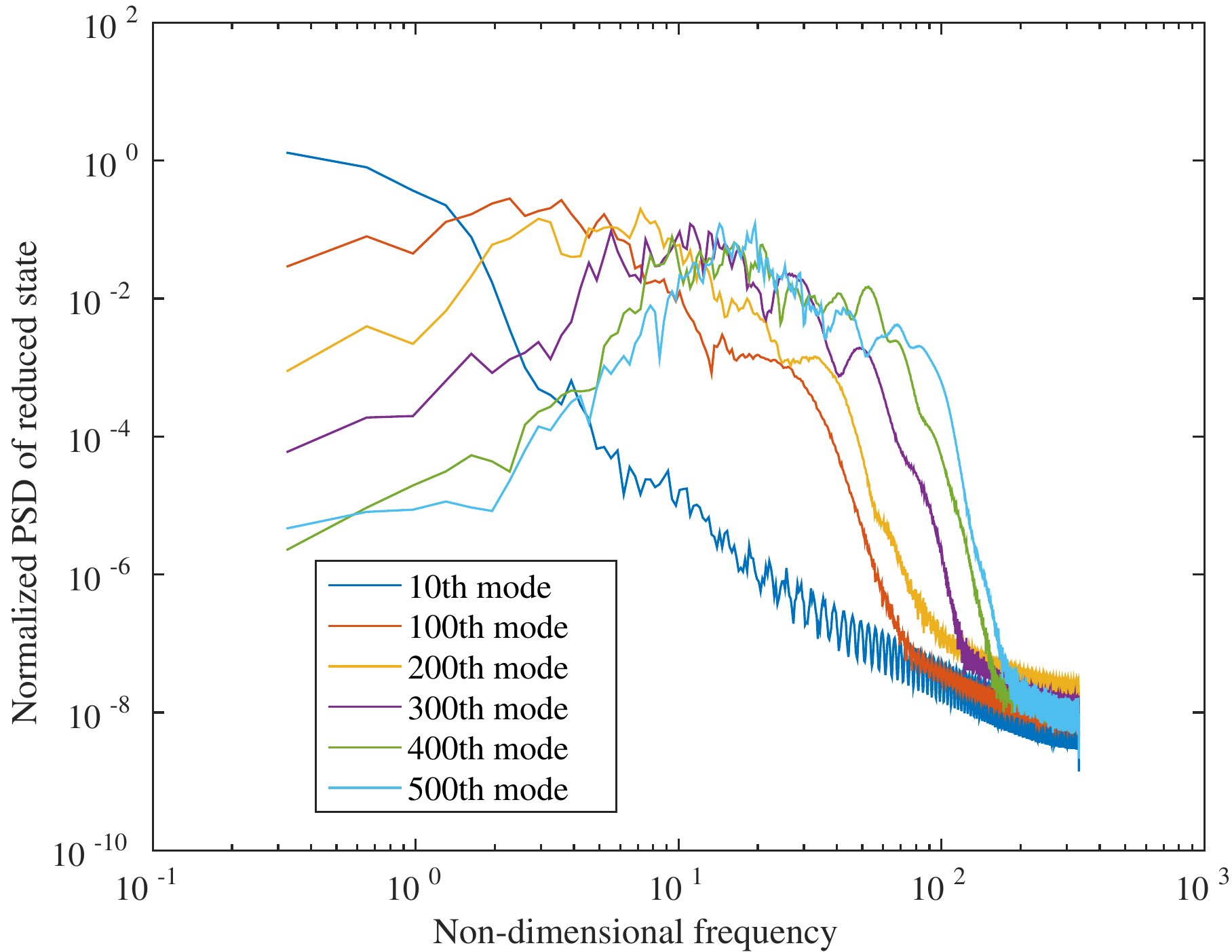}
}
\subfigure[Dependence of the POD mode characteristic time scale on
  mode number for the cavity-flow problem. \label{tau95}]{
\includegraphics[width=0.45\textwidth]{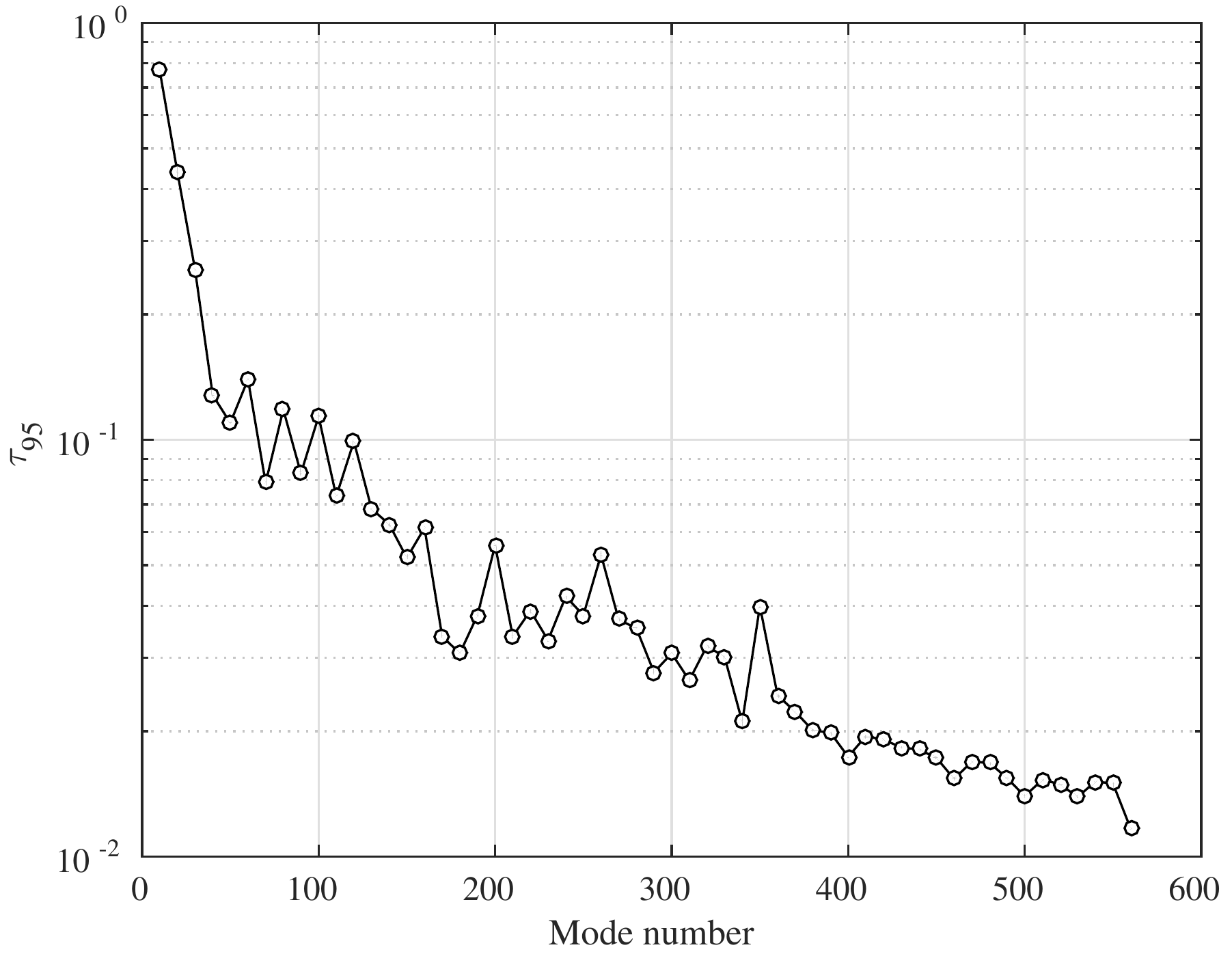}
}
\caption{Spectral content of the POD basis.}
\label{fig:spectralContent}
\end{figure}

\subsubsection{Error bound behavior}\label{sec:experimentErrorBound}

Having verified that higher POD mode numbers correspond to smaller
wavelengths, we now numerically assess quantities related to the error bound
\eqref{eq:auxToPlot2}. 
First, Figure \ref{fig:projErrorPlots} reports the dependence of the maximum
relative projection error $\max_k\relProjErrork{\podstate}{\dt}$ on the time
step
$\dt$ and the basis dimension, where
\begin{equation*}
\relProjErrork{\podstate}{\dt} \defeq
\frac{\|(\matI-\podstate\podstate^T)(\state_\star(k\dt) -
\state_\star((k-1)\dt))\|}{\|\state_\star(k\dt) -
\state_\star((k-1)\dt)\|}\reviewerC{.}
\end{equation*}
 Note that $\relProjErrorNok$ is closely related to  $\bar
\mu^k$ from error bound \eqref{eq:auxToPlot2}, as they are equal if $\vec
x_0 + \vec\Phi\stateRedP(t) = \state(t)$ and the \reviewerA{LSPG} ROM computes
$\stateRedP^k$ such that $\bar \mu^k$ is minimized.
 
These results confirm that adding basis
vectors---which we know has the effect of encoding higher frequency
content---significantly reduces the projection error for small time steps
$\dt$, but has less of an effect on larger time steps, as retaining the first
POD vectors already enables dynamics at that scale to be captured.

Next, Figure \ref{fig:errorboundresult1} plots the error bound
\eqref{eq:auxToPlot2} for a value of $\lipschitzConstant=1$ and with  $\bar \mu^k=
\relProjErrorNok$. This highlights an important result: \emph{selecting an
intermediate time step $\dt$ leads to the lowest error bound, regardless of
the basis dimension.} Even though this result corresponds to the backward
Euler integrator, we expect a similar trend to hold for the present
experiment, which uses the BDF2 scheme. The next section assesses the
performance of the \reviewerA{LSPG} ROM, including its dependence on the time
step.

 \begin{figure}[htbp] 
  \centering 
	\subfigure[Dependence of the maximum relative projection error
	$\max_{k}\bar\mu_\star^k$ on the time step $\dt$ and basis dimension
	$\nstate$]{
	 \includegraphics[width=0.45\textwidth]{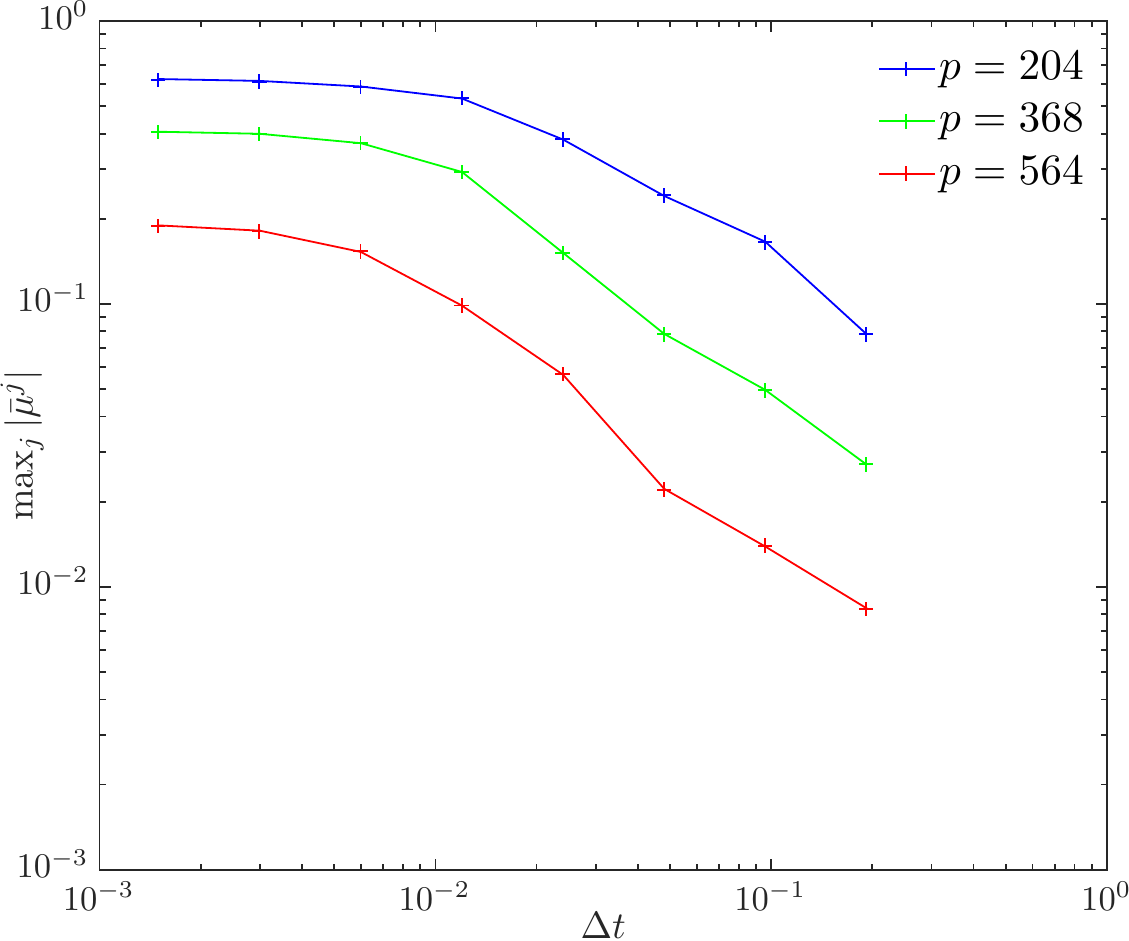} 
	 \label{fig:projErrorPlots}
		}
	\subfigure[\reviewerA{LSPG} error bound for the backward Euler
	scheme \eqref{eq:auxToPlot2} using experimental data for $\bar \mu_\star^k$ and
	setting $\lipschitzConstant=1$ and $\bar \mu^k= \bar \mu_\star^k$.]{
	 \includegraphics[width=0.45\textwidth]{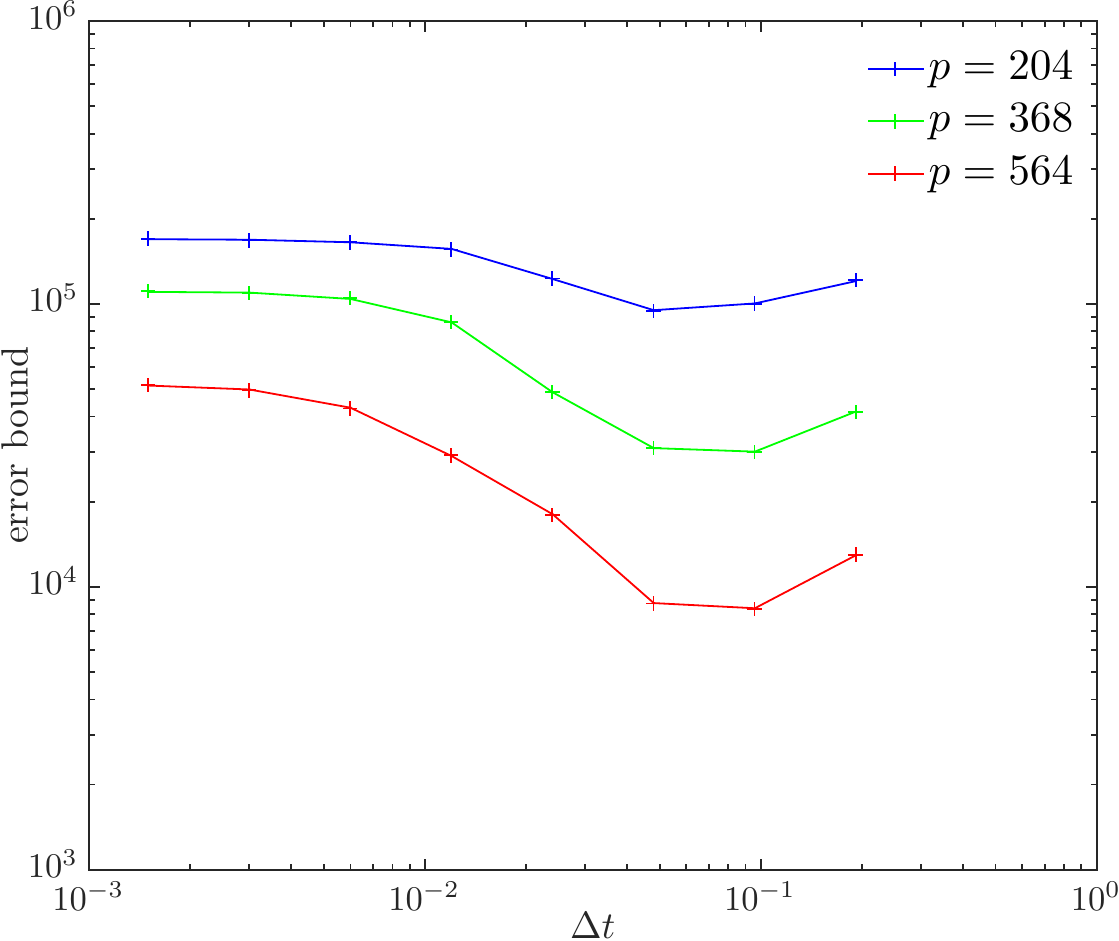} 
		 \label{fig:errorboundresult1} 
		}
	  \caption{Assessment of quantities appearing in error bound
		\eqref{eq:auxToPlot2}. This analysis suggests that an intermediate time step $\dt$ can reduce errors for
		the \reviewerA{LSPG} ROM.} 
		 \label{fig:errorboundresult} 
		  \end{figure} 

\subsection{LSPG ROM performance}

We now compare the accuracy and walltime performance of the \reviewerA{LSPG}
ROM as the dimension of the basis, time step, and time interval
change. The most salient result from Figure \ref{fig:pgSelfCompare} is that
choosing an intermediate time step leads to both better accuracy and
faster simulation times. This shows that our theoretical analysis of the error
bound performed in Section \ref{sec:experimentErrorBound} leads to an actual observed performance
improvement.
For example, consider the $\nstate=564$ case over the
time interval $0\leq t\leq 2.5$. In this case, a time step of
$\dt=1.875\times 10^{-4}$ leads to a relative error of $0.0140$ and a simulation
time of $289$ hours; increasing this value to $\dt=1.5\times 10^{-3}$ reduces the
relative error to $9.46\times 10^{-4}$ and the simulation time to $35.8$
hours, which constitutes roughly an order of magnitude improvement in both
quantities. Again, this supports the theoretical results of Corollary
\ref{corr:auxiliaryProblem} and highlights the critical importance of the
time step for \reviewerA{LSPG} reduced-order models.

In addition, Figure \ref{fig:pgSelfCompare} shows that 
as the basis dimension increases, the
optimal time step decreases; this was anticipated from the spectral analysis
performed in Section \ref{sec:spectralPOD}.
In addition, adding POD basis vectors does not improve accuracy for large time
steps.  We interpret this effect as follows: for larger time steps, the first
few POD modes accurately capture `coarse' phenomena on the scale of the time
step.  Therefore, accuracy improvement is not achieved by adding modes that
encode dynamics that evolve on a time scale finer than the time step itself.

Further, Figure \ref{fig:timeInterval} highlights that as the basis dimension
increases, the error generally decreases, which is an artifact of
\reviewerANew{the monotonic decrease in the FOM O$\Delta$E residual} achieved by the \reviewerA{LSPG} ROM (Remark
\ref{rem:discreteOptAPriori}).  Finally, the figure shows that as the time
interval grows, the optimal time step generally increases. 

\begin{figure}[htbp] 
\centering 
\subfigure[errors for $0\leq t\leq 2.5$]{
\includegraphics[width=0.3\textwidth]{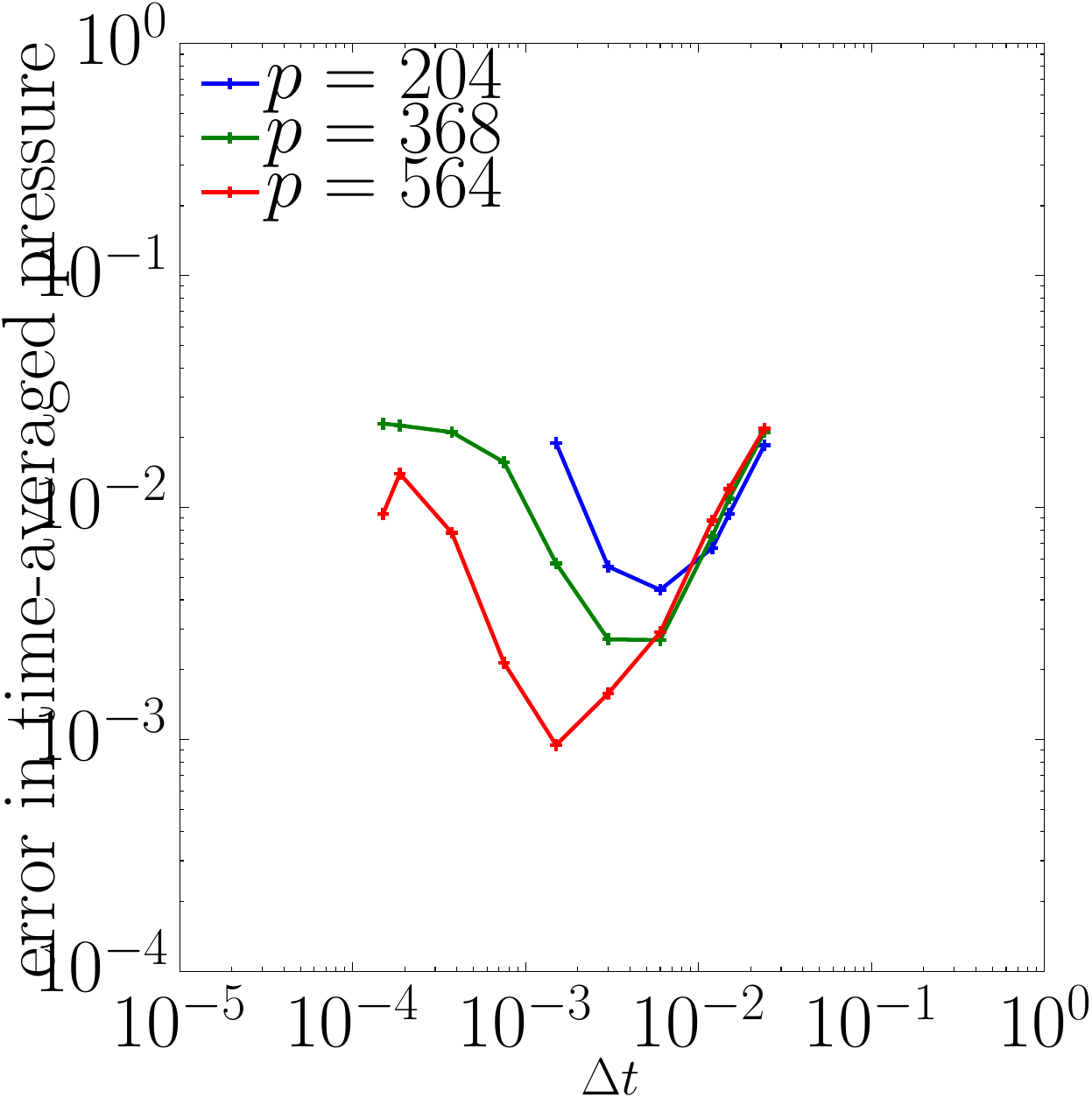}
}
\subfigure[errors for $0\leq t\leq 5.0$]{
\includegraphics[width=0.3\textwidth]{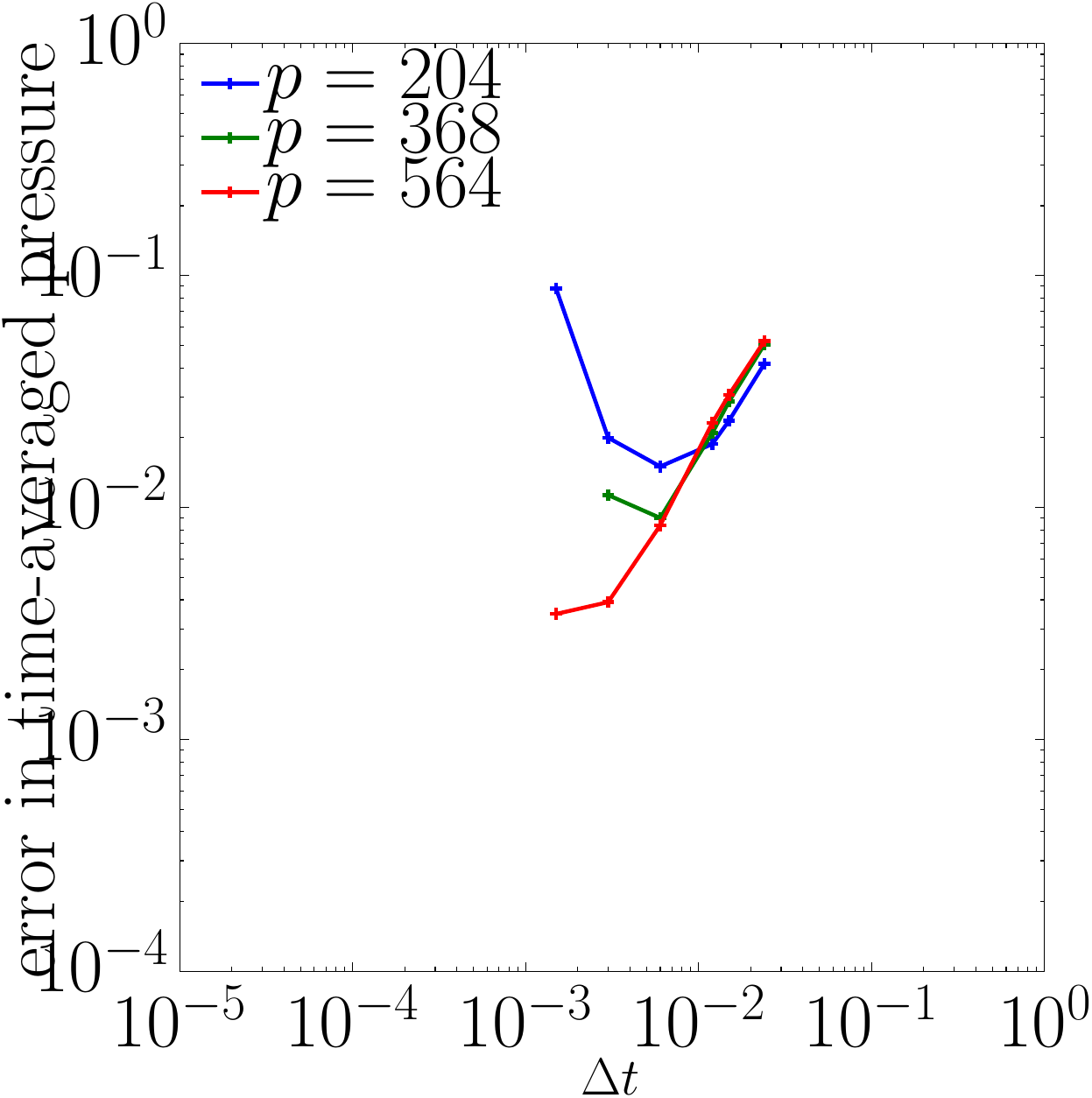}
}
\subfigure[errors for $0\leq t\leq 12.5$]{
\includegraphics[width=0.3\textwidth]{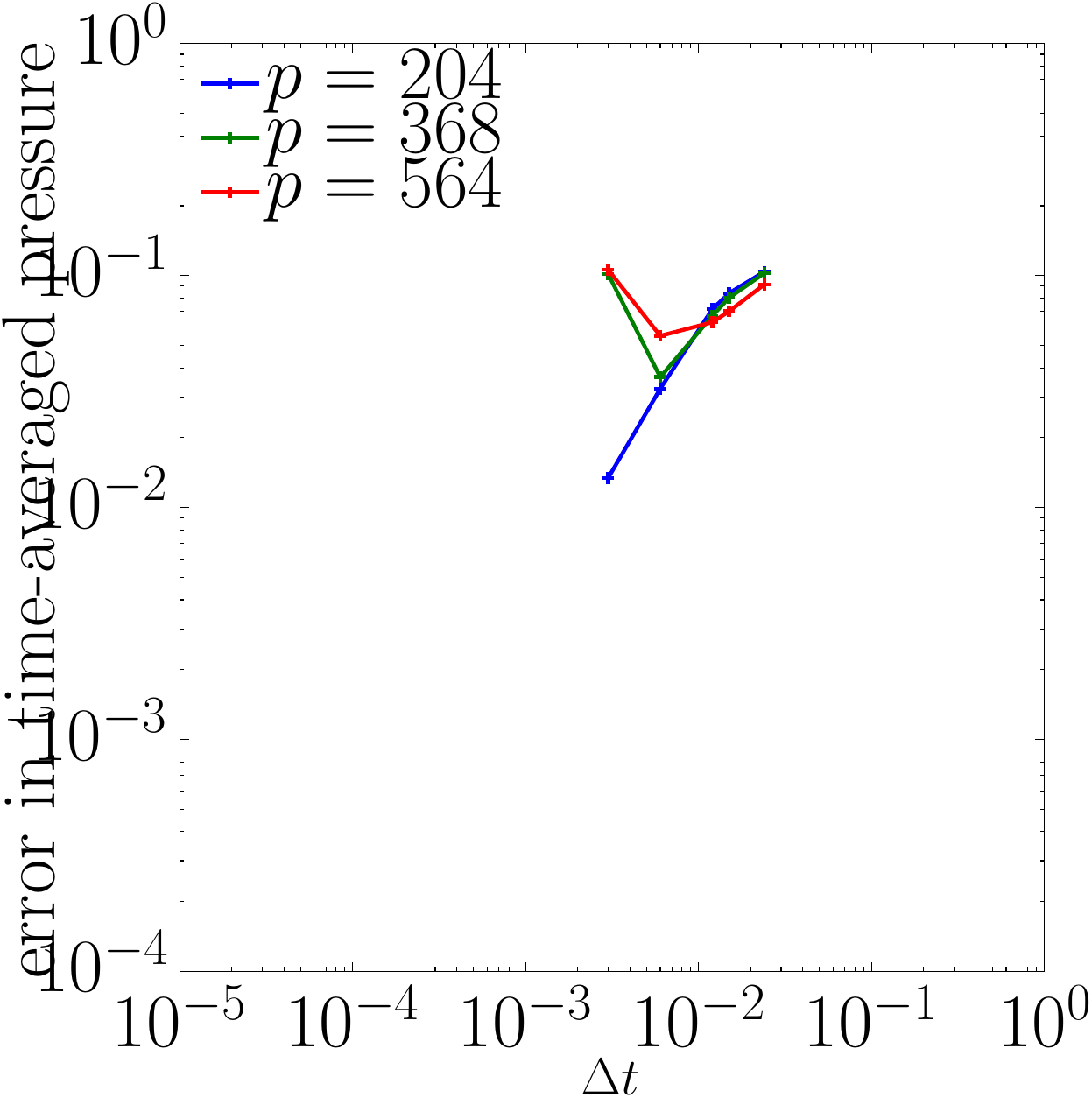}
}
\centering 
\subfigure[simulation times for $0\leq t\leq 2.5$]{
\includegraphics[width=0.3\textwidth]{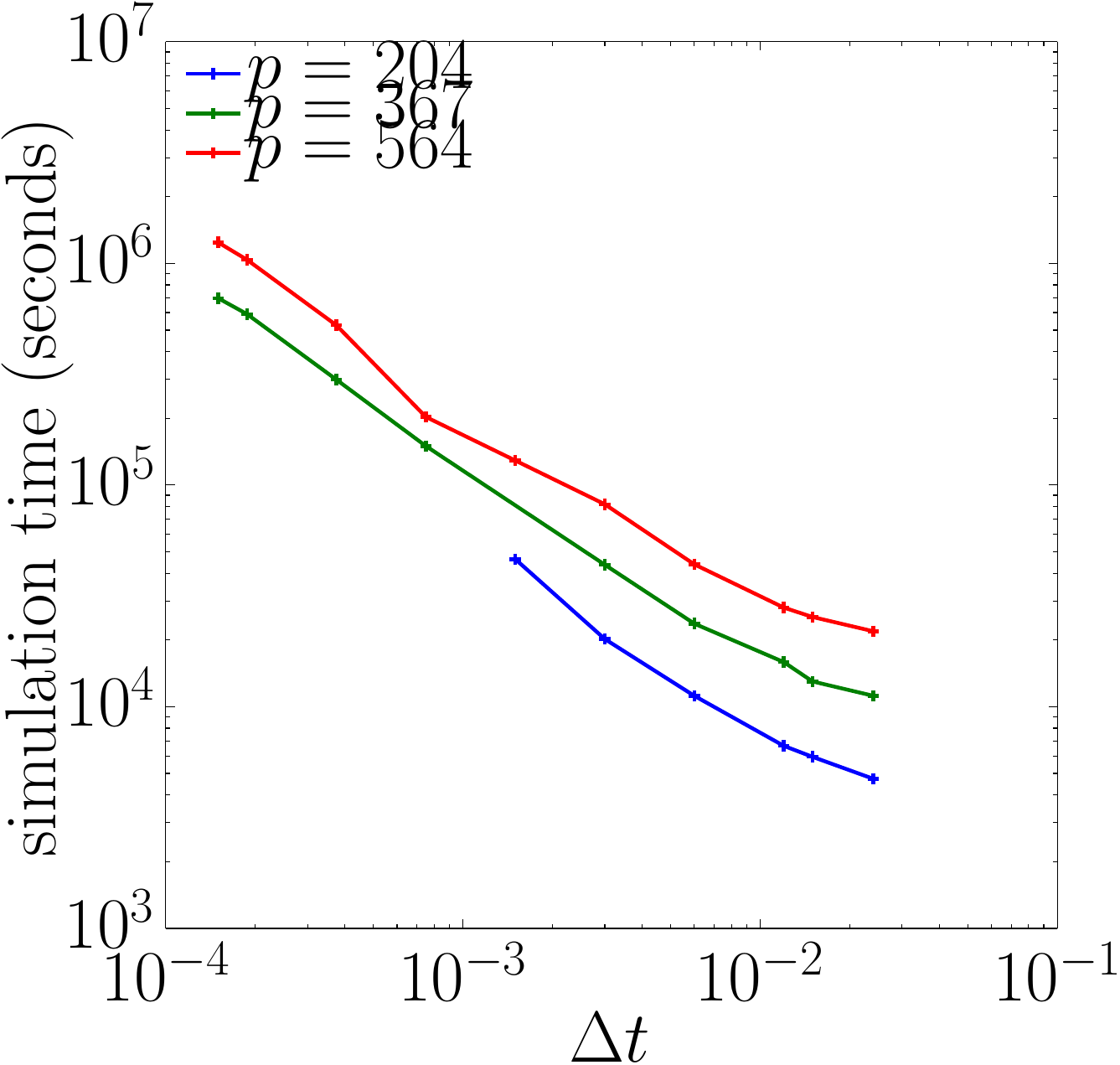}
}
\subfigure[simulation times for $0\leq t\leq 5.0$]{
\includegraphics[width=0.3\textwidth]{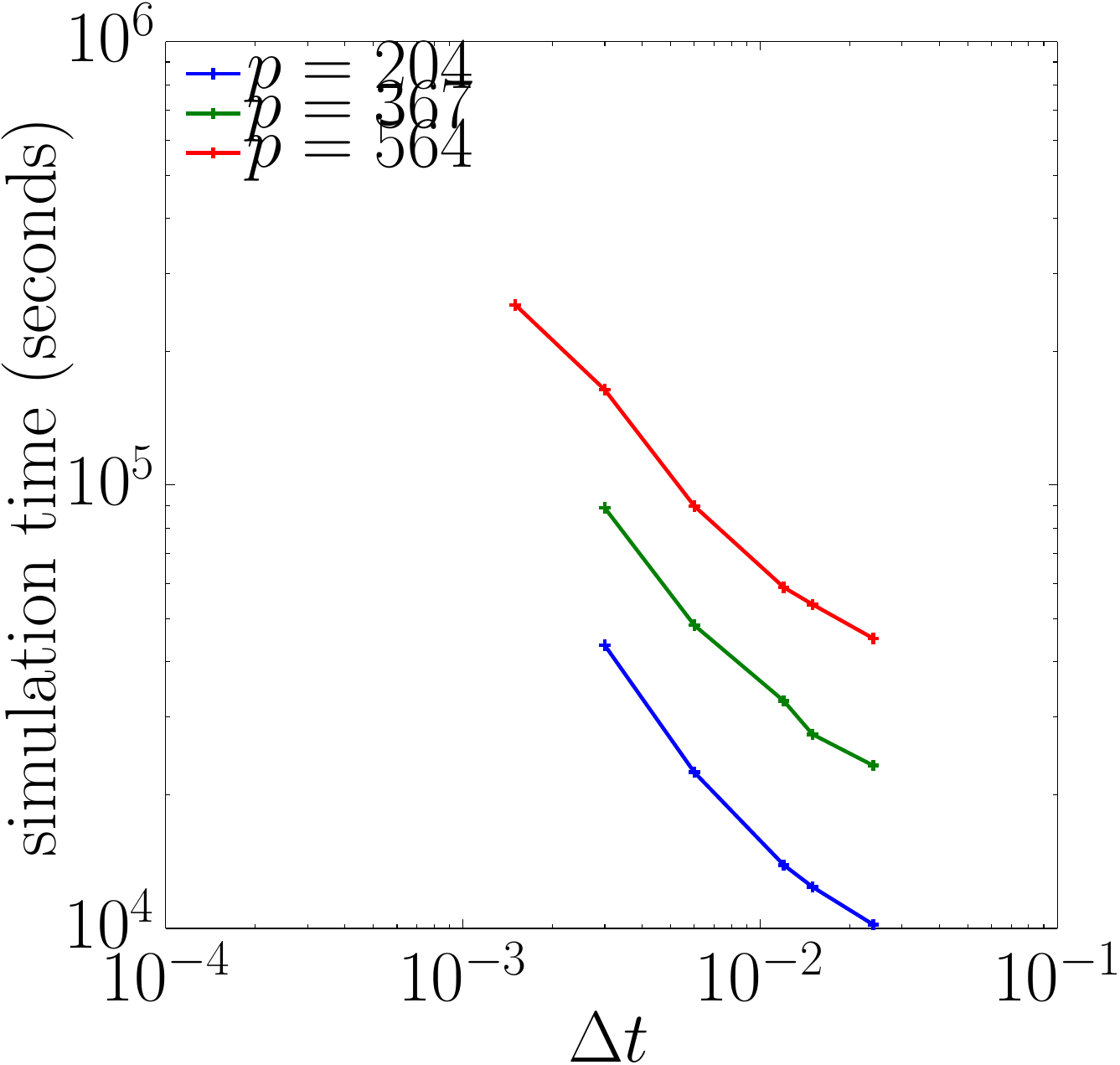}
}
\subfigure[simulation times for $0\leq t\leq 12.5$]{
\includegraphics[width=0.3\textwidth]{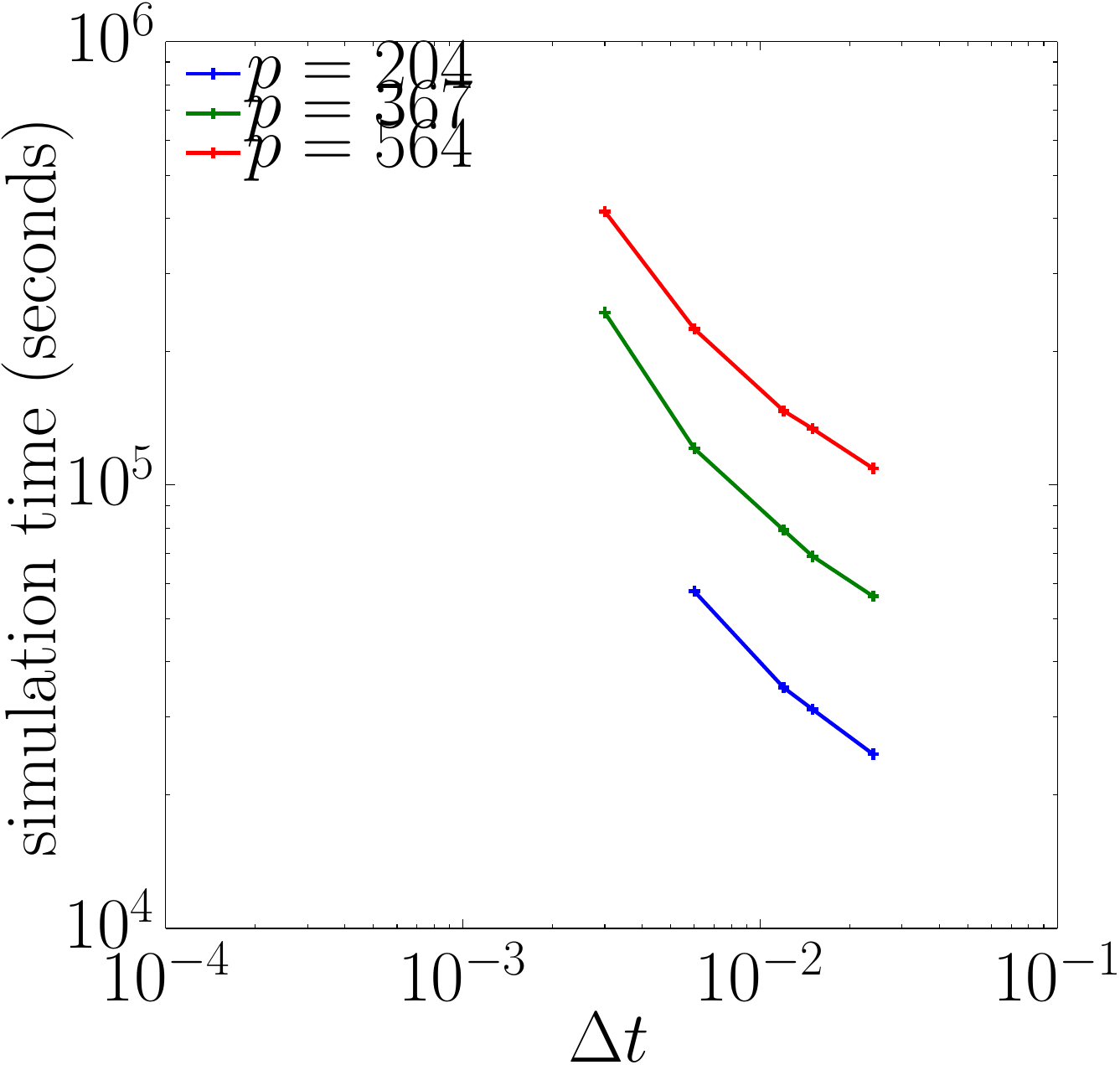}
}

\subfigure[Dependence of optimal time step (in terms of minimizing error)
on time interval and basis dimension]
{\includegraphics[width=0.5\textwidth]{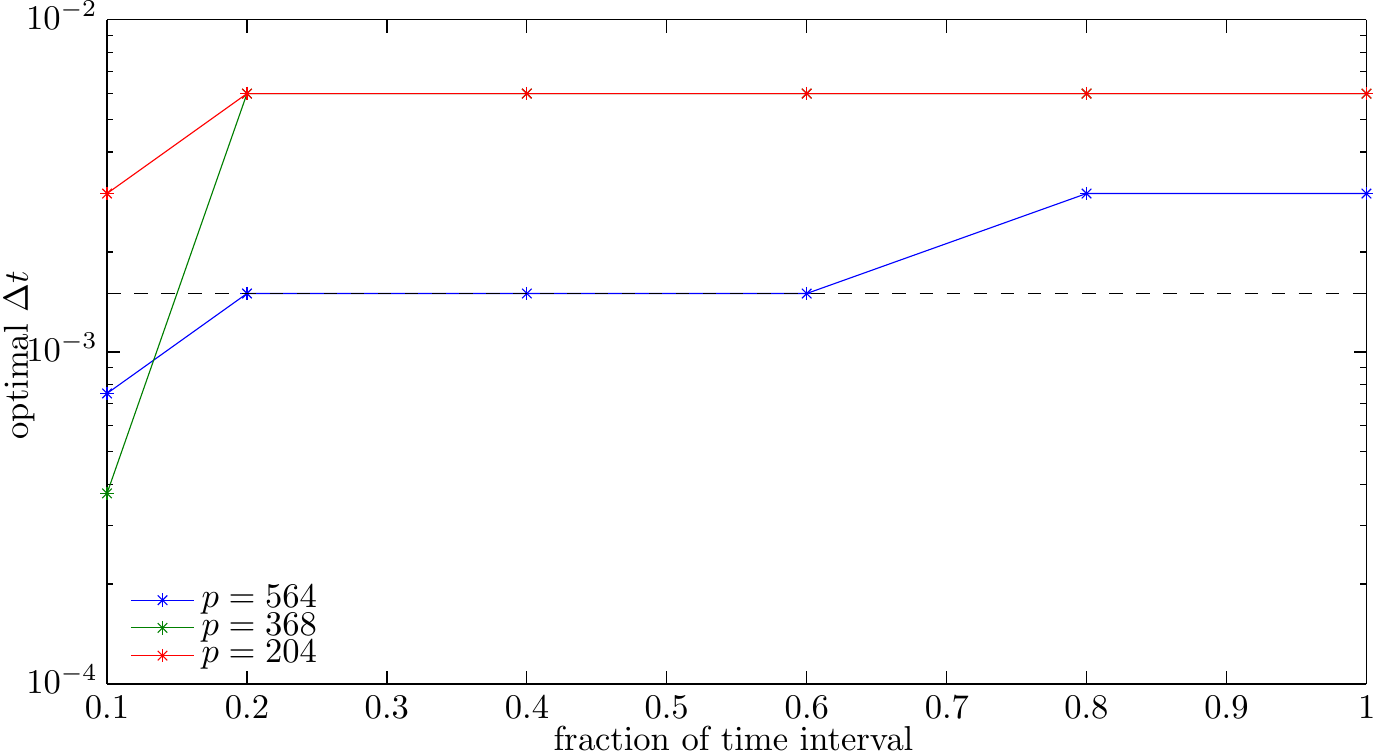} 
\label{fig:timeInterval}
}
\caption{
Dependence of error and simulation time for the \reviewerA{LSPG}
reduced-order model on the time step $\dt$, basis dimension, and time interval
} 
\label{fig:pgSelfCompare} 
\end{figure}

\subsection{GNAT: ROM with complexity reduction}

In this section, we perform a similar study, but equip the \reviewerA{LSPG}
ROM with complexity reduction in order to achieve computational savings. In
particular, we employ the GNAT method
\cite{CarlbergGappy,carlbergHawaii,carlbergJCP}, which solves
Eq.~\eqref{eq:discreteOptLinear} with $\weightingMatrix =
\left(\sampleMat\podres\right)^+\sampleMat$, where $\podres$ is a basis for
the residual and $\sampleMat$ consisting of selected rows of the identity
matrix.

The problem is identical to that described in Section \ref{sec:probDesc}
except that we take $\maxT=5.5$ \reviewerANew{time units} and employ a second-order
space-accurate dissipation scheme wherein a linear variation of the solution
is assumed within each control volume.\footnote{This is done to ensure the
sample mesh requires two layers of neighboring nodes for each sample
node.} For this simulation, the full-order model consumes 5.0 hours on 48
cores across six compute nodes.

To construct the trial basis $\podstate$ and basis for the residual $\podres$
for the GNAT models, we again employ POD.  In particular, we set
	$\podstate\leftarrow \podArgs{\snapsNo}{\energyCrit}$, where $\podArgsNo$ is
	computed via Algorithm \ref{alg:PODSVD} with snapshots consisting of the
	centered full-order model states $\snapsNo =
	\{\state_\star(k\dttruth)-\stateInitialNo\}_{k=1}^{3668}$. An energy criterion of
	$\energyCrit=1-10^{-5}$ ($\nstate = 179$) is used during the experiments.
	For the residual, we employ $\podres\leftarrow
	\podArgs{\snapsNo_\res}{\energyCrit_\res}$ via Algorithm \ref{alg:PODSVD}
	with snapshots $\snapsNo_\res = \{\res^\timestepit (\stateInitialNo +
	\podstate\unknownRed^{\timestepit(k)}),\ k\in\nat{K(\timestepit)},\
	\timestepit\in\nat{2228}\}$ and $\unknownRed^{\timestepit(k)}$ corresponding
	to the \reviewerA{LSPG} ROM solution at Gauss--Newton iteration $k$ within
	time step $\timestepit$ using a time step of $\dt = 6\times 10^{-3}$.  Here,
	$K(\timestepit)$ denotes the number of Newton iterations required for
	convergence of at time instance $\timestepit$. An
	energy criterion of $\energyCrit_\res=1.0$ is employed. In addition, the GNAT model
	sets the Jacobian basis equal to
	residual basis $\podjac=\podres$ and employs 
	$n_s=743$ sample nodes that define $\sampleMat$, which leads to 4458 rows in
	$\sampleMat$ as there are six
	conservation equations per node due to the turbulence model (see
	Ref.~\cite{carlbergJCP} for definitions).


The GNAT implementation in AERO-F is characterized by the sample-mesh concept
\cite{carlbergJCP}. Figure \ref{fig:sampleMesh} depicts the sample mesh for
this problem, which was constructed using $n_c = 2228$ working columns
\cite[Algorithm 3]{carlbergJCP}, and includes two layers of nodes around the
sample nodes (to enable the residual to be computed at the sample nodes). It
is characterized by 7,974 total nodes (4.1\% of the original mesh) and 17,070
total volumes (3.0\% of the original mesh). Due to the small footprint of the
sample mesh, the GNAT simulations are run using only 2 cores on a single
compute node. 

\begin{figure}[htbp] 
\centering 
\subfigure[Full domain]{
	\includegraphics[width=1.0\textwidth]{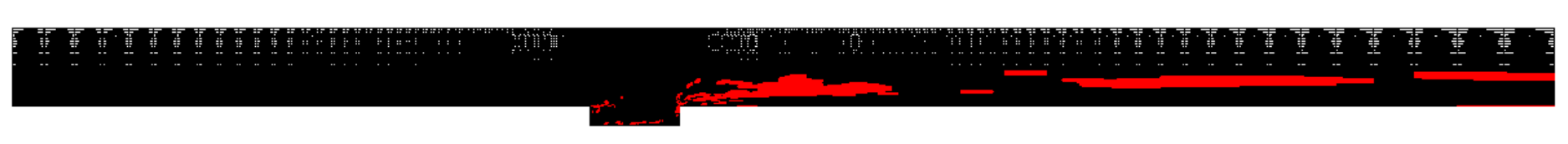} }
\subfigure[Zoom on cavity]{
	\includegraphics[width=0.5\textwidth]{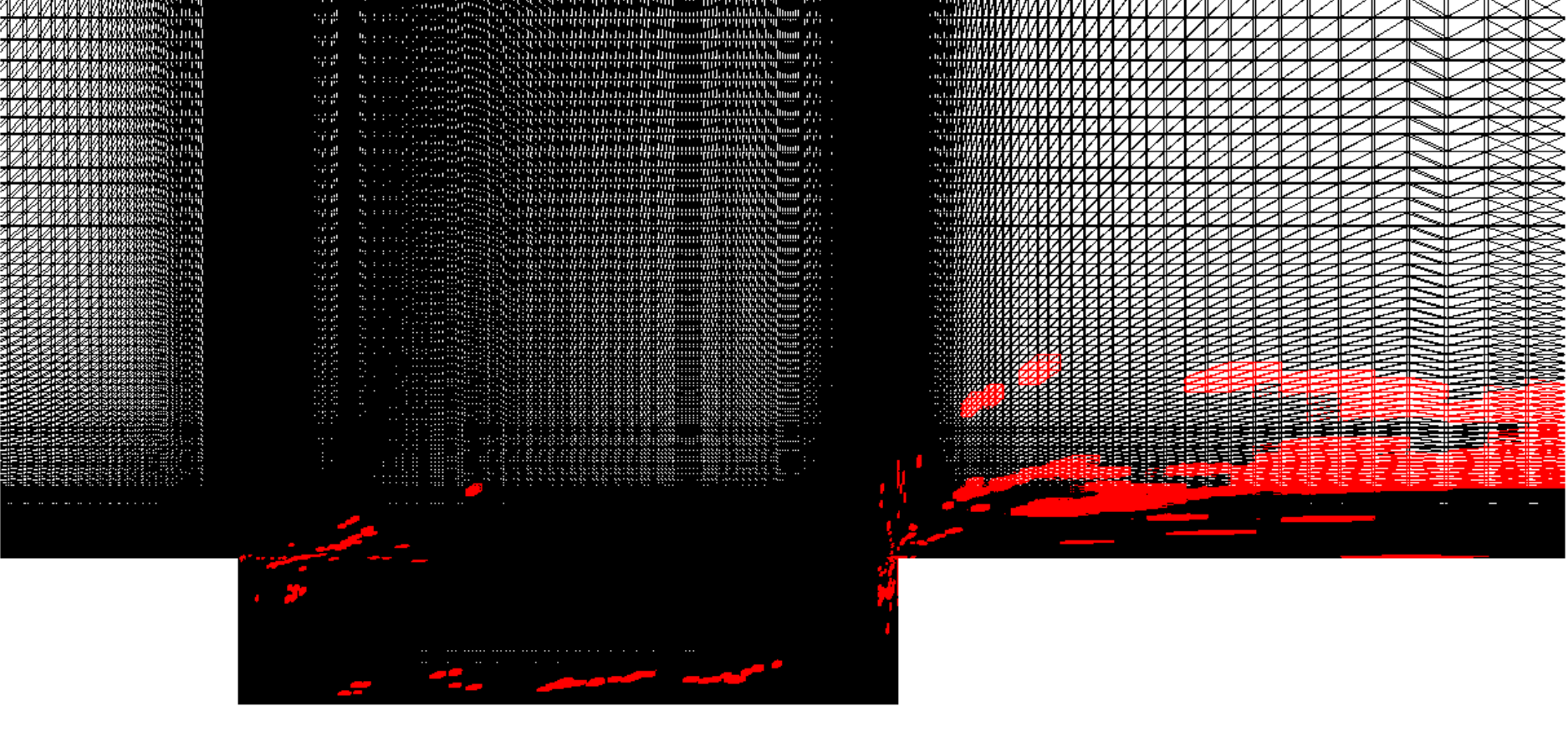} }
\caption{Sample mesh (red) embedded within original mesh. \reviewerC{The
sample mesh is used to to enable low-computational-footprint simulations with
the GNAT reduced-order model.}} \label{fig:sampleMesh} 
\end{figure} 

Figure \ref{fig:gnat} reports the results obtained with the GNAT ROM using
different time steps. Critically, note that the GNAT ROM also exhibits a `dip'
in the optimal time step, with a time step of $6.0\times 10^{-3}$ yielding the
lowest error. In fact, increasing the time step from $1.5\times 10^{-3}$ to
$6.0\times 10^{-3}$ decreases the error from 3.32\% to 2.25\% and also
significantly increases the
computational savings relative to the full-order model (as measured in core--hours) from 14.9 to 55.7. This
highlights that the analysis is also relevant to ROMs equipped with
complexity reduction.
\begin{figure}[htbp] 
\centering 
\subfigure[\reviewerA{responses}]{
\includegraphics[width=0.8\textwidth]{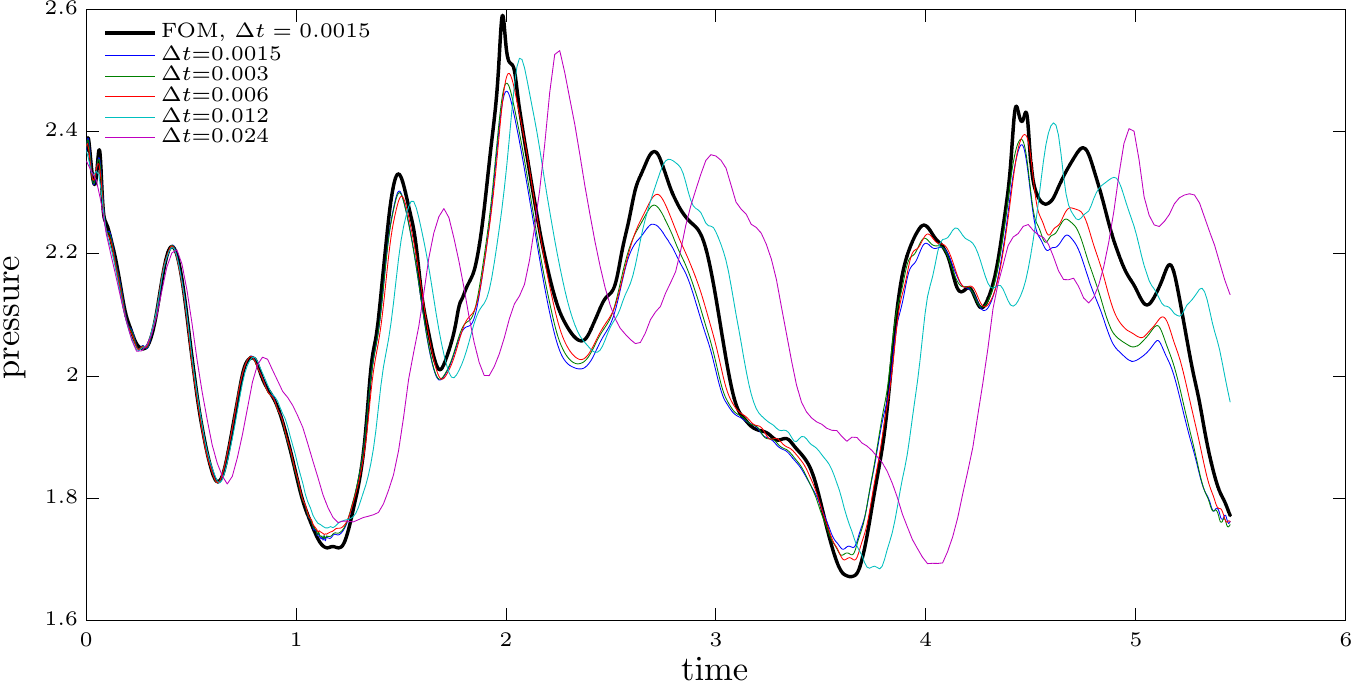} }
\subfigure[\reviewerA{errors}]{
\includegraphics[height=0.3\textwidth]{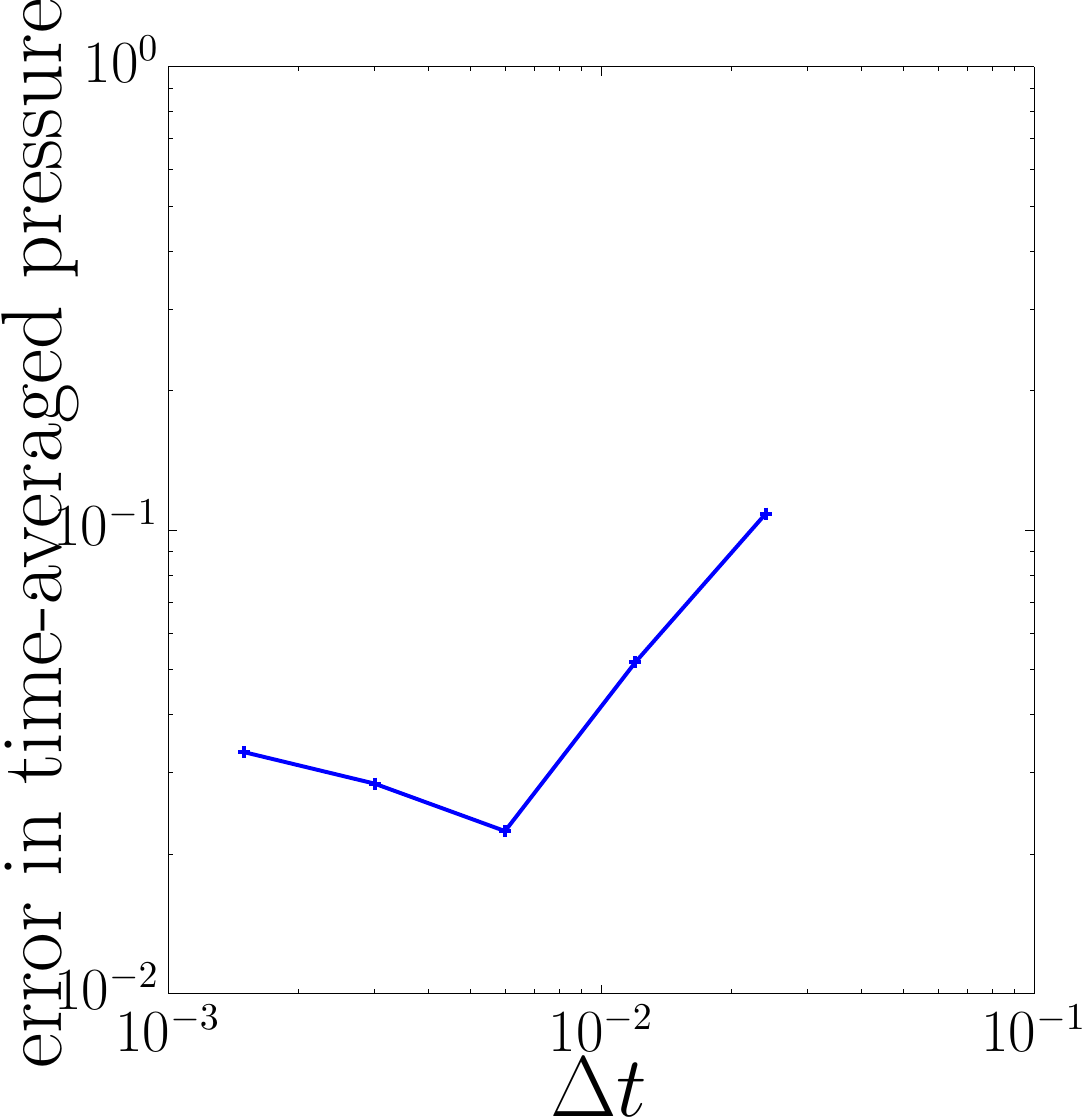} }
\subfigure[\reviewerA{computational savings}]{
\includegraphics[height=0.29\textwidth]{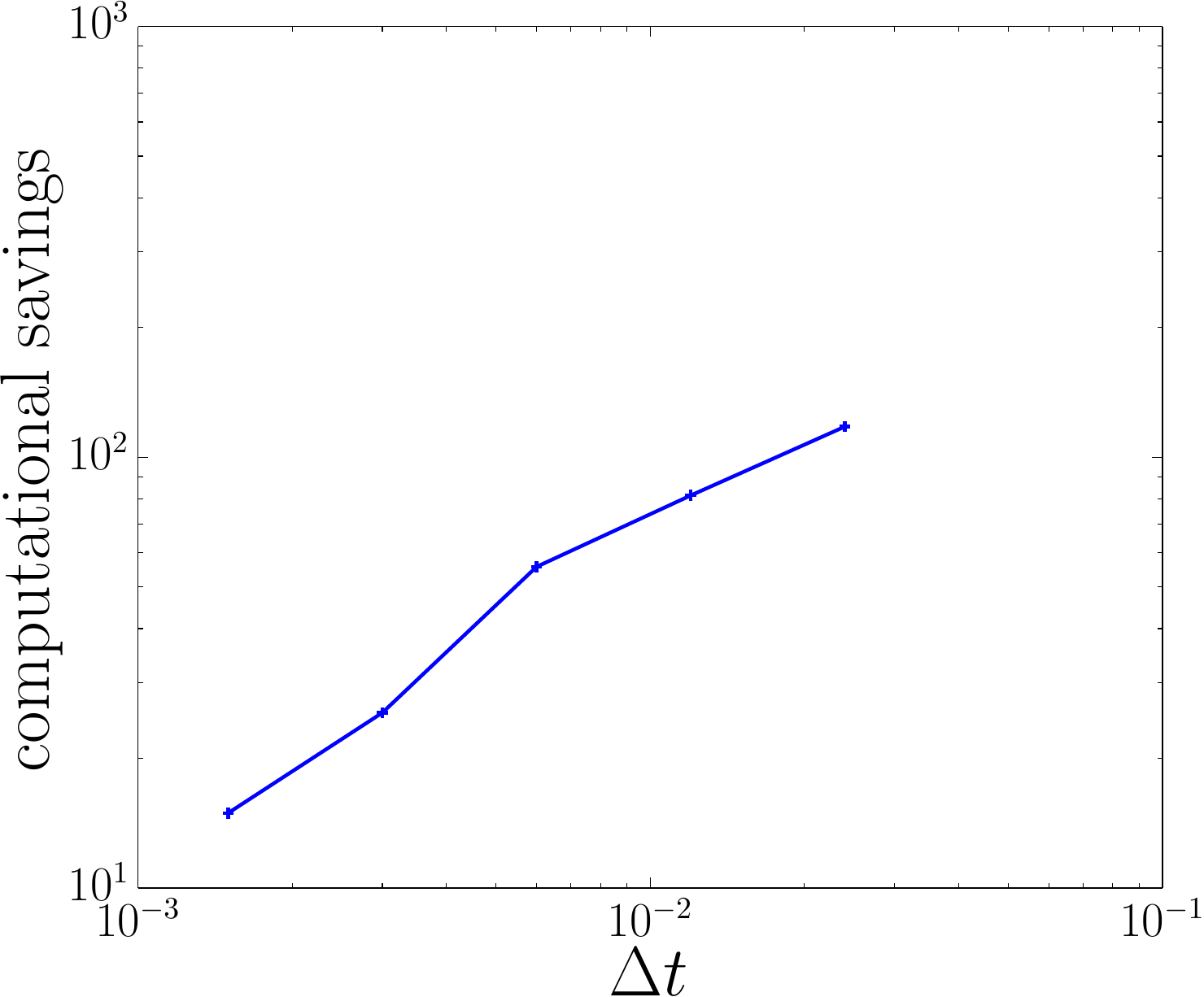} }
\caption{\reviewerA{Responses, errors
$\errorValue(\pressure_{\text{GNAT}},\pressure_{\text{FOM}_\star})$, and
computational savings (as measured in core--hours) produced by the GNAT reduced-order model for different
time steps $\dt $.}} 
\label{fig:gnat} 
\end{figure} 
\subsection{Summary of experimental results}
\noindent We now briefly summarize the main experimental results:
\begin{itemize} 
\item Galerkin ROMs are unstable for long time intervals (Figure
\ref{fig:ROMresponses}).
\item \reviewerA{LSPG} ROMs are only unstable for small time steps (Figure
\ref{fig:ROMresponses}).
\item Galerkin and \reviewerA{LSPG} ROMs are equivalent as $\dt\rightarrow 0$
(Figure \ref{fig:GalDiscOptErr}).
\item \reviewerA{LSPG} ROMs are more accurate than Galerkin ROMs over small
time windows where Galerkin is stable (Figure \ref{fig:galPgErrors}).
\item \reviewerA{LSPG} ROMs are most accurate for an intermediate time step
(Figure \ref{fig:galPgErrors}).
\item Adding POD modes has the effect of including higher-frequency response
components (Figure \ref{fig:spectralContent}).
\item The theoretical error bound for the \reviewerA{LSPG} ROM exhibits the
same time step `dip' as the experimentally observed error (Figure
\ref{fig:errorboundresult}).
\item The optimal time step for the \reviewerA{LSPG} ROM decreases as modes
are added to the POD basis (Figure \ref{fig:pgSelfCompare}).
\item Adding modes to the POD basis has little effect on \reviewerA{LSPG} ROM accuracy for large
time steps (Figure \ref{fig:pgSelfCompare}).
\item The optimal time step for the \reviewerA{LSPG} ROM tends to increase as
the time interval increases (Figure \ref{fig:timeInterval}).
\item The GNAT ROM, which is discrete optimal and is equipped
with complexity reduction, also produces minimal error for an intermediate
time step (Figure \ref{fig:gnat}).
\end{itemize}

\section{Conclusions}\label{sec:conclusions}

This work has performed a comparative theoretical and experimental analysis of
Galerkin and \reviewerA{LSPG} reduced-order models for linear multistep
schemes and Runge--Kutta schemes. We have demonstrated a number of new
findings that have important practical implications, including conditions
under which the \reviewerA{LSPG} ROM has a time-continuous representation,
conditions under which the two techniques are equivalent, and time-discrete
error bounds for the two approaches. 

Perhaps most surprisingly, we demonstrated that decreasing the time step does
not necessarily decrease the error for the \reviewerA{LSPG} ROM. This
phenomenon arose in both the theoretical analysis and in numerical
experiments. In particular, our results suggest that the time step should be
`matched' to the spectral content of the reduced basis. In the experiments, we
showed that increasing the time step to an intermediate value decreased both
the error and the simulation time by an order of magnitude in certain cases.
Alternatively, decreasing the time step cause the \reviewerA{LSPG} ROM to
become unstable for longer time intervals. This highlights the critical
importance of time-step selection for \reviewerA{LSPG} ROMs.

\section*{Acknowledgments}

We thank Prof.\ Stephen Pope for insightful conversations related to comparing
Galerkin and \reviewerA{least-squares Petrov--Galerkin} reduced-order models; these conversations
inspired this work. \reviewerC{We also thank the anonymous reviewers for their
extremely helpful and insightful comments and suggestions.} We also thank Prof.\ Charbel Farhat for permitting us the
use of AERO-F, as well as Julien Cortial, Charbel Bou-Mosleh, and David
Amsallem for their previous contributions in implementing nonlinear
reduced-order models in AERO-F.  K.~Carlberg acknowleges an appointment
to the Sandia National Laboratories Truman Fellowship in National Security
Science and Engineering.  The Truman Fellowship is sponsored by Sandia
National Laboratories. Sandia National Laboratories is a multi-program
laboratory managed and operated by Sandia Corporation, a wholly owned
subsidiary of Lockheed Martin Corporation, for the U.S.\ Department of
Energy's National Nuclear Security Administration under contract
DE-AC04-94AL85000. \reviewerA{H.\ Antil acknowledges the support by the NSF-DMS-1521590.
} The content of this publication does not necessarily
reflect the position or policy of any of these institutions, and no official
endorsement should be inferred.

\section*{Appendix}
Algorithm \ref{alg:PODSVD} reports the algorithm for computing a POD basis
using normalized snapshots.
\begin{algorithm}[htbp]
\caption{Proper-orthogonal-decomposition basis computation (normalized
snapshots)}
\begin{algorithmic}[1]\label{alg:PODSVD}
\REQUIRE Set of snapshots $\snapsNo\equiv\{\w _i\}_{i=1}^\nsnap\subset\RR{\ndof }$,
energy criterion $\energyCrit\in[0,1]$
\ENSURE  $\podArgs{\snapsNo}{\energyCrit}$
\STATE\label{step:SVD} Compute thin singular value decomposition $
\snapmat= \boldsymbol U \boldsymbol \Sigma \boldsymbol V^T $, where
$\boldsymbol W\equiv\left[\w _1/\|\w _1\|\ \cdots\
\w _{n_\w }/\|\w _{n_\w }\|\right]$.
\STATE Choose dimension of truncated basis
$\nstate = \nenergy(\nu)$, where
 \begin{align*} 
 \nenergy(\nu) &\reviewerC{=} \arg\min_{i\in \mathcal V(\nu)}i\\
 \mathcal V(\nu)&\reviewerC{\defeq} \{n\in\nat{\nsnap}\ | \
 \sum_{i=1}^n\sigma_i^2/\sum_{i=1}^{\nsnap}\reviewerA{\sigma_i^2}\geq\nu\},
  \end{align*}
  and $\boldsymbol \Sigma \equiv \text{diag}\left(\sigma_i\right)$.
\STATE $\podArgs{\snapsNo}{\energyCrit}=\vecmat{\boldsymbol u}{\nstate}$,
where $\boldsymbol U
\equiv\vecmat{\boldsymbol u}{\nsnap}$.
\end{algorithmic}
\end{algorithm}

\bibliography{references}

\begin{thebibliography}{10}

\bibitem{abgrall2015robust}
{\sc R.~Abgrall and D.~Amsallem}, {\em Robust model reduction by $l^1$-norm
  minimization and approximation via dictionaries: application to linear and
  nonlinear hyperbolic problems}, Stanford University Preprint,  (2015).

\bibitem{alexander1977diagonally}
{\sc R.~Alexander}, {\em Diagonally implicit runge-kutta methods for stiff
  {ODE}'s}, SIAM Journal on Numerical Analysis, 14 (1977), pp.~1006--1021.

\bibitem{amsallemLocalGnat}
{\sc D.~Amsallem, M.~Zahr, and K.~Washabaugh}, {\em Fast local reduced basis
  updates for the efficient reduction of nonlinear systems with
  hyper-reduction}, Advances in Computational Mathematics, Special Issue on
  Model Reduction of Parameterized Systems (2015).

\bibitem{an2008optimizing}
{\sc S.~An, T.~Kim, and D.~James}, {\em Optimizing cubature for efficient
  integration of subspace deformations}, ACM Transactions on Graphics (TOG), 27
  (2008), p.~165.

\bibitem{HAntil_SField_RHNochetto_MTiglio_2013}
{\sc H.~Antil, S.~Field, F.~Herrmann, R.~Nochetto, and M.~Tiglio}, {\em
  Two-step greedy algorithm for reduced order quadratures}, Journal of
  Scientific Computing, 57 (2013), pp.~604--637.

\bibitem{HAntil_MHeinkenschloss_DCSorensen_2013a}
{\sc H.~Antil, M.~Heinkenschloss, and D.~C. Sorensen}, {\em Application of the
  discrete empirical interpolation method to reduced order modeling of
  nonlinear and parametric systems}, vol.~8 of Springer MS\&A series: Reduced
  Order Methods for modeling and computational r G. Rozza, Eds, Springer-Verlag
  Italia, Milano, 2013.

\bibitem{astrid2007mpe}
{\sc P.~Astrid, S.~Weiland, K.~Willcox, and T.~Backx}, {\em Missing point
  estimation in models described by proper orthogonal decomposition}, IEEE
  Transactions on Automatic Control, 53 (2008), pp.~2237--2251.

\bibitem{aubry1988dynamics}
{\sc N.~Aubry, P.~Holmes, J.~L. Lumley, and E.~Stone}, {\em The dynamics of
  coherent structures in the wall region of a turbulent boundary layer},
  Journal of Fluid Mechanics, 192 (1988), pp.~115--173.

\bibitem{balajewicz2012stabilization}
{\sc M.~Balajewicz and E.~Dowell}, {\em Stabilization of projection-based
  reduced order models of the navier--stokes equations}, Nonlinear Dynamics, 70
  (2012), pp.~1619--1632.

\bibitem{balajewicz2013low}
{\sc M.~Balajewicz, E.~Dowell, and B.~Noack}, {\em Low-dimensional modelling of
  high-reynolds-number shear flows incorporating constraints from the
  navier--stokes equation}, Journal of Fluid Mechanics, 729 (2013),
  pp.~285--308.

\bibitem{barone2009stable}
{\sc M.~F. Barone, I.~Kalashnikova, D.~J. Segalman, and H.~K. Thornquist}, {\em
  Stable {G}alerkin reduced order models for linearized compressible flow},
  Journal of Computational Physics, 228 (2009), pp.~1932--1946.

\bibitem{barrault2004eim}
{\sc M.~Barrault, Y.~Maday, N.~C. Nguyen, and A.~T. Patera}, {\em An `empirical
  interpolation' method: application to efficient reduced-basis discretization
  of partial differential equations}, Comptes Rendus Math\'ematique Acad\'emie
  des Sciences, 339 (2004), pp.~667--672.

\bibitem{surveyWillcoxGugercin}
{\sc P.~Benner, S.~Gugercin, and K.~Willcox}, {\em A survey of model reduction
  methods for parametric systems}, SIAM Review, 57 (2015), pp.~483--531.

\bibitem{Bergmann2009516}
{\sc M.~Bergmann, C.-H. Bruneau, and A.~Iollo}, {\em Enablers for robust pod
  models}, Journal of Computational Physics, 228 (2009), pp.~516 -- 538.

\bibitem{scalapack}
{\sc L.~Blackford, A.~Cleary, J.~Choi, E.~d'Azevedo, J.~Demmel, I.~Dhillon,
  J.~Dongarra, S.~Hammarling, G.~Henry, A.~Petitet, et~al.}, {\em {ScaLAPACK
  Users' Guide}}, Society for Industrial and Applied Mathematics, 1997.

\bibitem{bos2004als}
{\sc R.~Bos, X.~Bombois, and P.~Van~den Hof}, {\em Accelerating large-scale
  non-linear models for monitoring and control using spatial and temporal
  correlations}, Proceedings of the American Control Conference, 4 (2004),
  pp.~3705--3710.

\bibitem{bui2008model}
{\sc T.~Bui-Thanh, K.~Willcox, and O.~Ghattas}, {\em {Model reduction for
  large-scale systems with high-dimensional parametric input space}}, SIAM
  Journal on Scientific Computing, 30 (2008), pp.~3270--3288.

\bibitem{carlbergThesis}
{\sc K.~Carlberg}, {\em Model reduction of nonlinear mechanical systems via
  optimal projection and tensor approximation}, PhD thesis, Stanford
  University, August 2011.

\bibitem{CarlbergGappy}
{\sc K.~Carlberg, C.~Bou-Mosleh, and C.~Farhat}, {\em {Efficient non-linear
  model reduction via a least-squares Petrov--Galerkin projection and
  compressive tensor approximations}}, International Journal for Numerical
  Methods in Engineering, 86 (2011), pp.~155--181.

\bibitem{carlbergHawaii}
{\sc K.~Carlberg, J.~Cortial, D.~Amsallem, M.~Zahr, and C.~Farhat}, {\em The
  {GNAT} nonlinear model reduction method and its application to fluid dynamics
  problems}, AIAA Paper 2011-3112, 6th AIAA Theoretical Fluid Mechanics
  Conference, Honolulu, HI,  (2011).

\bibitem{carlbergJCP}
{\sc K.~Carlberg, C.~Farhat, J.~Cortial, and D.~Amsallem}, {\em The {GNAT}
  method for nonlinear model reduction: effective implementation and
  application to computational fluid dynamics and turbulent flows}, Journal of
  Computational Physics, 242 (2013), pp.~623--647.

\bibitem{carlbergStructureAiaa}
{\sc K.~Carlberg, R.~Tuminaro, and P.~Boggs}, {\em Efficient
  structure-preserving model reduction for nonlinear mechanical systems with
  application to structural dynamics}, in AIAA Paper 2012-1969, 53rd
  AIAA/ASME/ASCE/AHS/ASC Structures, Structural Dynamics and Materials
  Conference, Honolulu, Hawaii, April 23--26 2012.

\bibitem{carlberg2012spd}
{\sc K.~Carlberg, R.~Tuminaro, and P.~Boggs}, {\em Preserving {L}agrangian
  structure in nonlinear model reduction with application to structural
  dynamics}, SIAM J. Sci. Comput., 37 (2015), pp.~B153---B184.

\bibitem{chaturantabut2010journal}
{\sc S.~Chaturantabut and D.~C. Sorensen}, {\em Nonlinear model reduction via
  discrete empirical interpolation}, SIAM Journal on Scientific Computing, 32
  (2010), pp.~2737--2764.

\bibitem{constantineResMin}
{\sc P.~Constantine and Q.~Wang}, {\em Residual minimizing model reduction for
  parameterized nonlinear dynamical systems}, SIAM J. Sci. Comput., 34 (2012),
  pp.~A2118--A2144.

\bibitem{drohmannEOI}
{\sc M.~Drohmann, B.~Haasdonk, and M.~Ohlberger}, {\em Reduced basis
  approximation for nonlinear parametrized evolution equations based on
  empirical operator interpolation}, SIAM Journal on Scientific Computing, 34
  (2012), pp.~A937--A969.

\bibitem{sirovichOrigGappy}
{\sc R.~Everson and L.~Sirovich}, {\em {K}arhunen--{L}o\`{e}ve procedure for
  gappy data}, Journal of the Optical Society of America A, 12 (1995),
  pp.~1657--1664.

\bibitem{Fang2013540}
{\sc F.~Fang, C.~Pain, I.~Navon, A.~Elsheikh, J.~Du, and D.~Xiao}, {\em
  Non-linear {P}etrov--{G}alerkin methods for reduced order hyperbolic
  equations and discontinuous finite element methods}, Journal of Computational
  Physics, 234 (2013), pp.~540 -- 559.

\bibitem{farhat2014dimensional}
{\sc C.~Farhat, P.~Avery, T.~Chapman, and J.~Cortial}, {\em Dimensional
  reduction of nonlinear finite element dynamic models with finite rotations
  and energy-based mesh sampling and weighting for computational efficiency},
  International Journal for Numerical Methods in Engineering, 98 (2014),
  pp.~625--662.

\bibitem{farhat2003application}
{\sc C.~Farhat, P.~Geuzaine, and G.~Brown}, {\em Application of a three-field
  nonlinear fluid-structure formulation to the prediction of the aeroelastic
  parameters of an {F-16} fighter}, Computers \& Fluids, 32 (2003), pp.~3--29.

\bibitem{foias1991dissipativity}
{\sc C.~Foias, M.~Jolly, I.~Kevrekidis, and E.~Titi}, {\em Dissipativity of
  numerical schemes}, Nonlinearity, 4 (1991), p.~591.

\bibitem{galbally2009non}
{\sc D.~Galbally, K.~Fidkowski, K.~Willcox, and O.~Ghattas}, {\em {Non-linear
  model reduction for uncertainty quantification in large-scale inverse
  problems}}, International Journal for Numerical Methods in Engineering, 81
  (2009), pp.~1581--1608.

\bibitem{galletti2004low}
{\sc B.~Galletti, C.~Bruneau, L.~Zannetti, and A.~Iollo}, {\em Low-order
  modelling of laminar flow regimes past a confined square cylinder}, Journal
  of Fluid Mechanics, 503 (2004), pp.~161--170.

\bibitem{geuzaine2003aeroelastic}
{\sc P.~Geuzaine, G.~Brown, C.~Harris, and C.~Farhat}, {\em Aeroelastic dynamic
  analysis of a full {F}-16 configuration for various flight conditions}, AIAA
  Journal, 41 (2003), pp.~363--371.

\bibitem{POD}
{\sc P.~Holmes, J.~Lumley, and G.~Berkooz}, {\em Turbulence, Coherent
  Structures, Dynamical Systems and Symmetry}, Cambridge University Press,
  1996.

\bibitem{iollo2000stability}
{\sc A.~Iollo, S.~Lanteri, and J.~A. Desideri}, {\em Stability properties of
  {POD}--{Galerkin} approximations for the compressible {N}avier--{S}tokes
  equations}, Theoretical and Computational Fluid Dynamics, 13 (2000),
  pp.~377--396.

\bibitem{jolly1991preserving}
{\sc M.~Jolly, I.~Kevrekidis, and E.~Titi}, {\em Preserving dissipation in
  approximate inertial forms for the {K}uramoto-{S}ivashinsky equation},
  Journal of Dynamics and Differential Equations, 3 (1991), pp.~179--197.

\bibitem{kalashnikova2010stability}
{\sc I.~Kalashnikova and M.~Barone}, {\em On the stability and convergence of a
  {G}alerkin reduced order model (rom) of compressible flow with solid wall and
  far-field boundary treatment}, International Journal for Numerical Methods in
  Engineering, 83 (2010), pp.~1345--1375.

\bibitem{lall2003structure}
{\sc S.~Lall, P.~Krysl, and J.~Marsden}, {\em Structure-preserving model
  reduction for mechanical systems}, Physica D: Nonlinear Phenomena, 184
  (2003), pp.~304--318.

\bibitem{LeGresleyThesis}
{\sc P.~A. LeGresley}, {\em Application of proper orthogonal decomposition
  ({POD}) to design decomposition methods}, PhD thesis, Stanford University,
  2006.

\bibitem{marion1989nonlinear}
{\sc M.~Marion and R.~Temam}, {\em Nonlinear {G}alerkin methods}, SIAM Journal
  on Numerical Analysis, 26 (1989), pp.~1139--1157.

\bibitem{nguyen2005certified}
{\sc N.~Ngoc~Cuong, K.~Veroy, and A.~T. Patera}, {\em Certified real-time
  solution of parametrized partial differential equations}, in Handbook of
  Materials Modeling, S.~Yip, ed., Springer Netherlands, 2005, pp.~1529--1564.

\bibitem{noack2005need}
{\sc B.~R. Noack, P.~Papas, and P.~A. Monkewitz}, {\em The need for a
  pressure-term representation in empirical {G}alerkin models of incompressible
  shear flows}, Journal of Fluid Mechanics, 523 (2005), pp.~339--365.

\bibitem{prud2002reliable}
{\sc C.~Prud'homme, D.~V. Rovas, K.~Veroy, L.~Machiels, Y.~Maday, A.~T. Patera,
  and G.~Turinici}, {\em {Reliable real-time solution of parameterized partial
  differential equations: Reduced-basis output bound methods}}, Journal of
  Fluids Engineering, 124 (2002), pp.~70--80.

\bibitem{rathinam:newlook}
{\sc M.~Rathinam and L.~R. Petzold}, {\em A new look at proper orthogonal
  decomposition}, SIAM Journal on Numerical Analysis, 41 (2003),
  pp.~1893--1925.

\bibitem{rempfer2000low}
{\sc D.~Rempfer}, {\em On low-dimensional {G}alerkin models for fluid flow},
  Theoretical and Computational Fluid Dynamics, 14 (2000), pp.~75--88.

\bibitem{rowley2004mrc}
{\sc C.~W. Rowley, T.~Colonius, and R.~M. Murray}, {\em Model reduction for
  compressible flows using {POD} and {Galerkin} projection}, Physica D:
  Nonlinear Phenomena, 189 (2004), pp.~115--129.

\bibitem{rozza2007reduced}
{\sc G.~Rozza, D.~B.~P. Huynh, and A.~T. Patera}, {\em {Reduced basis
  approximation and a posteriori error estimation for affinely parametrized
  elliptic coercive partial differential equations}}, Archives of Computational
  Methods in Engineering, 15 (2008), pp.~229--275.

\bibitem{ryckelynck2005phm}
{\sc D.~Ryckelynck}, {\em A priori hyperreduction method: an adaptive
  approach}, Journal of Computational Physics, 202 (2005), pp.~346--366.

\bibitem{san2013proper}
{\sc O.~San and T.~Iliescu}, {\em Proper orthogonal decomposition closure
  models for fluid flows: Burgers equation}, arXiv preprint arXiv:1308.3276,
  (2013).

\bibitem{shen1990long}
{\sc J.~Shen}, {\em Long time stability and convergence for fully discrete
  nonlinear {G}alerkin methods}, Applicable Analysis, 38 (1990), pp.~201--229.

\bibitem{sirisup2004spectral}
{\sc S.~Sirisup and G.~Karniadakis}, {\em A spectral viscosity method for
  correcting the long-term behavior of {POD} models}, Journal of Computational
  Physics, 194 (2004), pp.~92--116.

\bibitem{sirovich1987tad3}
{\sc L.~Sirovich}, {\em Turbulence and the dynamics of coherent structures.
  {III:} dynamics and scaling}, Quarterly of Applied Mathematics, 45 (1987),
  pp.~583--590.

\bibitem{spalartAllmaras}
{\sc P.~Spalart and S.~Allmaras}, {\em A one-equation turbulence model for
  aerodynamic flows}, La Recherche A{\'e}rospatiale, 1 (1994), p.~5.

\bibitem{veroy2005certified}
{\sc K.~Veroy and A.~T. Patera}, {\em Certified real-time solution of the
  parametrized steady incompressible {N}avier--{S}tokes equations: Rigorous
  reduced-basis a posteriori error bounds}, International Journal for Numerical
  Methods in Fluids, 47 (2005), pp.~773--788.

\bibitem{veroy2003peb}
{\sc K.~Veroy, C.~Prud'homme, D.~V. Rovas, and A.~T. Patera}, {\em \textit{A
  posteriori} error bounds for reduced-basis approximation of parametrized
  noncoercive and nonlinear elliptic partial differential equations}, AIAA
  Paper 2003-3847, 16th AIAA Computational Fluid Dynamics Conference, Orlando,
  FL,  (2003).

\bibitem{wang2012proper}
{\sc Z.~Wang, I.~Akhtar, J.~Borggaard, and T.~Iliescu}, {\em Proper orthogonal
  decomposition closure models for turbulent flows: a numerical comparison},
  Computer Methods in Applied Mechanics and Engineering, 237 (2012),
  pp.~10--26.

\end{thebibliography}
\bibliographystyle{siam}
\end{document}